\documentclass[12pt]{amsart}
\usepackage{graphicx}
\usepackage{color}
\usepackage[colorlinks=true]{hyperref}
\usepackage{mathtools,cite}
\usepackage{stmaryrd}
\usepackage{xparse}
\definecolor{RedOrange}{cmyk}{ 0, 0.77, 0.87, 0}
\definecolor{RoyalPurple}{cmyk}{ 0.84, 0.53, 0, 0}
\definecolor{YellowGreen}{cmyk}{ 0.44, 0, 0.74, 0}
\definecolor{Fuchsia}{cmyk}{ 0.47, 0.91, 0, 0.08}
\definecolor{Blue}{cmyk}{ 0.84, 0.53, 0, 0}
\definecolor{BlueViolet}{cmyk}{ 0.84, 0.53, 0, 0}
\definecolor{Black}{cmyk}{ 0.75, 0.68, 0.67, 0.9}
\usepackage{xcolor}
\usepackage{verbatim}
\usepackage{amsmath}
\usepackage{amssymb, mathrsfs}
\usepackage{amsbsy}
\usepackage{amscd}
\usepackage{amsthm}
\usepackage{bbold}
\usepackage{ulem}
\usepackage[english]{babel}
\usepackage{todonotes}
\usepackage[T1]{fontenc}
\DeclarePairedDelimiterX{\Iintv}[1]{\llbracket}{\rrbracket}{\iintvargs{#1}}
\NewDocumentCommand{\iintvargs}{>{\SplitArgument{1}{,}}m}
{\iintvargsaux#1}
\NewDocumentCommand{\iintvargsaux}{mm} {#1\mkern1.5mu,\mkern1.5mu#2}
\usepackage{color}
\newcommand{\eb}{\mathbf{e}}

\newcommand{\lf}{\lfloor}
\newcommand{\rf}{\rfloor}
\newcommand{\R}{\mathbb{R}}
\newcommand{\bL}{\mathbb{L}}

\newcommand{\N}{\mathbb{N}}
\newcommand{\bS}{\mathbb{S}}
\newcommand{\e}{\varepsilon}
\newcommand{\E}{\mathbb{E}}
\newcommand{\Z}{\mathbb{Z}}

\newcommand{\pp}{\mathbb{P}}

\newcommand{\eps}{\epsilon}
\newcommand{\kA}{\mathcal{A}}

\newcommand{\kC}{\mathcal{C}}

\newcommand{\kR}{\mathcal{R}}
\newcommand{\kO}{\mathcal{O}}
\newcommand{\kP}{\mathcal{P}}

\newcommand{\kF}{\mathcal{F}}

\newcommand{\kE}{\mathcal{E}}

\newcommand{\kL}{\mathcal{L}}

\newcommand{{\rmT}}{\mathrm{T}}
\newcommand{\rmB}{\mathrm{B}}

\newcommand{\Con}{\mathrm{Con}}
\newcommand{\sct}{\mathscr{T}}
\newcommand{\hsct}{\mathscr{T}^*}
\newcommand{\scl}{\mathscr{L}}
\newcommand{\tT}{\tilde{\mathrm{T}}}
\newcommand{\fgd}{\mathfrak{g}_d}
\newcommand{\fsd}{\mathfrak{s}_d}
\newcommand{\frd}{\mathfrak{r}_d}

\newcommand{\1}[1]{\mathbf{1}_{#1}}

\newcommand{\lin}{\left[\kern-0.15em\left[}
\newcommand{\rin} {\right]\kern-0.15em\right]}
\newcommand{\linf}{[\kern-0.15em [}
\newcommand{\rinf} {]\kern-0.15em ]}
\newcommand{\ilin}{\left]\kern-0.15em\left]}
\newcommand{\irin} {\right[\kern-0.15em\right[}
\def\be#1{\begin{equation*}#1\end{equation*}}
\def\ben#1{\begin{equation}#1\end{equation}}
\def\bea#1{\begin{eqnarray*}#1\end{eqnarray*}}
\def\bean#1{\begin{eqnarray}#1\end{eqnarray}}
\def\al#1{\begin{align*}#1\end{align*}}
\def\aln#1{\begin{align}#1\end{align}}
\usepackage{constants}
\newconstantfamily{c}{symbol=c}
\newconstantfamily{a}{symbol=\alpha}
\newcommand{\secno}[1]{\thesection.\arabic{#1}}
\newconstantfamily{kE}{
symbol=\mathcal{E},
format=\secno,
reset={section}
}

\renewcommand{\tilde}{\widetilde}

\DeclareMathOperator{\T}{T}
\newtheorem{lem}{Lemma}[section]

\newtheorem{prop}[lem]{Proposition}
\newtheorem{theo}[lem]{Theorem}

\newtheorem {rem}[lem] {Remark}

\usepackage{color}
\definecolor{lilas}{RGB}{182, 102, 210}

\numberwithin{equation}{section}
\def\be#1{\begin{equation*}#1\end{equation*}}
\def\ben#1{\begin{equation}#1\end{equation}}
\def\bea#1{\begin{eqnarray*}#1\end{eqnarray*}}
\def\bean#1{\begin{eqnarray}#1\end{eqnarray}}

\begin{document}
\title[Upper tail large deviations  in the frog model]{Upper tail large deviation for the first passage time in the frog model on $Z^d$ with $d\geq 2$}

\author[V.~H.~CAN]{Van Hao CAN}
\address[V.~H.~CAN]{Institute of Mathematics, Vietnam Academy of Science and Technology, 18 Hoang Quoc Viet, Nghia Do, Hanoi, Vietnam.}
\email{cvhao@math.ac.vn}

\author[N.~KUBOTA]{Naoki KUBOTA}
\address[N.~KUBOTA]{College of Science and Technology, Nihon University, Chiba 274-8501, Japan.}
\email{kubota.naoki08@nihon-u.ac.jp}

\author[S.~NAKAJIMA]{Shuta NAKAJIMA}
\address[S.~NAKAJIMA]{Department of Mathematics, Keio University, Kanagawa 223-8522, Japan.}
\email{njima@keio.jp}

\begin{abstract}
We study upper  tail large deviations for the first passage time in the
frog model on $\mathbb Z^d$, $d\ge 2$. Let $\rmT(0,x)$ denote the first
time at which the infection, started from the origin, reaches $x$, and
let $\mu$ be the associated time constant. We prove that, for any Euclidean unit vector 
$v\in\mathbb R^d$ and $\xi>0$,
\[
\lim_{n\to\infty}
\frac{-1}{\mathfrak s_d(n)}
\log \mathbb P\big(\rmT(0,\lfloor nv\rfloor)\ge(\mu(v)+\xi)n\big)
=\mathfrak r_d\xi,
\]
where $\mathfrak s_2(n)=n/\log n$ and  $\mathfrak r_2=\pi$ for $d=2$;  $\mathfrak s_d(n)=n$ and $\mathfrak r_d$ is the escape probability of
simple random walk for $d\ge3$. 
Moreover, on the  upper tail large deviation event, we show that a certain slowdown mechanism occurs near the endpoint. 
To rule out other mechanisms, we analyze the growth of the infection near the starting point using a modified exploration process, and investigate its  propagation through the intermediate region  using the renormalization
and covering arguments.
\end{abstract}

\maketitle

\section{Introduction}
Front propagation is a fundamental phenomenon arising in a wide variety of physical, biological, and social systems. A prototypical example is an autocatalytic reaction described by 
\[
X+Y\rightarrow 2X,
\]
which appears in the spread of information and infectious diseases as well as in combustion processes. In this reaction, particles of type $X$ convert particles of type $Y$ into the same type upon contact. Depending on the application, \(X\)-particles may represent heat sources, informed individuals, or infected individuals, whereas \(Y\)-particles represent unburned material, uninformed individuals, or susceptible individuals. Understanding the propagation induced by this reaction has attracted considerable attention over the past decades. On the deterministic side, reaction-diffusion equations provide a classical framework for describing propagating fronts, with the Fisher--Kolmogorov--Petrovsky--Piscounov equation as a canonical example; see \cite{Pan04,vSaa03}  and the references therein. On the stochastic side,  according to \cite{KesSid05}, the study of particle systems modeling the reaction \(X+Y\to2X\) goes back to work of Spitzer and Kesten around 1980. When \(X\)- and \(Y\)-particles are interpreted as infected and susceptible individuals, respectively, such systems are often referred to as susceptible-infected (SI) models.

The frog model is one of the basic stochastic particle systems arising from this picture. In this model, susceptible particles remain immobile until they are activated, after which they perform independent random walks. Alves, Machado, and Popov \cite{AlvMacPop02}  established a shape theorem for the frog model, showing that the activated region grows linearly and converges, after rescaling, to a deterministic shape. A related SI model was studied by Kesten and Sidoravicius \cite{KesSid05,KesSid08}, in which both susceptible and infected particles perform random walks and infection is transmitted upon contact. They established a shape theorem when the two types of particles move at the same rate. More recently, Dauvergne and Sly \cite{DauSly23} proved linear growth when the two types move at different rates. 

In this paper, we focus on the frog model with one frog initially placed at each site of $\mathbb Z^d$ and study the rare-event behavior of its first passage time. Let \(\rmT(x,y)\) denote the first time at which the frog initially located at \(y\) is activated when the process starts with only the frog at \(x\) active. As we recall below, \(\rmT(0,nv)\) grows linearly in \(n\), with a deterministic rate depending on the direction \(v\). Our aim is to estimate the probability that \(\rmT(0,nv)\) is substantially larger than this typical value.

In one dimension, the first passage time possesses a renewal structure, which allows for a precise analysis of its asymptotic behavior. In particular, this structure yields a central limit theorem for the first passage time \cite{ComQuaRam07,ComQuaRam09} and also provides an effective tool for studying rare slowdown events. In dimensions \(d\geq 2\), however, no analogous renewal structure is available. The propagation is substantially more complicated, as activation can spread along many competing routes. Even the typical fluctuations of the first passage time remain poorly understood; for instance, its variance is only known to be sublinear in \(n\) \cite{CanNak19}, in contrast to the one-dimensional case, where the renewal structure implies that \((\mathrm{T}(0,n) - \mu n)/\sqrt{n}\) converges in distribution to a centered Gaussian random variable

The study of rare slowdown events was initiated by Bérard and Ramírez \cite{BR} for the one dimensional continuous time frog model. They proved that for any $\xi>0$,
\[
\exp(-t^{1/2+o(1)})
\le
\mathbb P\left(\rmT(0,\lfloor t\rfloor)\ge \E(\rmT(0,\lfloor t\rfloor))+ \xi t\right)
\le
\exp(-t^{1/3+o(1)}),
\]
showing that the probability has stretched-exponential decay, while leaving the exact asymptotics open. 
In subsequent work, we studied the corresponding discrete-time frog model and
determined the exact asymptotics of  upper tail large deviation in one dimension
\cite{CanKubNak26}; see Subsection~\ref{rel1d}  for the precise statement
and the underlying slowdown mechanism. 

These results naturally raise the corresponding question in
higher dimensions. The problem is not merely a higher-dimensional extension of the one-dimensional one: the absence of a renewal structure and the fact that local obstructions can be avoided in higher dimensions lead to a qualitatively different slowdown mechanism. In this paper, we determine the exact large deviation asymptotics in every dimension \(d\geq2\), revealing both a dimensional transition in the decay speed and a localization mechanism that is absent in one dimension.

\subsection{Model and main result}
\label{subsect:model}

 We consider the discrete time frog model on the integer lattice $\mathbb{Z}^d$ with $d\ge2$. Initially, one active frog is located at the origin, while every other site contains exactly one sleeping frog. Active frogs perform independent discrete time simple symmetric random walks. Whenever an active frog visits a site containing sleeping frogs, all frogs at that site become active. Once activated, each frog continues to evolve independently according to the simple random walk.

We next make this description precise. Let $(S_k^x)_{k=0}^\infty$, $x \in \Z^d$, be independent simple random walks on $\Z^d$ with $S_0^x=x$ for each $x \in \Z^d$. This describes the trajectory of the frog initially located at $x$ after activation. We denote by $\pp$  the underlying probability measure of the independent simple random walks. For later convenience, we extend the first-passage time to $x,y\in\R^d$ by identifying them with their coordinatewise integer parts.
For $x,y \in \R^d$, we define 
\begin{align*}
	\T(x,y) \coloneq \inf\left\{ \sum_{i=0}^{m-1} t(x_i,x_{i+1}):
    \begin{minipage}{5cm}
        $m \geq 1$,\, $x_0=\lfloor x\rfloor,\,x_m=\lfloor y\rfloor$\\
        and $x_1,\dots,x_{m-1} \in \Z^d$
    \end{minipage}
    \right\},
\end{align*}
where $t(x,y) \coloneq \inf\{k \geq 0:S_k^{\lfloor x\rfloor}=\lfloor y\rfloor \}$ denotes the hitting time of $\lfloor y\rfloor$ by the random walk starting from $\lfloor x \rfloor$. Here $\lfloor z \rfloor$ is the vertex of $\Z^d$ obtained by taking the floor function to each coordinate of $z$.  Thus, $\T(x,y)$ represents the time at which the frog initially located at $\lfloor y\rfloor$ is activated when only the frog at $\lfloor x\rfloor$ is initially active.

The first passage times satisfy subadditivity. In particular, Alves, Machado, and Popov \cite{AlvMacPop02} proved a shape theorem: there exists a deterministic norm $\mu$ on $\mathbb{R}^d$ such that
\[
\lim_{n\to\infty}\frac{\T(0,nv)}{n}=\mu(v),
\]
almost surely and in $L^1$ for every $v\in\mathbb{R}^d \setminus \{0\}$.\\

We study the upper tail probability that the first passage time exceeds its typical value. Our main result identifies both the exact decay scale and the corresponding rate function.


\begin{theo}\label{theo:main} Let $d\geq 2$. 
For any Euclidean unit vector  $v \in \R^d$ and $\xi>0$,
\begin{align*}
	\lim_{n \to \infty} \frac{-1}{\mathfrak{s}_d(n)} \log \pp\bigl( {\rm T}(0,nv) \geq (\mu(v)+\xi)n \bigr)=\mathfrak{r}_d \xi,
\end{align*}
where \begin{align*}
\mathfrak{s}_d(n) \coloneq
\begin{dcases}
\frac{n}{\log n}, & \text{if } d=2,\\
	n, & \text{if } d \geq 3,	
\end{dcases}
\qquad
\mathfrak{r}_d \coloneq
\begin{dcases}
	\pi, & \text{if } d=2,\\
	\mathfrak{g}_d, & \text{if } d \geq 3,
\end{dcases}
\end{align*}
and $\mathfrak{g}_d \coloneq \pp(S_k^0 \neq 0 \text{ for all } k \geq 1)$ is the escape probability from the origin $0$ of the simple random walk $S^0$.
\end{theo}
Thus, the speed of the upper tail large deviation exhibits a dimensional transition, from \(n/\log n\) in dimension \(2\) to \(n\) in dimensions \(d\geq 3\), while the rate function is linear in \(\xi\) in every dimension.

\subsection{Discussion and related results}

\subsubsection{Relation with the one-dimensional case} \label{rel1d}

As mentioned above, the upper tail large deviation for the one-dimensional frog model was established in our previous work \cite{CanKubNak26}. More precisely, we proved that there exists a constant $\mathfrak r_*\in(0,\infty)$ such that, for every $\xi>0$,
\begin{align}\label{eq:uld1}
\lim_{n\to\infty}\frac{-1}{\sqrt n}
\log\pp\bigl({\rm T}(0,n)\geq(\mu(1)+\xi)n\bigr)
=\mathfrak r_*\sqrt{\xi}.
\end{align}
The constant $\mathfrak r_*$ is characterized through a variational problem describing the optimal strategy for slowing down the infection.

A key feature of the slowdown phenomenon in one dimension is the following localization mechanism. A delay of order $n$ can be produced by slowing down the infection inside a bad interval of size of order $\sqrt n$, and this interval can be located essentially anywhere between the origin and the target. This mechanism is specific to one dimension, since every propagation route from $0$ to $n$ must cross the bad interval.

The slowdown mechanism  is fundamentally different in dimensions $d\geq2$. A localized obstruction in the intermediate region can potentially be avoided by the infection, and it is therefore not clear a priori where an optimal slowdown should occur. We compare three possible regimes: a slowdown near the starting point, during the intermediate propagation, or near the target. Our analysis shows that slowdowns in all three regimes occur on the same large deviation scale $\fsd(n)$, making it necessary to compare their rate functions. This comparison reveals that the least costly mechanism is the last one: the infection propagates essentially at its typical speed to a neighborhood of the target, while the frogs there avoid hitting the target for an additional time of order $\xi n$. Heuristically, the probability of this dominant mechanism can be expressed as
\begin{align*}
\pp\bigl({\rm T}(0,nv)\geq(\mu(v)+\xi)n\bigr)
&\approx
\pp\Bigl(t(x,nv)\geq\xi n
\text{ for all }x\in\Lambda(nv,\mathcal O(\sqrt n))\Bigr)\\
&\approx
\exp\left(-\E[|\mathcal R_{\xi n}|]\right)
\approx
\exp\left(-\mathfrak r_d\xi\mathfrak s_d(n)\right),
\end{align*}
where $\mathcal R_m$ denotes the range of a simple random walk up to time $m$. Thus, in contrast to the one dimensional case, the  rate function for higher dimensions is determined by the cost of avoiding the target in the final stage. The main challenge is therefore to show that neither of the other slowdown mechanisms can achieve a smaller cost.

\subsubsection{Main ingredients of the proof}
 We analyze separately the initial growth near the
origin, the propagation through the intermediate region, and the final hitting of the target. The final stage is essentially governed by the avoidance probability described above, whereas the first two stages require new ingredients.

For the initial stage, we first show that the infection quickly produces an
sufficiently large fixed number of active frogs and then spreads rapidly away from the origin. The main  ingredient is an exploration procedure that reveals the infection generation by generation and exploits approximate independence when exploring previously unvisited regions. This allows us to show that producing a macroscopic delay near the starting point is more costly than delaying the final hit of the target.

The intermediate stage requires a multiscale renormalization argument based on the propagation of the infection through a family of spatial slices. This is reminiscent of renormalization arguments in first-passage percolation introduced by Kesten \cite{Kes86}, but there is an important additional difficulty. In first-passage percolation, the growth is determined by a static random environment, whereas in the frog model it is driven by a dynamically evolving system of interacting random walks. Long excursions of individual walks can therefore create substantial
correlations between spatially separated regions. To overcome this difficulty, we combine the  renormalization argument with a covering procedure. This procedure groups the bad regions into a well-separated collection, which restores the approximate independence needed for the subsequent estimates. These estimates show that a slowdown during the intermediate stage is less likely than a slowdown near the target.

Together, the estimates for the three stages show that the final stage, a slowdown near the target, is the least costly of the three competing mechanisms. The organization of the proof and the key estimates for each stage are given in Subsection~\ref{subsec:orgproof}.

\subsubsection{On the lower large deviation problem}

A natural counterpart to the  problem studied in this paper is the lower tail large deviation 
$$\pp({\rm T}(0,nv) \leq (\mu(v)-\xi)n), \qquad 0<\xi<\mu(v),$$
which describes an atypically fast propagation of the infection. Its
mechanism is fundamentally different from the localized slowdown responsible
for the upper tail. Indeed, to reduce the passage time by an amount of order
$n$, the infection must propagate atypically fast over a macroscopic distance. The speedup must therefore be sustained along the route from $0$ to $nv$, rather than localized in a small region. Consequently, the lower tail probability has the linear large deviation speed $n$. This problem will be addressed in a forthcoming paper.

\subsubsection{Related results}
The one-dimensional results of B\'erard and  Ram\'{\i}rez~\cite{BerRam10}
and our previous work~\cite{CanKubNak26} were discussed above. We therefore focus here on similar problems  of related models in higher dimensions.

In dimensions $d\geq2$, Kubota~\cite[Theorem~1]{Kub19} proved that the upper tail large deviation probability decays at least stretched-exponentially: for some $c>0$ and $\alpha\in(0,1)$, it is bounded above by $\exp(-cn^\alpha)$ for all sufficiently large $n$. Theorem~\ref{theo:main} sharpens this result by identifying both the optimal decay speed and the corresponding rate function.

Upper tail large deviations  have also been studied in
first-passage percolation and related models. For first-passage percolation, Basu, Ganguly, and Sly~\cite{BasGanSly21} and Cosco and Nakajima~\cite{CosNak23} obtained upper tail large deviation estimates whose decay speed depends on the tail of the edge weight. For the chemical distance in Bernoulli bond percolation, Dembin and Nakajima~\cite{DemNak22_arXiv} established an upper tail large deviation principle with explicit speed and rate function. 
For chemical distances associated with level sets of the Gaussian free field, Peretz~\cite[Theorem~1.2]{Per25_arXiv} obtained a related result, with decay speed $n/\log n$ in dimension $d=3$ and $n$ in dimensions $d \geq 4$. These results illustrate how the geometry and dependence structure of the underlying model can strongly affect the large deviations of chemical distances.

Beyond questions concerning passage times and large deviations, the frog model has been studied extensively from many other perspectives. These include recurrence and transience, the effect of the initial particle configuration and the underlying graph, and variants in which the motion or activation rules are modified; see, e.g. \cite{TelWor99,JohRol19,HofJohJun17,MulWie20,MicRos20,KosZer17,GuoTanWei22} and the references therein.

\subsection{Organization of the paper and notation}
\label{sect:notation}

The remainder of the paper is organized as follows.
In Section~\ref{sect:prelim}, we collect preliminary estimates and present a detailed organization of the proof of Theorem~\ref{theo:main}, together with the key propositions used in the argument.
Section \ref{sec:lbae} is devoted to the lower bound via the avoidance event near the target.
The upper bound is established separately in dimensions $d\geq 3$ and $d=2$ in Sections \ref{sec:theomain1}  and \ref{sec:theomain2}, respectively.
The main ingredients needed for the upper bound are proved in the subsequent sections: Section~\ref{sec:proptln} analyzes the growth of the infection in the initial stage, while Section~\ref{sect:LB_uld3} develops the renormalization and covering arguments for the propagation through the intermediate region.\\

\noindent {\bf Notation}. We collect here notation that will be used repeatedly
throughout the paper. For $a\in \N$, set $[a]:=[1,a]\cap \Z$. For any finite set $A$, write $|A|$ for its cardinality, and recall that
\begin{align*}
	\mathfrak{g}_d:=\pp(S_n^0 \neq 0 \text{ for all } n \geq 1).
\end{align*}

Given $x,y \in \Z^d$ and a set $A\subset \Z^d$ containing $x$, define
\begin{align*}
	\T_A(x,y):=\inf \left\{\sum_{i=1}^{\ell-1} t(x_i,x_{i+1}): \ell \geq 1,\, x=x_1,\, x_2,\dots,x_{\ell-1}\in A,\, x_\ell=y \right\}.
\end{align*}
If $A\cap B=\emptyset$, then for any $x_1 \in A$, $x_2 \in B$, and $y_1,y_2 \in \Z^d$, the random variables $\T_A(x_1,y_1)$ and $\T_B(x_2,y_2)$ are independent. For subsets $A,B\subset \Z^d$, we also write
\begin{align*}
	{\rm T}(A,B):=\inf_{x\in A,\, y\in B}{\rm T}(x,y),\qquad
	\tilde{\rm T}(A):=\sup_{x,y \in A}{\rm T}(x,y).
\end{align*}
For $A\subset \Z^d$ and $n \geq 0$, let $\rmB_n^A$ be the set of sites that can be reached from $A$ within time $n$, namely
\begin{align*}
	\rmB_n^A:=\{y \in \Z^d:{\rm T}(A,\{y\}) \leq n\}.
\end{align*}
When $A=\{x\}$, we abbreviate $\rmB_n^A$ to $\rmB_n^x$; in particular, $\rmB_n:=\rmB_n^0$. For any $x,y \in \Z^d$, let $\kO_{x,y}$ denote the set of optimal paths from $x$ to $y$ for ${\rm T}(x,y)$, that is, the set of all sequences $\gamma=(x_i)_{i=0}^\ell$ in $\Z^d$ such that $x_0=x$, $x_\ell=y$, and
\begin{align*}
	{\rm T}(x,y)=\sum_{i=0}^{\ell-1} t(x_i,x_{i+1}).
\end{align*}

Finally, write $\mathbf{e}_1,\dots,\mathbf{e}_d$ for the canonical unit vectors of $\R^d$, and write $|\cdot|_1$, $|\cdot|$, and $|\cdot|_\infty$ for the $\ell^1$-, $\ell^2$-, $\ell^\infty$-norms on $\R^d$, respectively. Thus, for $x=(x_1,\dots,x_d) \in \R^d$,
\begin{align*}
	|x|_1:=|x_1|+\dots+|x_d|,\qquad
	|x|:=\sqrt{x_1^2+\dots+x_d^2},\qquad
	|x|_\infty:=\max_{1 \leq i \leq d}|x_i|.
\end{align*}
For $A\subset \Z^d$ and $r>0$, define
\begin{align*}
	\Lambda(A,r):=\left\{x\in \Z^d:\inf_{y \in A}|x-y|_\infty \leq r \right\}.
\end{align*}
When $A=\{a\}$, we write $\Lambda(a,r)$, and we further set $\Lambda(r):=\Lambda(0,r)=[-r,r]^d \cap \Z^d$. Finally, for $A_0\subset \Z^d$ and $k\in \N$, define
\begin{align*}
	\Con_{\geq k}(A_0):=\left\{A \subset \Z^d: A \text{ is connected},\, A_0 \subset A,\, |A| \geq k \right\}.
\end{align*}

\section{Preliminaries}\label{sect:prelim}

\subsection{Organization of proof and key propositions} \label{subsec:orgproof}
For simplicity, we present a sketch of the proof for the case $d\geq 3$ and $v=\eb_1$, and  write $\mu=\mu(\mathbf{e}_1)$. Our main goal is to prove that 
\be{
\pp({\rm T}(0, n\eb_1) \geq (\mu +\xi) n) = \exp \left( -(\fgd \xi +o(1))n \right).
}

\subsubsection{Lower bound via the avoidance probability}
For the lower bound, we exhibit a slowdown mechanism whose
cost matches the rate in Theorem~\ref{theo:main}. The idea is to let the
infection propagate essentially normally toward the target, but require the
frogs near the target to avoid the vertex $ne_1$ for an additional time of
order $\xi n$.
Specifically, we prove in Section~\ref{sec:avprob} that
\[
\pp\bigl(\T(0,ne_1)\geq(\mu+\xi)n\bigr)
\geq
\frac12
\pp\Bigl(
t(z,ne_1)\geq\xi n
\ \text{for all }z\in
\Lambda(ne_1,n^{2/3})\setminus\{ne_1\}
\Bigr).
\]
This avoidance probability is analyzed in Section~\ref{sec:avprob}.
By translation invariance and estimates on the range of a
simple random walk, its leading order cost is
\[
\begin{aligned}
\pp\Bigl(
t(z,0)\geq\xi n
\ \text{for all }z\in\Lambda(n^{2/3})\setminus\{0\}
\Bigr)
&\geq
\pp\bigl(t(z,0)\geq\xi n\ \text{for all }z\neq0\bigr)\\
&\geq
\exp\bigl(-(1+o(1))\E|\mathcal R_{\xi n}|\bigr)\\
&\geq
\exp\bigl(-(\mathfrak g_d\xi+o(1))n\bigr).
\end{aligned}
\]
Here $\mathcal R_{\xi n}$ is the range of a simple random walk up to time
$\xi n$, and the last inequality follows from
\[
\E|\mathcal R_{\xi n}|
=(\mathfrak g_d\xi+o(1))n.
\]
This yields the required lower bound.

\subsubsection{Upper bound by analyzing three stages of infection}
\label{subsubsect:upper_infect}

\begin{figure}
    \centering
    \includegraphics[width=0.8\linewidth]{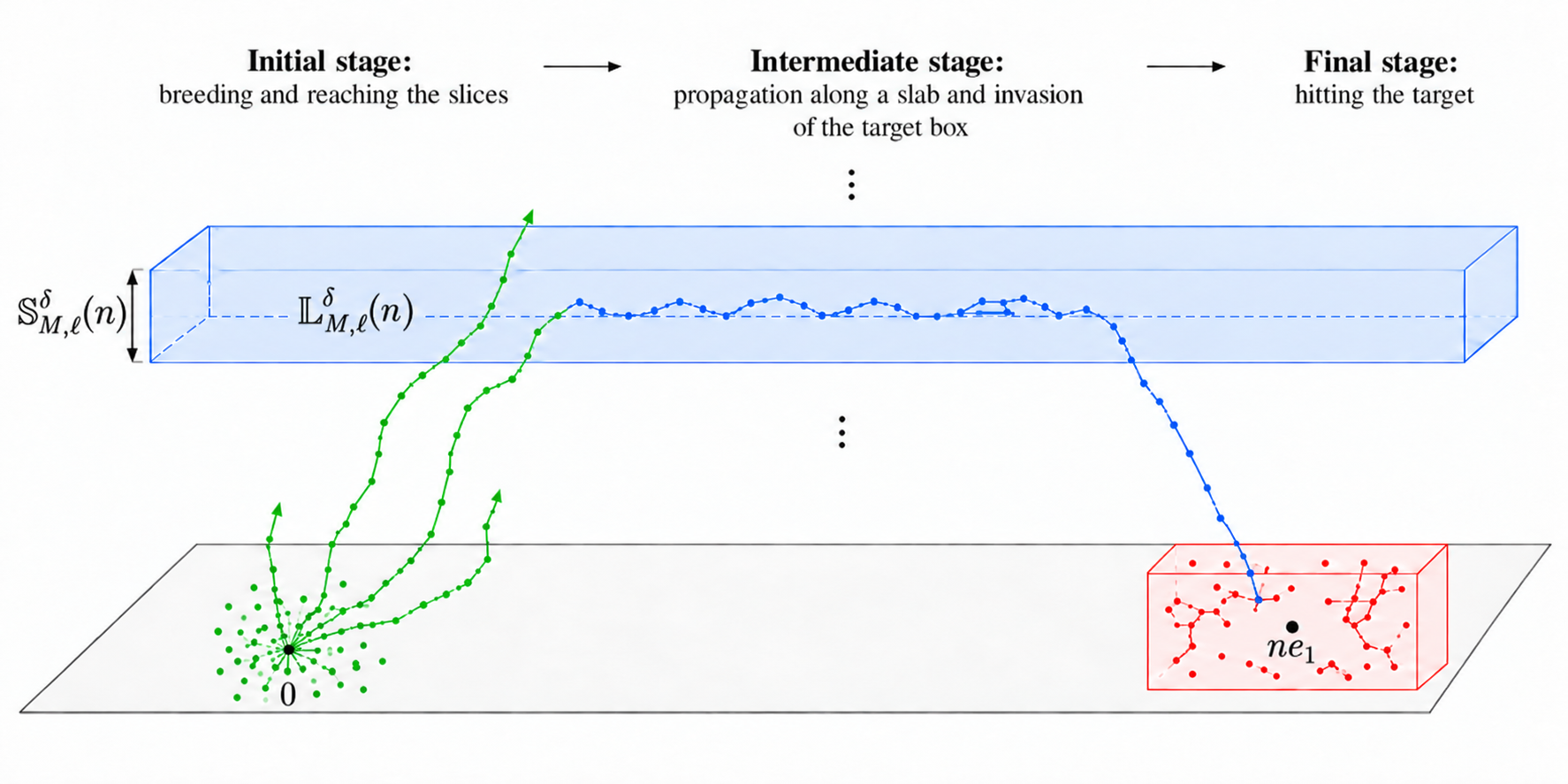}
    \caption{Three stages of infection}
    \label{fig:placeholder}
\end{figure}

For the upper bound, we separate the propagation of the infection into
three stages:
\[
  \text{reaching the slices}
  \quad\longrightarrow\quad
  \text{invading the target box}
  \quad\longrightarrow\quad
  \text{hitting the target}.
\]
Fix large integers $M,L$ and a small constant $\delta>0$, and introduce
the slices
\be{
\bL_{M,\ell}^\delta(n)
:=
[-\delta n,(1+\delta)n]
\times\{2\ell\lfloor M\sqrt n\rfloor\}
\times[-\delta n,\delta n]^{d-2}
\cap\Z^d,
\qquad 1\le |\ell|\le L.
}
Associated with each slice $\bL_{M,\ell}^\delta(n)$ is a slab
$\bS_{M,\ell}^\delta(n)$, obtained essentially by thickening the slice
by $M\sqrt n$ in the $\mathbf{e}_2$-direction. These slabs will be used to
restrict the propagation in the intermediate stage.
We define also
\[
{\rm T}_1
:=
\max_{1\le|\ell|\le L}
{\rm T}(0,\bL_{M,\ell}^\delta(n)).
\]
Let $\Lambda^{\varepsilon,\delta}_{M,L}(n)$ denote the set of vertices
in $\Lambda(n\eb_1,M\sqrt n)\setminus\{n\eb_1\}$ that can be reached
from one of the auxiliary slices, uniformly over the starting point on
that slice, within time $(\mu+\varepsilon)n$. We define the invasion
event
\[
{\rm Invad}
:=
\left\{
|\Lambda^{\varepsilon,\delta}_{M,L}(n)|
\ge
(1-\delta)|\Lambda(n\eb_1,M\sqrt n)|
\right\},
\]
and the final hitting time
\be{
{\rm T}_2
:=
\min\big\{
t(x,n\eb_1):
x\in\Lambda^{\varepsilon,\delta}_{M,L}(n)
\big\}.
}
Thus, ${\rm T}_1$ is the time needed for the infection to reach all
the auxiliary slices. On ${\rm Invad}$, one of these slices can
propagate the infection to almost all of the target box within an
additional $(\mu+\varepsilon)n$ steps, and ${\rm T}_2$ measures the
remaining time needed to hit the target. Consequently, on
${\rm Invad}$,
\[
{\rm T}(0,n\eb_1)
\le
{\rm T}_1+(\mu+\varepsilon)n+{\rm T}_2,
\]
and hence
\ben{\label{tnt12}
\pp\big({\rm T}(0,n\eb_1)\ge(\mu+\xi)n\big)
\le
\pp({\rm Invad}^c)
+
\pp\big({\rm T}_1+{\rm T}_2\ge(\xi-\varepsilon)n\big).
}
The estimates needed for this decomposition are
\ben{\label{ldt1}
\pp({\rm T}_1\ge k)
\le
\exp\big(-(2\mathfrak g_d-o_M(1))k\big),
}
\ben{\label{ldt2}
\pp({\rm T}_2\ge k)
\le
\exp\big(-(\mathfrak g_d-o_M(1))k\big),
}
and
\ben{\label{proinvad}
\pp({\rm Invad}^c)
\le
\exp(-Cn),
}
where $o_M(1)$ denotes a quantity tending to zero as $M\to\infty$,
and in \eqref{proinvad} the constant $C>0$ can be chosen arbitrarily
large by taking $L$ sufficiently large. 

Note further that $\rmT_1$ and $\rmT_2$ depend respectively on the frogs near the starting point $0$
and the ending point $n\eb_1$. Therefore, the two passage times are nearly independent,  which implies 
\begin{align*}
   \pp({\rm T}_1 + {\rm T}_2 \geq (\xi -\varepsilon )n) &\approx  \sup_{1 \leq k \leq (\xi-\varepsilon) n} \pp({\rm T}_1 \geq k) \pp({\rm T}_2 \geq (\xi-\varepsilon) n - k)\\
   & \approx \exp(-(\mathfrak{g}_d-2\varepsilon) \xi n),
\end{align*}
using \eqref{ldt1} and \eqref{ldt2}.  This together with
\eqref{tnt12} and \eqref{proinvad} yields the upper bound for large deviation probability.

The estimate \eqref{ldt2} follows from the avoidance probability
estimate and is proved in Subsection~\ref{sec:avprob}. The proofs of
\eqref{ldt1} and \eqref{proinvad} are more involved and are described
in the following two subsections.

\subsubsection{Initial stage: breeding and growth}

The control of the initial stage consists of two steps. We first show
that the infection quickly produces sufficiently many active frogs.
Once sufficiently many frogs have been activated, we show that the
infection rapidly reaches the auxiliary slices.

For the breeding step, define the confinement exponent
\[
\kappa:=
-\sup \big  \{ |A|
\limsup_{n\to\infty}
\frac1n\log\mathbb P(\mathcal R_n\subset A): \, \, 0 \in A \subset \mathbb{Z}^d, \text{ finite and connected} \big  \}.
\]
We prove in Lemma~\ref{kvth}  that
\[
\kappa>2g_d,
\]
and, for every fixed $K$,
\begin{equation}
\mathbb P(|\rmB_k|\le K)
\le
\exp\{-(\kappa-o_M(1))k\}.
\tag{3.6}
\end{equation}
Thus, preventing the infection from growing beyond a bounded number
of vertices has an exponential cost strictly larger than $2g_d$.

Once sufficiently many frogs have been activated, we need a uniform
estimate ensuring rapid spatial growth. The main ingredient is the
following proposition.

\begin{prop} \label{tln}
	The following claims hold for all sufficiently large $n\in \N$.
	\begin{itemize}
		\item [(i)] Assume that $d\geq 3$. For any $M, \e >0$ there exists $c=c(M,\e)>0$, such that
		\[ \pp( \T_{\Lambda(M \sqrt{n})} (0,x) \leq \e n \, ,\ \forall \, x \in \Lambda(\sqrt{n})  ) \geq 1- \exp(-cn).  \]
		\item[(ii)] Assume that $d\geq 3$. For any $M,C, \delta >0$, there exists $K=K(M,C,\delta)\in\N$ such that, for any connected set $A\subset \Z^d$ containing $0$ with $|A|\geq K$,
		\[  \pp(\, \mathrm{B}^A_{\delta n} \subset \Lambda(M\sqrt{n})) \leq  \exp(-Cn). \]
		\item[(iii)] Assume that $d=2$. For any $M, \e >0$ there exists $c=c(M,\e)>0$ such that,
		\[ \pp( \T_{\Lambda(M \sqrt{n})} (0,x) \leq \e n \, ,\ \forall \, x \in \Lambda(\sqrt{n})) \geq 1- \exp(-cn/\log n).  \]
		\item[(iv)] Assume that $d=2$. For any $M, \delta >0$ there exists $c=c(M,\delta)>0$, such that
		\[  \pp(\mathrm{B}_{\delta n} \subset \Lambda(M\sqrt{n})) \leq \exp(-cn). \]
	\end{itemize}
\end{prop}

The proof of Proposition~\ref{tln} is based on a modified exploration process
designed to recover independence in the growth of the infection. At
each generation, we activate only a selected subset of the available
frogs and follow their descendants only in designated fresh regions,
while all other frogs are kept inactive. The modified process is
dominated by the original infection, but has substantially more
independence between successive generations. Showing that this process
still grows sufficiently well therefore gives the required lower bound
on the growth of the actual infection.

Proposition~\ref{tln} is then used in Lemma~\ref{lem:t1} to show that, once
sufficiently many frogs have been activated, the infection reaches all
the auxiliary slices within a prescribed linear time with overwhelming
probability. Combined with the breeding estimate, this gives the
control of $\rmT_1$. The construction and
analysis of the exploration process are given in Section~\ref{sec:proptln}.

\subsubsection{Intermediate stage: invading the target box}

The intermediate stage requires uniform control of the infection as it
propagates from the auxiliary slices toward the target box. The main
ingredient is a restricted passage time estimate on the diffusive
scale.

\begin{prop} \label{tuld}
	There exists $M \in \N$ such that, for any $\e>0$ there exists $c=c(\e,M)>0$ such that the following claims hold for all $n\in \N$ large enough, all $x \in \Z^d$  with $|x|_1 \leq \sqrt{n}$ and all $t$ with $(1+\e)\mu(x) \leq t \leq n$.
	\begin{itemize}
		\item [(i)] If $d\geq 3$, then
		$$\pp(\T_{\Lambda(M\sqrt{n})}(0,x) \geq  t ) \leq \exp(-ct).$$
		\item[(ii)]  If $d=2$, then
		$$\pp(\T_{\Lambda(M\sqrt{n})}(0,x) \geq  t ) \leq \exp(-ct/\log t).$$
	\end{itemize}
\end{prop}
Proposition~\ref{tuld} is a local estimate on the diffusive scale. To control
the restricted passage time along a slab over a macroscopic distance,
we cover a route by boxes of diameter of order $\sqrt n$, all contained
in the slab. The slab restricted passage time is then bounded by the
sum of the corresponding box restricted passage times.
Proposition~\ref{tuld} gives a large deviation estimate for each term, while
the terms have only finite range dependence. This yields an
exponentially small failure probability for each slab.
Since the auxiliary slabs are disjoint, the corresponding events are
independent. By taking sufficiently many slices, the failure
probability can therefore be made smaller than $e^{-Cn}$ for any
prescribed $C>0$. Once one slice succeeds, a further growth estimate
shows that the infection reaches most of the target box within a short
additional time, yielding the invasion estimate \eqref{proinvad}.

The proof of Proposition~\ref{tuld} is based on a renormalization and covering
argument. The renormalization produces a path of boxes connecting the
two endpoints, consisting mostly of good boxes but possibly containing
some bad ones. Passage times through the good boxes are well
controlled, so the main difficulty is to control the bad boxes, which
may occur in nearby clusters. To this end, we associate with each bad
box a radius measuring the extent of the surrounding region in which
propagation may still be anomalously slow, and introduce a covering
process that progressively merges these regions. The resulting
disjoint bad regions can then be treated collectively, allowing us to
control their total contribution to the passage time. This yields
Proposition~\ref{tuld}; the details are given in Section~\ref{sect:LB_uld3}.

\subsection{Some useful large deviation estimates}
\begin{lem} \label{dobbi}
	Let $(\kF_i)_{i\geq 0}$ be an increasing filtration with the trivial $\sigma$-field $\kF_0$, and $\tau$ a stopping time with respect to $(\kF_i)_{i\geq 1}$.
	Fix $p\in(0,1)$, and assume that $(X_i)_{i\geq 1}$ is a sequence of Bernoulli random variables such that for all $ i \in\N$, $X_i$ is $\kF_i$-measurable and  $\E(X_i \mathbf{1}( \tau \geq i)\mid \kF_{i-1}) \leq p$ a.s.
	Then for all $n \geq k \geq 1$,
	\be{
		\pp	 \left(\sum_{i=1}^n X_i \geq  k, \, \tau \geq n \right) \leq \binom{n}{k} p^k.
	}
	As a consequence, there exists a constant $c=c(p)\in (0,1)$ such that for any $n\in \N$
	\ben{ \label{consx}
		\pp	 \left(\sum_{i=1}^n X_i \geq  (1-c) n, \, \tau \geq n \right) \leq \exp(-cn).
	}
\end{lem}
\begin{proof}
	We prove by induction on $n\ge 1$ that for all $k=1,\dots,n$,
	\[
	\pp\Bigl(\sum_{i=1}^n X_i\ge k,\; \tau\ge n\Bigr) \le \binom{n}{k} p^k.
	\]
	The claim is trivial when $n=1$. Assume the claim holds for $n-1$. For $n\ge 2$ and any $1\le k\le n$, we have
	\[
	\Bigl\{\sum_{i=1}^n X_i\ge k,\; \tau\ge n\Bigr\} \subset  A\cup B,
	\]
	where
	\[
	A:=\Bigl\{\sum_{i=1}^{n-1} X_i\ge k,\; \tau\ge n\Bigr\}, \quad B:=\Bigl\{\sum_{i=1}^{n-1} X_i\ge k-1,\; X_n=1,\; \tau\ge n\Bigr\}.
	\]
	Since $A\subset \Bigl\{\sum_{i=1}^{n-1} X_i\ge k,\; \tau\ge n-1\Bigr\}$, the induction hypothesis gives
	\[
	\pp(A) \le \binom{n-1}{k}p^k.
	\]
	For $B$, using the tower property and the assumption,
	\al{
	\pp(B)&=\E\Bigl[\mathbf{1}\Bigl\{\sum_{i=1}^{n-1} X_i\ge k-1,\; \tau\ge n-1\Bigr\}\E\Bigl(X_n\mathbf{1}
	\{\tau\ge n\}\mid \kF_{n-1}\Bigr)\Bigr] \\
	&\le p\,\pp\Bigl(\sum_{i=1}^{n-1} X_i\ge k-1,\; \tau\ge n-1\Bigr).
}
	Thus, by the induction hypothesis,
	\[
	\pp(B) \le p\binom{n-1}{k-1}p^{\,k-1} = \binom{n-1}{k-1}p^k.
	\]
	Combining the bounds for $A$ and $B$, we obtain
	\[
	\pp\Bigl(\sum_{i=1}^n X_i\ge k,\; \tau\ge n\Bigr) \le \binom{n-1}{k}p^k + \binom{n-1}{k-1}p^k = \binom{n}{k}p^k,
	\]
	where the last equality follows from the well-known binomial identity.

	By taking $k=\lfloor (1-c)n\rfloor$ with  $c\in (0,1)$ chosen appropriately (depending on $p$) and applying standard estimates for the binomial coefficient, we get  \eqref{consx}. 
\end{proof}
	\begin{lem}
		 \label{lem:ldlc}
		 \begin{itemize}
		 	\item [(i)] Let $m, \ell \in \N$ and let $(X_i)_{i=1}^\ell$ be random variables such that, for every $I\subset [\ell]$ satisfying $|i-j|\ge m$ for all distinct $i,j\in I$, the family $(X_i)_{i\in I}$ is independent. Then for all $c >0$ and $h>0$,
		 		\al{
		 		\pp \left( \sum_{i=1}^\ell X_i \geq h  \right) \leq  m \exp(-c h/m)  \left(\max_{1 \leq i \leq \ell} \E[\exp(c X_i)] \right)^{\lceil \ell/m \rceil}.
		 	}
	 	\item[(ii)] 	Let $c, a, b$ be positive real numbers and $m, \ell$ be  integers. Suppose that $(X_i)_{i=1}^\ell$ are non-negative random variables satisfying $X_i$ and $X_j$ are independent when $|i-j| \geq m$, and
	 	\be{
	 		\pp(X_i \geq  t) \leq \exp (-ct)  \quad \forall \, t \in [a,b], \qquad \forall \, i \in [\ell].
	 	}
	 	Then for any $h$
	 	\be{
	 		\pp \left( \sum_{i=1}^\ell X_i \geq \ell a +h  \right) \leq m\exp(-ch/2m) 2^{\lceil \ell/m \rceil} + \ell \exp(-cb).
	 	}
		 \end{itemize}
\end{lem}
\begin{proof} We start by proving (i). For $t  \in [m]$, let
	\be{
		I_t =\{i: 1 \leq i \leq \ell, i \equiv t
		\, ({\rm mod } \, m) \}.
	}
	We then have $ |I_t| \leq \lceil \ell/m \rceil$, and $(X_i)_{i \in I_t}$ are independent random variables for each $t \in [ m]$. Therefore, using  the exponential Markov inequality, we get
	\al{
		\pp \left( \sum_{i=1}^\ell X_i \geq h  \right) \leq \sum_{t=1}^m \pp \left( \sum_{i \in I_t} X_i \geq h/m  \right) &\leq \sum_{t=1}^m \exp(-ch/m) \prod_{i \in I_t}\E[\exp(cX_i)] \\
		&\leq m \exp(-ch/m) \left(\max_{1 \leq i \leq \ell} \E[\exp(c X_i)] \right)^{\lceil \ell/m \rceil}.
	}
	For (ii), noting  that $X \leq a + X \1{\{X \geq a\}}$, we have
	\ben{ \label{xyilh}
		\pp \left( \sum_{i=1}^\ell X_i \geq \ell a +h  \right) \leq     \pp \left( \sum_{i=1}^\ell Y_i \geq h  \right) + \pp ( \, \exists \, i \in [\ell]: X_i \geq b),
	}
	where $Y_i = X_i \1{X_i \in [a,b]}$. By the assumption of $(X_i)_{i=1}^\ell$, we have $Y_i$ and $Y_j$ are independent if $|i-j| \geq m$ and
	\be{
		\pp(Y_i \geq t) \leq \exp(-ct) \quad \forall t>0, \quad \forall \, i \in [\ell].
	}
	In particular,  we have
	$\E [\exp(cY_i/2)] \leq 2$ for all $i \in [\ell]$.  Hence, using Part (i) we get
	\be{
	\pp \left( \sum_{i=1}^\ell Y_i \geq h  \right) \leq m\exp(-ch/2m) 2^{\lceil \ell/m \rceil}.
}
	On the other hand,  since $\pp(X_i \geq b) \leq \exp(-cb)$ for all $i \in[\ell]$,
	\be{
		\pp ( \, \exists \, i \in [\ell]: X_i \geq b) \leq \ell \exp(-cb).
	}
	Combining the last two estimates and \eqref{xyilh}, we obtain the result.
\end{proof}
\subsection{Prior estimates on simple random walks}\label{subsect:prior}
For any $A \subset  \Z^d$ and $s \in \R$, let $\kR^A_s$ be the set of vertices visited by the simple random walks starting in $A$ up to time $s$, i.e.,
\begin{align}\label{eq:RWrange}
    \kR^A_s := \bigcup_{x \in A} \{y \in \Z^d: t(x,y) \leq s \}.
\end{align}
Moreover, we simply write $\kR^x_s$ for $\kR^{\{x\}}_s$ and, in particular, set $\kR_s:=\kR_s^0$.
\begin{lem} \label{pzin}
	For any $m \geq 1$ and $A \subset \Z^d$, we have
	\be{
\pp \left( |\kR_m \cap A| \geq \tfrac{1}{2}\E[|\kR_m \cap A|] \right) \geq \frac{\E[|\kR_m \cap A|]}{12 \sup_{x \in A}\E[|\kR^x_m \cap A|]}.
}
\end{lem}
\begin{proof}
	We have
		\be{
		\E[|\kR_m \cap A|] = \sum_{x \in A} \pp(t(0,x) \leq m)
	}
and
\be{
\E[|\kR_m \cap A|^2] = \E[|\kR_m \cap A|] + \sum_{\substack{x,y \in A\\x \neq y}} \pp(t(0,x) \leq m, t(0, y) \leq  m).
}
By the Markov property,
	\bea{
		 &&\pp(t(0,x) \leq m, t(0, y) \leq  m) \\
		  &\leq &
		 \sum_{i=0}^{m-1} \Bigl\{ \pp(t(0,x)=i <t(0,y) \leq m) + \pp(t(0,y)=i <t(0,x) \leq m) \Bigr\} \notag \\
		& \leq & \sum_{i=0}^{m-1} \Bigl\{ \pp(t(0,x)=i)\,\pp(t(x,y) \leq m) + \pp(t(0,y)=i)\,\pp(t(y,x) \leq m) \Bigr\}\notag \\
		& \leq  & \pp(t(0,x) \leq m)\,\pp(t(x,y) \leq m) + \pp(t(0,y) \leq m) \,\pp(t(y,x) \leq m).
		}
Therefore,
\bea{
\E[|\kR_m \cap A|^2] & \leq & \E[|\kR_m \cap A|] +2 \sum_{\substack{x,y \in A\\x \neq y}}\pp(t(0,x) \leq m) \pp(t(x,y) \leq m)\\
&\leq& \E[|\kR_m \cap A|] +2 \E[|\kR_m \cap A|] \sup_{x \in A} \E[|\kR^x_m \cap A|] \\
&\leq & 3 \E[|\kR_m \cap A|] \sup_{x \in A} \E[|\kR^x_m \cap A|].
}
Recall the Paley--Zygmund inequality: for a nonnegative random variable $X$ with finite variance,
\begin{align*}
    \pp(X \geq \E[X]/2) \geq \E[X]^2/(4 \E[X^2]).
\end{align*}
We use this inequality with $X=|\mathcal{R}_m \cap A|$ to obtain
\begin{align*}
    \pp \left( |\kR_m \cap A| \geq \frac{1}{2}\E[|\kR_m \cap A|] \right)  \geq \frac{\E[|\kR_m \cap A|]^2}{4\E[|\kR_m \cap A|^2]} \geq \frac{\E[|\kR_m \cap A|]}{12 \sup_{x \in A}\E[|\kR^x_m \cap A|]},
\end{align*}
which is the desired conclusion.
\end{proof}
\begin{lem} \label{kr}
The following hold:
	\begin{itemize}
		\item [(i)] If $d\geq 3$, then for any  $a>0$, there  exists  $c=c(a)>0$ such that
\begin{align} \label{eq: hitting time estimate d>2}
	\pp(t(0,z)\leq a|z|^2) \geq \frac{c}{|z|^{d-2}} \qquad \forall \, z \in \Z^d \setminus \{0\}.
\end{align}
Moreover, there exists a universal constant $C$ such that
\ben{ \label{uptoz}
 \pp(t(0,z)<\infty)\leq  \frac{C}{|z|^{d-2}} \qquad \forall \, z \in \Z^d \setminus \{0\}.
}
 If $d=2$,  then for any  $a>0$, there  exists  $c=c(a)>0$ such that
\begin{align} \label{eq: lower hitting time estimate d=2}
	\pp(t(0,z)\leq a|z|^2) \geq \frac{c}{\log{(1+|z|)}} \qquad \forall \, z \in \Z^d \setminus \{0\}.
\end{align}
Moreover, for any $n$  large enough and $z\in \Z^2$ with $\log n \leq |z| \leq \sqrt{n}$,
 \begin{equation}
     \pp(t(0,z)\leq n) \leq\frac{2 + \log (\tfrac{n}{|z|^2})}{\log n}.  \label{eq: hitting time estimate d=2}
 \end{equation}
 \item[(ii)] For any $a>0$, there exists $c=c(a)>0$ such that for all subsets $A \subset \Lambda(\sqrt{n}/a)$ with $|A| \geq a n^{d/2}$,
 \begin{alignat*}{2}
     &\pp \left( |\kR_{a n} \cap A| \geq cn \right) \geq c &&\qquad \text{if $d\geq 3$},\\[0.5em]
     &\pp \left( |\kR_{a n} \cap A| \geq  cn/ \log n \right) \geq c &&\qquad \text{if $d=2$}.
 \end{alignat*}
\item[(iii)] For any $\e>0$ and $\mathbf{a}=(a_1,\ldots,a_d)\in  (0,\infty)^d$,  there exists $c=c(\e,\mathbf{a})>0$ such that
\begin{alignat*}{2}
    &\min_{x \in \Lambda} \pp(|\kR^x_{\e n}\cap F_i| \geq c\sqrt{n} \,\, \forall \, i\in [2d]) \geq c && \qquad \text{if $d\geq 3$},\\[0.5em]
    &\min_{x \in \Lambda} \pp(|\kR^x_{\e n}\cap F_i| \geq c \sqrt{n}/\log n \,\, \forall \, i \in [2d]) \geq c && \qquad\text{if $d=2$},
\end{alignat*}
where   $\Lambda:=\prod_{i=1}^d[0,a_i\sqrt{n}]\cap \Z^d$ and  $(F_i)_{i\in [2d]}$ are its facets 
 \end{itemize}
\end{lem}
\begin{proof}
In this proof, the symbols $c$ and $C$ represent positive constants, with $c$ denoting small values and $C$ large ones. Their values may change from line to line. When these constants depend on specific parameters, we indicate this explicitly, e.g., by writing $c = c(a)$ to show dependence on $a$. Let us first prove part~(i).
The inequalities \eqref{eq: lower hitting time estimate d=2} and \eqref{eq: hitting time estimate d=2} are direct consequences of \cite[Theorem 1.6]{U1}.
Furthermore, the inequality \eqref{eq: hitting time estimate d>2} is obtained in \cite[Theorem~{2.2} (ii)]{AlvMacPop02} and \eqref{uptoz} is a consequence of an estimate of the Green's function (see for instance \cite[Theorem 4.3.1]{LawLim10_book}) as
\begin{align*}
    \pp(t(0,z)< \infty)
    \leq \E[|\{n \geq 1: S_n=z \}|] \leq  \frac{C}{|z|^{d-2}}.
\end{align*}
We next treat part~(ii).
If $d \geq 3$, then by \eqref{eq: hitting time estimate d>2},  for all $x \in A \subset \Lambda(\sqrt{n}/a)$,
	\be{
\pp(t(0,x) \leq a n)\geq \pp(t(0,x) \leq a^2 |x|^2/d) \geq c(a) |x|^{2-d} \geq c(a) n^{1-d/2}.
}
Thus, since $|A| \geq a n^{d/2}$,
	\ben{ \label{era}
\E[|\kR_{an} \cap A|]=\sum_{x\in A}\pp(t(0,x)\leq an) \geq c(a) |A|	n^{1-d/2} \geq c(a) n.
}
 Moreover, since $|\kR^x_{an}|\leq an+1$ always holds, we have
$\sup_{x \in A}\E[|\kR^x_{a n}\cap A|] \leq a n+1$.
Hence, part~(ii) in $d \geq 3$ follows from the last two estimates and Lemma \ref{pzin}.
In the case $d=2$, by \eqref{eq: hitting time estimate d=2}, we instead have
\[\E[|\kR_{an} \cap A|]=\sum_{x\in A}\pp(t(0,x)\leq an) \geq \frac{c(a) |A|}{\log{n}} \geq \frac{c(a) n}{\log{n}}.\]
Moreover, it is known from \cite[Theorem~1]{DvoErd51} that in $d=2$, $\E[|\kR^x_{an}\cap A|]\leq  \E[|\kR^x_{an}|]\leq C(a) n/\log n$ holds, and hence part~(ii) in $d=2$ also follows from Lemma~\ref{pzin}. 

Finally, we consider part~(iii). Using the strong Markov property at the times when the random walk hits faces $(F_i)_{i\in [2d]}$,  it suffices to show that for some positive constant $c(\e,\mathbf{a})>0$,
\ben{ \label{erbi}
\min_{x \in \Lambda} \min_{i\in[ 2d]} \pp(|\kR^x_{\e n/2d}\cap F_i|\geq c\sqrt{n}) \geq c(\e,\mathbf{a}).
}
For $d \geq 3$, similarly to  \eqref{era}, one has for each $x \in \Lambda$ and $i \in [2d]$,
\ben{ \label{earobix}
\E[|\kR^x_{\e n/2d}\cap F_i|] = \sum_{y\in F_i}\pp(t(x,y)\leq \e n/2d)\geq c(\e,\mathbf{a}) |F_i| n^{1-d/2} \geq c(\e,\mathbf{a}) \sqrt{n}.
}
By \eqref{eq: hitting time estimate d>2},  for any distinct $x,y \in \Z^d$,
\be{
\pp(t(x,y) \leq \e n/2d)\leq  \pp(t(x,y) <\infty) \leq C(\e) |x-y|^{2-d}.
}
Therefore, since $|\{y: y \in F_i, |x-y|_1=k \}| \leq  C k^{d-2}$ for any $x \in \Z^d$,
\begin{align*}
    \E[|\kR^x_{\e n/2d}\cap F_i|] \leq  1+C(\e) \sum_{\substack{y \in F_i\\x \neq y}}|x-y|^{2-d} \leq 1+C(\e) \sum_{k=1}^{a_i\sqrt{n}} k^{2-d} k^{d-2} \leq C(\e,\mathbf{a}) \sqrt{n}.
\end{align*}
Combining this with \eqref{earobix} and Lemma~\ref{pzin}, we obtain \eqref{erbi} for $d\geq 3$.
Now assume that $d=2$. By \eqref{eq: lower hitting time estimate d=2},   for all $x,y  \in \Lambda =\prod_{i \in [d]} [0,a_i \sqrt{n}] \cap \Z^d$ with $x\not=y$,
\be{
\pp(t(x,y) \leq \e n/2d) \geq \pp(t(x,y) \leq c(\e,\mathbf{a}) |x-y|^2) \geq \frac{c(\e,\mathbf{a})}{\log(1+|x-y|)} \geq \frac{c(\e,\mathbf{a})}{ \log n}.
}
Hence, we have
\begin{align}\label{eq:range_lower}
    \E[|\kR^x_{\e n/2d}\cap F_i|]
    \geq \sum_{y\in F_i \setminus \{x\}}\pp(t(x,y)\leq \e n/2d)\geq \frac{c(\e,\mathbf{a})|F_i|}{\log{n}} \geq  \frac{c(\e,\mathbf{a}) \sqrt{n}}{\log{n}}.
\end{align}
On the other hand, by \eqref{eq: hitting time estimate d=2},  for all $z \in \Z^d \setminus \{0\}$ and $m \geq 1$,
\begin{align*}
    \pp(t(0,z) \leq m)
    &=\pp(t(0,z) \leq |z|^2) + \pp(|z|^2<t(0,z) \leq m)\\
    &\leq \frac{2}{\log(1+|z|)} +  \pp(|z|^2<t(0,z) \leq m).
\end{align*}
From \cite[Corollary~1.3]{U1}, for all $m \in \N$ and $z \in \Z^d$ satisfying $|z| \geq 2$ and $|z|^2 <m$,
\begin{align*}
    \pp(|z|^2<t(0,z) \leq m)
    &\leq C \sum_{k=|z|^2}^m \frac{\log |z|}{k (\log k)^2}\\
    &\leq C \log |z| \int_{|z|^2-1}^m \frac{dt}{t (\log t)^2}
    \leq \frac{C \log (m/|z|^2)}{\log m}.
\end{align*}
Hence,
\be{
\pp(t(0,z) \leq m) \leq  \frac{2}{\log(1+|z|)} + \frac{C \log (m/|z|^2) \vee 0}{\log m},
}
where $a \vee b = \max\{a,b\}$. It follows that  for each $x \in \Lambda$,
\begin{align*}
    \E[|\kR^x_{\e n/4}\cap F_i|]
    &\leq 5^d+\sum_{\substack{y \in F_i\\|x-y| \geq 2}} \pp(t(x,y) \leq \e n/4)\\
    &\leq 5^d+C(\e,\mathbf{a}) \sum_{\substack{y \in F_i\\|x-y| \geq 2}} \left( \frac{1}{ \log(1+|x-y|)} + \frac{\log (n/|x-y|^2)\vee 0}{\log n} \right)\\
    &\leq C(\e,\mathbf{a}) \sum_{k=2}^{C(\e,\mathbf{a}) \sqrt{n}} \left( \frac{1}{\log k} + \frac{\log (n/k^2)\vee 0}{\log n} \right) \leq \frac{C(\e,\mathbf{a}) \sqrt{n}}{ \log n}.
\end{align*}
Combined with \eqref{eq:range_lower} and Lemma~\ref{pzin}, this yields  \eqref{erbi} for $d=2$.
\end{proof}
\begin{lem}  \label{rnab}
For any $\delta \in (0,1)$, there exists  $c=c(\delta)>0$ such that, for all $n$ sufficiently large, the following holds:
for any sets $A, B \subset \Z^d$ satisfying $\max_{x \in A, y \in B} |x-y|_{\infty} \leq \sqrt{n}/\delta$ and $|B| \geq \delta n^{d/2}$, we have
\begin{alignat*}{2}
    &\pp \left( |\kR^A_n \cap B| \geq \min \{c n|A|, (1-\delta)|B|\} \right) \geq 1-\exp(-c|A|) &&\qquad \text{if $d\geq 3$},\\[0.5em]
    &\pp \left( |\kR^A_n \cap B| \geq \min \{c n|A|/\log n, (1-\delta)|B|\} \right) \geq 1-\exp(-c|A|) &&\qquad \text{if $d=2$}.
\end{alignat*}
\end{lem}
\begin{proof}
We only prove here the case $d\geq 3$, since the proof for $d=2$ is similar. By Lemma~\ref{kr}~(ii), there exists a constant  $c_0=c_0(\delta) \in (0,1)$ such that $\inf_{x\in A}{\inf_{\Gamma\subset B,|\Gamma| \geq \delta |B|}}\pp( | \kR^{x}_{n} \cap  \Gamma| \geq c_0 n )\geq c_0,$ or equivalently
\ben{\label{pxi=1}
\sup_{x\in A}\sup_{\Gamma\subset B, |\Gamma| \geq \delta |B|}\pp( | \kR^{x}_{n} \cap  \Gamma| < c_0 n )< 1-c_0.
}
We enumerate the elements of $A$ as $x_1,\dots,x_\ell$ with $\ell=|A|$, and  define for $i=1,\ldots \ell $,
\[ X_i := \mathbf{1}\left\{ |\kR^{x_i}_{n} \cap (B\backslash \left(\cup_{j<i } \, \kR^{x_j}_{n}\right))| < c_0 n \right\}. \]
Set $\kF_i:=\sigma(S^{x_1}_\cdot,\dots,S^{x_i}_\cdot)$ for $i\geq 1$.
It is clear that $X_i$ is $\kF_i$-measurable.
We now define the stopping time
\[ \tau := \inf \left \{ i \leq \ell: \Big|B\setminus \Big(\cup_{j \leq i} \kR^{x_i}_n\Big)\Big|  < \delta|B| \right \}. \]
By \eqref{pxi=1}, for all $i \leq \ell$
\be{
\E[X_i \mathbf{1}\{\tau > i-1\} \mid \kF_{i-1}] \leq 1-c_0 \, \textrm{ a.s.}
}
Hence, by Lemma~\ref{dobbi}, we can find $c_1=c_1(c_0)>0$ such that
\ben{ \label{sxi}
\pp \left( \sum_{i=1}^\ell X_i \geq (1-c_1)\ell,\, \tau \geq \ell \right) \leq \exp(-c_1 \ell).
}
If $\tau<\ell$, then
\begin{align*}
    |\kR^A_n \cap B|=\Bigl|\Bigl(\cup_{j \leq \ell} \kR^{x_i}_n\Bigr) \cap B\Bigr|
    \geq \Bigl|\Bigl( \cup_{j \leq \tau} \kR^{x_i}_n\Bigr) \cap B \Bigr|
    \geq (1- \delta)|B|.
\end{align*}
On the other hand, if $\sum_{i=1}^\ell X_i<(1-c_1)\ell$, or equivalently $\sum_{i=1}^\ell (1-X_i)>c_1 \ell $, then
\begin{align*}
    | \kR^A_n \cap B|=\Big| \Big(\cup_{j \leq \ell} \kR^{x_i}_n\Big) \cap B\Big|>c_0 n\cdot c_1 \ell=cn|A|,
\end{align*}
with $c:=c_0c_1$.
These observations, combined with \eqref{sxi} and $|A|=\ell$, prove that
\begin{align*}
    &\pp \left( |\kR^A_n \cap B|<\min \{c n|A|, (1-\delta)|B|\} \right)\\
    &\leq \pp \left( |\kR^A_n \cap B|<cn|A|,\,\tau \geq \ell \right)+\pp\left( |\kR^A_n \cap B|<(1-\delta)|B|,\,\tau<\ell \right)\\
    &\leq \pp \left( \sum_{i=1}^\ell X_i \geq (1-c_1)\ell,\,\tau \geq \ell \right)
    \leq \exp(-c_1\ell) \leq \exp(-c|A|),
\end{align*}
as $c<c_1$, and the lemma follows.
\end{proof}

\section{Lower Bound on the Large Deviation Probability via the Avoidance Event} \label{sec:lbae}

In this section, we assume that  $v=\eb_1$, since the proof for a general unit vector  $v \in  \R^d$ is essentially the same and omitted. We  aim to prove the lower bound
\be{
\pp ({\rm T}(0, n\eb_1) \geq (\mu+\xi)n) \geq \exp (-\fsd(n)(\frd \xi +o(1))),
}
where $\mu=\mu(\eb_1)$, or equivalently
\begin{align}
	\label{lr3d}
	&\limsup_{n \rightarrow \infty} \frac{-1}{\fsd(n)} \log \pp({\rm T}(0, n\eb_1) \geq (\mu+\xi)n) \leq \mathfrak{r}_d \xi.
\end{align}
As discussed in the subsection \ref{subsec:orgproof}, we shall give a  lower bound on the large deviation probability via the avoidance event, i.e. by  the probability that the endpoint $n \eb_1$ is not visited by any particle by time $\xi n$. Moreover, we show that this probability decays as $\exp(-(\mathfrak{r}_d \xi +o(1))\mathfrak{s}_d(n))$

For the latter use in the  upper bound, we present a modified version of avoidance probability. Roughly speaking, with $M$ a large constant, the ending vertex $n \eb_1$ is ignored within $\xi n$ steps by  most particles in $\Lambda(n \eb_1, M\sqrt{n})$  with probability  $\exp(-(\mathfrak{r}_d \xi +o_M(1))\mathfrak{s}_d(n)))$. More precisely,    given $\xi >0$, $n \in \N$ and $A\subset \Z^d$, we define the avoidance rate
\begin{align*}
	r_{A,n} (\xi):=\frac{-1}{\mathfrak{s}_d(n)} \log\pp(t(x,0)> \xi n \quad \forall x \in A).
\end{align*}
For $M,n \in \N$ and $\delta \in (0,1)$, define
\begin{align}\label{eq:kA_def}
	\kA_M^\delta(n):=\bigl\{ A\subset\Z^d \setminus \{0\}: |A\cap  \Lambda(M\sqrt{n})| \geq (1-\delta) |\Lambda(M\sqrt{n})| \bigr\}.
\end{align}
\begin{lem} [Avoidance rate is linear] \label{Step73d} For any $\xi >0$, it holds that
	\ben{ \label{step7}
		\limsup_{M\rightarrow \infty} \limsup_{\delta \rightarrow 0} \limsup_{n \to \infty} \sup_{A \in \kA_M^\delta(n)} |r_{A,n}(\xi) - \mathfrak{r}_d \xi | = 0.
	}
\end{lem}

\subsection{Proof of the lower bound~\eqref{lr3d} assuming Lemma~\ref{Step73d}} \label{polwb}
\begin{proof}
Fix $\xi>0$ and let $\e>0$ be a small positive constant.
Then, define for any $n \in \N$,
\begin{align*}
    &\kE_{\text{box}}(n):=\bigl\{ {\rm T}(0,\partial \Lambda(n\eb_1,n^{2/3})) \geq (\mu-\e) n \bigr\},\\
    &\kE_{\text{in}}(n):=\bigl\{ t(x, n\eb_1) \geq (\xi+\e) n \quad \forall x \in \Lambda(n\eb_1,n^{2/3}) \setminus \{n\eb_1 \} \bigr\},\\
    &\kE_{\text{out}}(n):=\bigl\{ t(x, n\eb_1) \geq (\mu +\xi)n \quad \forall\ x \in \Z^d \setminus \Lambda(n\eb_1,n^{2/3}) \bigr\},
\end{align*}
where for $A \subset \Z^d$, $\partial A$ is the inner boundary of $A$, i.e.,
	\begin{align*}
		\partial A:=\{ y \in A:\exists x \not \in A: |y-x|_1 =1 \}.
	\end{align*}
	Note that the frog that finally hits  $n\eb_1$ in the optimal path of ${\rm T}(0,n\eb_1)$ either starts inside $\Lambda(n\eb_1,n^{2/3})$ or outside $\Lambda(n\eb_1,n^{2/3})$. 
	In the case where the finally approaching frog starts at $x_0 \in \Lambda(n\eb_1,n^{2/3})$, we have on $\kE_{\text{box}} \cap \kE_{\text{in}}$,
	\begin{align*}
		{\rm T}(0,n\eb_1)
		\geq {\rm T}(0,\Lambda(n\eb_1,n^{2/3}))+t(x_0,n\eb_1)
		\geq (\mu-\varepsilon)n+(\xi+\varepsilon)n=(\mu+\xi)n.
	\end{align*}
	On the other hand, in the case where the finally approaching frog starts at $x_0 \not\in \Lambda(n\eb_1,n^{2/3})$, ${\rm T}(0,n\eb_1) \geq t(x_0,n\eb_1) \geq (\mu+\xi)n$ holds on $\kE_{\text{out}}$.
	Hence, ${\rm T}(0,n\eb_1) \geq (\mu+\xi)n$ on $\kE_{\text{box}} \cap \kE_{\text{in}} \cap \kE_{\text{out}}$.
	Since $\varepsilon$ is arbitrary, \eqref{lr3d} follows once we prove
	\begin{align}\label{eq:upper_goal}
\limsup_{n \to \infty}\frac{-1}{\fsd(n)}\log\pp(\kE_{\text{box}}(n) \cap \kE_{\text{in}}(n) \cap \kE_{\text{out}}(n))
\leq \mathfrak{r}_d\xi+2\varepsilon.
\end{align}
Let us first estimate $\pp(\mathcal{E}_{\text{box}}(n)^c)$.
	To this end, we recall the result obtained in \cite[Theorem~{1.2}]{Kub19}:
	given $\delta>0$, there exist positive constants $\alpha$ (which is independent of $\delta$) and $c=c(\delta)$ such that for all $y \in \Z^d$,
	\begin{align*}
		\pp({\rm T}(0,y) \leq (1-\delta)\mu(y)) \leq \exp(-c |y|_1^\alpha).
	\end{align*}
Moreover, for all $y \in \partial\Lambda(n\eb_1,n^{2/3})$, one has 
$$\mu(y) \geq \mu(n\eb_1)-\mu(n\eb_1-y) \geq n\mu -Cn^{2/3},$$
with $C$ a universal constant.  Therefore,  for  $n$ large enough,
\be{
\pp({\rm T} (0,y) \leq n(\mu-\e) ) \leq \pp({\rm T} (0,y) \leq (1-\delta) \mu(y)) \leq \exp(-c|y|_1^\alpha),
}
where $\delta=\delta(\e)$ and $c=c(\delta)$ are positive constants. Hence,
	\begin{align*}
		\pp(\kE_{\text{box}}(n)^c) \leq \sum_{y \in \partial\Lambda(n\eb_1,n^{2/3})} \exp(-c |y|_1^\alpha) \leq \frac{1}{4}.
	\end{align*}
	Next, since $\Lambda(n^{2/3}) \setminus \{0\} \in \kA_M^\delta(n)$ for all $n$ sufficiently large, Lemma~\ref{Step73d} shows that
	\begin{align*}
		\pp(\kE_{\text{in}}(n))
		= \pp\Bigl( t(x,0) \geq (\xi+\e) n \quad \forall x \in \Lambda(n^{2/3}) \setminus \{0\} \Bigr)
		\geq \exp(-(\mathfrak{r}_d \xi+2\e)\fsd(n)).
	\end{align*}
	Finally, we consider $\pp(\kE_{\text{out}}(n)^c)$. By \cite[Proposition~{2.1.2}~(b)]{LawLim10_book},
	there exists a  constant $C>0$ such that for all $n \in \N$ and $s>0$,
    \begin{align} \label{rsn}
		\pp\Bigl( \max_{0 \leq i \leq n}|S_i^0| \geq s\sqrt{n} \Bigr) \leq Ce^{-s^2/C}.
	\end{align}
	This tells us that for all $n$ sufficiently large,
	\begin{align*}
		&\pp(\kE_{\text{out}}(n)^c)
		\leq \sum_{x \not\in \Lambda(n^{2/3})} \pp(t(0,x)<(\mu+\xi)n)\\
		&\leq \sum_{x \not\in \Lambda(n^{2/3})} \pp \left( \max _{0 \leq i \leq (\mu +\xi)n } |S^0_i| \geq |x|_1\right)\\
		&\leq \sum_{x \not \in\Lambda(n^{2/3})} C\exp\biggl( -\frac{\beta|x|_1^2}{(\mu+\xi)n} \biggr)
		\leq C\sum_{k\geq n^{2/3}} (2k+1)^{d} \exp\Bigl( -\frac{ k^2}{C(\mu+ \xi)n} \Bigr)
		\leq \frac{1}{4}.
	\end{align*}
	Consequently, noting that $\kE_{\text{in}}(n)$ and $\kE_{\text{box}}(n) \cap \kE_{\text{out}}(n)$ are independent, one has for all $n$ sufficiently large,
	\begin{align*}
		\pp(\kE_{\text{box}}(n) \cap \kE_{\text{in}}(n) \cap \kE_{\text{out}}(n))
		&= \pp(\kE_{\text{in}}(n))\,\pp(\kE_{\text{box}}(n) \cap \kE_{\text{out}}(n))\\
		&\geq \pp(\kE_{\text{in}}(n)) \bigl(1-\pp(\kE_{\text{box}}(n)^c)-\pp(\kE_{\text{out}}(n)^c) \bigr)\\
		&\geq \frac{1}{2}\pp(\kE_{\text{in}}(n)) \geq \frac{1}{2} \exp(-(\mathfrak{r}_d \xi+2\eps)\fsd(n)),
	\end{align*}
	and \eqref{eq:upper_goal} immediately follows.
\end{proof}

\subsection{Estimates on the avoidance probabilities: Proof of Lemma \ref{Step73d}}  \label{sec:avprob}

We first recall a well-known result (see \cite{DvoErd51}) on  the range of a simple  random walk: 
	\ben{ \label{exran}
		\lim_{n \rightarrow \infty} \frac{\E[|\mathcal{R}_n]|}{\mathfrak{s}_d(n)} = \mathfrak{r}_d,
	}
	where 
	\be{
	\mathfrak{s}_d(n) =
    \begin{cases}
        n& (d\geq 3)\\
        \frac{n}{\log n}& (d =2),
    \end{cases}, \qquad 	\mathfrak{r}_d =\begin{cases}
         \mathfrak{g}_d & (d\geq 3)\\
       \pi& (d =2),
    \end{cases},
    }
	where $\mathfrak{g}_d = \pp(S_i \neq 0 \, \forall \, i \geq 1 )$ is the escape probability  of the simple random walk started at $0$.

\begin{proof}[Proof of Lemma \ref{Step73d}]  Recall the statement of the lemma: 
	\ben{ \label{step7p}
		\limsup_{M\rightarrow \infty} \limsup_{\delta \rightarrow 0} \limsup_{n \to \infty} \sup_{A \in \kA_M^\delta(n)} |r_{A,n}(\xi) - \mathfrak{r}_d \xi | = 0,
	}
where for $A \subset \Z^d$ and $n \in \N$
	\begin{align*}
		r_{A,n}(\xi) =\frac{-1}{\mathfrak{s}_d(n)} \log\pp(t(x,0) \geq   \xi n \quad \forall x \in A),
	\end{align*}
and for $M \in \N$, $\delta \in (0,1)$ and $n \in \N$
	\begin{align*}
		\kA_M^\delta(n)=\bigl\{ A\subset\Z^d \setminus \{0\}: |A\cap  \Lambda(M\sqrt{n})| \geq (1-\delta) |\Lambda(M\sqrt{n})| \bigr\}.
	\end{align*}	

First, we assume \eqref{step7p} with $\xi=1$, and prove that \eqref{step7p} holds for every other positive number $\xi$. 
Fix $\xi \not= 1$ and $M \in \N$, and set $M_\xi:=\lceil M/\sqrt{\xi}\rceil$.
Then, since $M\sqrt{n} \leq M_\xi\sqrt{\lceil \xi n \rceil}$ and $\lim_{n \to \infty} \lceil \xi n \rceil/\xi n=1$, for any $\delta \in (0,1/2)$, $\mathcal{A}_M^\delta(n) \subset \mathcal{A}_{M_\xi}^{2\delta}(\lceil \xi n \rceil)$ holds if $n \in \N$ is sufficiently large.
Hence, for any $\delta \in (0,1/2)$, if $n$ is large enough, then
\begin{align*}
	\sup_{A \in \kA_M^\delta(n)}|r_{A,n} (\xi)- \mathfrak{r}_d \xi|
	&\leq \sup_{A \in \kA_{M_\xi}^{2\delta}(\lceil \xi n \rceil)}|r_{A,n} (\xi)- \mathfrak{r}_d \xi|\\
	&\leq \frac{\mathfrak{s}_d(\lceil \xi n\rceil)}{\mathfrak{s}_d(n)} \sup_{A \in \kA_{M_\xi}^{2\delta}(\lceil \xi n \rceil)}\left|r_{A,\lceil \xi n \rceil}(1)- \mathfrak{r}_d\right|+\mathfrak{r}_d\left| \frac{\mathfrak{s}_d(\lceil \xi n\rceil)}{\mathfrak{s}_d(n)}-\xi \right|.
\end{align*}
By the assumption \eqref{step7p} with $\xi=1$ and the definition of $\mathfrak{s}_d(\cdot)$,
\begin{align*}
	&\limsup_{M\rightarrow \infty} \limsup_{\delta \rightarrow 0} \limsup_{n \to \infty} \sup_{A \in \kA_M^\delta(n)} |r_{A,n}(\xi) - \mathfrak{r}_d \xi|\\
	&\leq \limsup_{M\rightarrow \infty} \limsup_{\delta \rightarrow 0} \limsup_{n \to \infty} \sup_{A \in \kA_{M_\xi}^{2\delta}(\lceil \xi n \rceil)}\left|r_{A,\lceil \xi n \rceil}(1)-\mathfrak{r}_d\right|=0,
\end{align*}
and \eqref{step7p} is proved for all $\xi \not=1$.

 Now, it remains to prove \eqref{step7p} for $\xi =1$. For simplicity of notation, we write $	r_{A,n}=r_{A,n} (1)$. 	 We first consider the case  $d\geq 3$. 	Having $r_{A,n}\leq r_{\Z^d\setminus \{0\},n}$ in mind, we first prove
	\begin{align}\label{eq: first goal}
		\limsup_{n\to\infty}r_{\Z^d\setminus \{0\},n}\leq \mathfrak{g}_d.
	\end{align}
	Since $(t(z,0))_{z \in \Z^d}$ are independent, we have the following decomposition 
	\begin{align*}
		\pp( t(z,0)> n \quad \forall z \in \Z^d \setminus \{0\})
		& = \prod_{  |z| <  \log n} \pp( t(z,0)> n) \times  \prod_{ |z| \geq \log n} \pp( t(z,0)> n)\\
		&= \prod_{  |z| <  \log n} \pp( t(0,z)> n) \times  \prod_{ |z| \geq \log n} \pp( t(0,z)> n)\\
		&=: X^{(1)}_n \times X^{(2)}_n.
	\end{align*}
	By \eqref{rsn}, there exists a constant $s>1$, such that for all $n\in \N$,
	\ben{ \label{rns}
		\pp(\kR_n \subset \Lambda(s\sqrt{n})) \geq 1/2.
	}
	For any $z \in \Z^d$, we take a vertex $x$  with $|x| \geq 2s \sqrt{n}$, and a path $\gamma$ starting at $0$ and ending at $x$ with $|\gamma| \leq 4s \sqrt{n}$ and $z \not \in \gamma \cup \Lambda(x,s\sqrt{n})$. We then have
	\al{
		\pp(t(0,z) \geq n) = \pp(z \not \in \kR_n) &\geq \pp((S_0,\ldots,S_{|\gamma|})=\gamma, \forall k\in \Iintv{|\gamma|+1,n},\, S_k\in  \Lambda(x,s\sqrt{n}))\\
		&= (1/2d)^{|\gamma|} \pp(\kR^x_{n-|\gamma|} \subset \Lambda(x,s\sqrt{n})) \\
		&\geq  (1/2d)^{|\gamma|}\times (1/2) \geq  \exp(-16 d s \sqrt{n}),
	}
	where we have used \eqref{rns} and the Markov property. Therefore,
	\ben{ \label{x13d}
		0  \geq \log  X^{(1)}_n \geq -16 d s \sqrt{n}(2 \log n +1)^d.
	}
	Next, we  estimate $X^{(2)}_n$. By Lemma~\ref{kr}~(i), there exists $C=C(d)$ such that for any $z\in \Z^d \setminus \{0\}$,
	\ben{ \label{tozinf}
		\pp(t(0,z)< \infty) \leq C |z|^{2-d}.
	}
	Hence, we have 
	\ben{ \label{pzlogn}
		\lim_{n\to\infty}\sup_{|z| \geq \log n}\pp(t(0,z)\leq n) =0.
	}
	Consequently, the approximation $-x(1+o(1)) = \log (1-x)$ as $|x|\to 0$ implies
	\begin{align*}
		- \log \pp(t(0,z)>n) =- \log (1-\pp(t(0,z) \leq n))\leq (1+\epsilon_n)\pp(t(0,z) \leq n),
	\end{align*}
	where as $n\to \infty$, $\epsilon_n$ goes to $0$ uniformly in $z\in\Z^d$ with $|z| \geq \log n$.
	Thus,
	\begin{align*}
		-\frac{1}{n}\log X^{(2)}_n \leq  \frac{1+\epsilon_n}{n} \sum_{\substack{z \in \Z^d\\|z|\geq \log n}} \pp(t(0,z) \leq n)\leq  \frac{1+\epsilon_n}{n} \E[|\kR_n|] .
	\end{align*}
	This, combined with \eqref{x13d} and the fact that $\E[|\kR_n|] = n (\mathfrak{g}_d +o(1))$ by  \eqref{exran}, yields 
	\begin{align*}
		\limsup_{n\to\infty} r_{\Z^d\setminus \{0\},n}
		\leq \limsup_{n\to\infty}\Big(-\frac{1}{n}\log X^{(1)}_n-\frac{1}{n}\log X^{(2)}_n\Big) \leq \mathfrak{g}_d,
	\end{align*}
	and hence \eqref{eq: first goal} follows.	Our second task is to show the opposite bound, i.e.,
	\begin{align}\label{eq: second goal}
		\liminf_{M\rightarrow \infty} \liminf_{\delta \rightarrow 0} \liminf_{n \to \infty} \inf_{A \in \kA_M^\delta(n)} r_{A,n} \geq \mathfrak{g}_d.
	\end{align}
	Let $\delta\in (0, (4M)^{-16d})$.
	Define for any $n \in \N$ and $A \in \kA_M^\delta(n)$,
	\begin{align*}
		X_{A,n}:= \prod_{\substack{z\in A\\|z| \geq \log n}} \pp( t(0,z)> n).
	\end{align*}
	The fact that $x\leq -\log(1-x)$ for $x\in [0,1)$ implies that
	\begin{align}\label{eq: second goal}
		\begin{split}
			-\log X_{A,n}
			&\geq \sum_{\substack{z\in A\\|z|\geq \log n}} \pp(t(0,z) \leq n)			\geq \E[|\kR_n \cap A|]-|\Lambda(\log n)|\\
			&\geq \E[|\kR_n|]-\E[|\kR_n\cap  \Lambda(M\sqrt{n})^c|] \\
			& \hspace{3 cm} -\E[|\kR_n \cap A^c \cap \Lambda(M\sqrt{n})|]-(2\log n +1)^d.
		\end{split}
	\end{align}
	We first estimate the second term on the right-hand side of \eqref{eq: second goal}. 
	By the local central limit theorem (see for instance \cite[Theorem 2.1.3]{LawLim10_book}), there exists a positive constant $C$ such that, for $z\in \Z^d\setminus\{0\}$,
	\begin{align*}
		\pp(t(0,z) \leq n)
		&\leq \sum_{k=|z|}^n \pp(S_k=z) \leq  C \sum_{k=|z|}^n k^{-d/2} \exp(-|z|^2/(Ck))\\
		& \leq C\int_{|z|}^n t^{-d/2} \exp(-|z|^2/(Ct))\,dt
		= C|z|^{2-d} \int_{|z|^2/n}^{|z|} u^{(d-4)/2} \exp(-u/C)\,du\\
		&\leq C |z|^{2-d} \exp\Bigl(-\frac{|z|^2}{Cn} \Bigr).
	\end{align*}
	Therefore,
	\begin{align}\label{eq:RLambda}
		\begin{split}
			\E [|\kR_n \cap  \Lambda(M\sqrt{n})^c|]
			&=\sum_{|z| \geq M\sqrt{n} }\pp(t(0,z) \leq n)
			\leq C \sum_{|z| \geq M\sqrt{n} }|z|^{2-d} \exp\Bigl( -\frac{|z|^2}{Cn} \Bigr)\\
			&\leq C\sum_{k \geq M\sqrt{n} } k \exp\Bigl( -\frac{k^2}{Cn} \Bigr)
			\leq C\int_{M\sqrt{n}-1}^\infty t \exp\Bigl( -\frac{t^2}{Cn} \Bigr)\,dt\\
			&\leq C n \exp(-M^2/C).
		\end{split}
	\end{align}
	Next, we estimate the third term on the right-hand side of \eqref{eq: second goal}. Note that 
\ben{
|B| \leq \delta |\Lambda(M\sqrt{n})| \leq \delta (4M\sqrt{n})^d, \quad \textrm{where} \quad B:= A^c \cap \Lambda(M\sqrt{n}).
}
It has been shown in \cite[Appendix A.2]{AC} that
\be{
\E[|\kR_\infty \cap B|] \leq C' |B|^{2/d} \leq 16C'M^2 \delta^{2/d} n,
}
where $C'$ is a universal constant.  
Therefore,
\bea{
\E[|\kR_n \cap A^c \cap \Lambda(M\sqrt{n})|]  \leq \E[|\kR_\infty \cap B|] \leq  16C'M^2 \delta^{2/d} n.
}
	Combining this with \eqref{eq: second goal} and \eqref{eq:RLambda}, one has
	\begin{align*}
		-\log X_{A,n}
		\geq \E[|\kR_n|] - (C+C')n\bigl\{ \exp(-cM^2/4)+16M^2\delta^{2/d} \bigr\}.
	\end{align*}
Consequently, for any $M,n \in \N$ and $\delta$ sufficiently small, for each $A \in \kA_M^\delta(n)$,
	\begin{align*}
r_{A,n}(1)-\mathfrak{g}_d
		&\geq -\frac{1}{n}\log X_{A,n}-\mathfrak{g}_d\\
		&\geq \frac{\E[|\kR_n|]}{n}-(C+C')\bigl\{ \exp(-cM^2/4)+M^2\delta^{1/(8d)} \bigr\}-\mathfrak{g}_d.
	\end{align*}
In addition, 	 $\E[|\kR_n|] = n (\mathfrak{g}_d +o(1))$. Therefore, 
\begin{align*}
\liminf_{M\rightarrow \infty} \liminf_{\delta \rightarrow 0} \liminf_{n \to \infty} \inf_{A \in \kA_M^\delta(n)} r_{A,n}(1)-\mathfrak{g}_d  \geq 0,
	\end{align*}
and \eqref{eq: second goal} follows. 
	
	\medskip
	Next, we consider the case $d=2$. As in the case $d\geq 3$, we decompose
	\begin{align*}
		&\pp( t(z,0)> n \quad \forall z \in \Z^2 \setminus \{0\})\\
		&= \prod_{\substack{z \in \Z^2\\|z| <\log n}} \pp( t(0,z)> n) \times  \prod_{\substack{z \in \Z^2\\ \log n \leq |z|<\tfrac{\sqrt{n}}{\log n}}} \pp( t(0,z)> n) \prod_{\substack{z \in \Z^2\\  |z|>\tfrac{\sqrt{n}}{\log n}}} \pp( t(0,z)> n)\\
		&=: Y^{(1)}_n \times Y^{(2)}_n \times Y^{(3)}_n.
	\end{align*}
	First, note that the estimate \eqref{x13d} holds for all dimensions and it gives
	\ben{ \label{y1}
		0  \geq \log  Y^{(1)}_n \geq -c\sqrt{n}( \log n)^2,
	}
	with $c$ a positive constant. Furthermore, by Lemma \ref{kr} (i), we have
	\[
	\sup_{|z|\geq \tfrac{\sqrt{n}}{\log n}} \pp(t(0,z)\leq n) \leq\frac{2 + \log  \log n }{\log n} =o(1).
	\]
	Therefore, using the same argument as for $X_n^{(2)}$, with the above estimate in place of \eqref{pzlogn}, we obtain
	\begin{align}
		\label{y3}
		-\log Y^{(3)}_n \leq  (1+o_n(1))\E[|\kR_n|].
	\end{align}
	Finally, Finally, Lemma~\ref{kr}~(i) gives, for all sufficiently large $n$, the following bound:
    \bea{
		Y^{(2)}_n &\geq& \prod_{k=\lfloor \log n \rfloor}^{\lceil \sqrt{n}/\log n \rceil} \left(1 -\frac{2+\log(\tfrac{n}{k^2})}{\log n}\right)^{Ck}  \geq  \prod_{k=\lfloor \log n \rfloor}^{\lceil \sqrt{n}/\log n \rceil} \left(\frac{1}{\log n}\right)^{Ck} \\
		&\geq &  \exp \left(- C \log \log n \sum_{k=\lfloor \log n \rfloor}^{\lceil \sqrt{n}/\log n \rceil} k \right) \geq \exp \left(- \frac{Cn \log \log n}{ \log^2 n} \right),
	}
where $C$ is a positive constant. This together with \eqref{y1} and \eqref{y3} and the fact that $\E[|\kR_n|]=( \pi + o(1)) n/ \log n$ implies that
	\be{
\limsup_{n\rightarrow \infty}r_{\Z^2 \setminus \{0\},n}=\limsup_{n\rightarrow \infty} \frac{-\log n}{n} \log\pp( t(z,0)> n \quad \forall z \in \Z^2 \setminus \{0\})  \leq   \pi.
	}
	The remaining part of the proof of the opposite inequality
\begin{align*}
		\liminf_{M\rightarrow \infty} \liminf_{\delta \rightarrow 0} \liminf_{n \to \infty} \inf_{A \in \kA_M^\delta(n)} r_{A,n} \geq \pi
\end{align*}
is  the same as in the case $d\geq 3$, so we omit the details.
\end{proof}

\section{Upper bound on the large deviation probability  for $d\geq 3$} \label{sec:theomain1} We concentrate here on the $\mathbf{e}_1$ direction. Comments on the proof in the general direction are given in Subsection~\ref{subsec: general direction upper d>2}.  Let $d\geq 3$ and $\mu =\mu(\eb_1)$. The aim of this section is to prove that for any $\xi>0$
\be{
\liminf_{n \rightarrow \infty} \frac{-1}{n} \log 	\pp({\rm T}(0,n\eb_1) \geq (\mu + \xi) n) \geq \fgd \xi,
}
or equivalently,  for any small constant $\e>0$, if $n$ is large enough, then
\ben{ \label{uuld}
	\pp({\rm T}(0,n\eb_1) \geq (\mu + \xi) n) \leq \exp(-\mathfrak{g}_d \xi(1- \e) n).
}
\subsection{Elements of the proof}\label{section: elements}
We begin by giving a brief exposition of the steps of the proof of \eqref{uuld}. 
The proof is divided into three main steps  and a preparatory step as discussed in the subsection \ref{subsec:orgproof},  and the main technique is the renormalization via slab arguments. Given $\e>0$, let $K,L,M,\delta$ be constants chosen carefully (see Lemmas \ref{kvth}, \ref{lem:t1}, \ref{lem:pec}  and the proof of \eqref{uuld} below).

\medskip

\noindent \underline{Step 0} (Breeding seeds).  We show that the infection quickly spreads:
\begin{lem} \label{kvth}
	Let
	\begin{align}\label{def: kappa}
		\kappa := -  \sup \biggl\{ |A| \times \limsup_{n \rightarrow \infty} \frac{\log \pp(\kR_n \subset A)}{n} : 0 \in A \subset \Z^d \text{ connected, finite} \biggr\}.
	\end{align}
	Then, for any $K \in \N$ and $\e \in (0,1)$, if $n$ is large enough, then 
	\begin{align}\label{eq:kappa-g}
	\pp( |{\rm B}_{ n}|\leq K)
		\leq \exp(-(\kappa -\e)n). 
	\end{align}
Moreover, $\kappa > 2 \mathfrak{g}_d$. Consequently, for any $K \in \N$  if $n$ is large enough, then
\begin{align}\label{eq:kappa}
	\pp( |{\rm B}_{ n}|\leq K)
	\leq \exp(-2\mathfrak{g}_d n).
\end{align}
\end{lem}

\noindent \underline{Step 1} (Reaching the slices). Considering slices $(\bL_{M,\ell}^\delta(n))_{\ell \in \Z}$ defined in \eqref{slill}, we show  that these slices are infected shortly after breeding the seeds:
\begin{lem} \label{lem:t1}
	Let $s, \e >0$, $M \in \N$ and $\delta \in (0,\e/3]$.
	There exists $K=K( s,\e,M,\delta)\in \N$  such that for any $L \in \N$, if $n$ is sufficiently large, then
	\begin{align*}
		\pp\biggl( \exists \, A \subset \Z^d: A \textrm{ is connected},\,   0\in A,\, |A| \geq K ,\, \max_{\substack{\ell \in \Z\\1 \leq |\ell| \leq L}}  {\rm T}(A, \bL_{M,\ell}^\delta(n)) \geq \e n \biggr) \leq \exp(-s n),
	\end{align*}
	where we define the slices $\bL_{M,\ell}^\delta(n)$ by
	\begin{align}\label{slill}
		\bL_{M,\ell}^\delta(n):=[-\delta n, (1+\delta)n]\times \{ 2\ell {\lfloor M\sqrt{n}\rfloor} \}\times [-\delta n, \delta n]^{d-2} \cap \Z^d.
	\end{align}
\end{lem}
  
\medskip
\noindent \underline{Step 2} (Invading the target box).
We show that  starting from a slice, the infected set becomes dense in the  box $\Lambda(n\eb_1,M\sqrt{n})$ within $(\mu+o(1))n$ steps. 

 Given $\e,\delta>0$ and $M, L,n \in \N$, we define 
\begin{align}\label{lmne}
	\Lambda^{\e,\delta}_{M,L}(n)
	:= \biggl\{ x\in  \Lambda(n \eb_1, M\sqrt{n}) \setminus \{ n\eb_1 \}: \min_{\substack{\ell \in \Z\\1 \leq |\ell| \leq L}} \, \max_{y\in\bL_{M,\ell}^\delta (n) }{\rm T}_{\overline{\Lambda}_M(n)}(y,x)<(\mu+\e)n \biggr\},
\end{align}
where $\overline{\Lambda}_M(n)=\Lambda(M\sqrt{n})^c \cap \Lambda(n \eb_1, M\sqrt{n})^c \cap \Z^d$.
\begin{lem}  \label{lem:pec}
	Let $s,\e>0$.
There exist $M_0 \in \N$ and $\delta_0>0$ such that the following  holds:
	for all $M \geq M_0$ and $\delta \in (0,\delta_0]$, there exists $L_0=L_0(s,\e,M,\delta)\in \N$ such that if $L\geq L_0$  and $n$ is large enough, then
	\begin{align*}
		\pp\bigl( |\Lambda^{\e,\delta}_{M,L}(n)|<(1-\delta) |\Lambda(n\eb_1,M\sqrt{n})| \bigr) \leq \exp(- s n).
	\end{align*}
\end{lem}

\medskip
\noindent \underline{Step 3} (Hitting the endpoint). Once the infection is dense in the box $\Lambda(n\eb_1,M\sqrt{n})$, the target vertex $n\eb_1$ is reached quickly.  
The proof of Step 3 is presented below, while the proofs of Steps 0, 1, and 2 are given in  subsections \ref{subsec:breed}, \ref{subsect:lem4.3} and \ref{subsec:invad} respectively.

\subsection{Hitting the endpoint: Proof of \eqref{uuld} assuming Lemmas \ref{kvth}--\ref{lem:pec}}
\begin{proof}[ Proof of \eqref{uuld}]
	Fix $\xi>0$. Let $\e>0$ be sufficiently small and define
	\begin{align*}
		&{\rm T}_1:=\max_{\substack{\ell \in \Z\\1 \leq |\ell| \leq L}} {\rm T}(0, \bL_{M,\ell}^\delta(n))=\max_{\substack{\ell \in \Z\\1 \leq |\ell| \leq L}}\min_{y\in \bL_{M,\ell}^\delta(n)} {\rm T}(0,y),\\
		&{\rm T}_2:=\min\bigl\{ t(x,n\eb_1):x \in \Lambda^{\e,\delta}_{M,L}(n) \bigr\},
	\end{align*}
	where $\Lambda^{\e,\delta}_{M,L}(n)$ is defined in \eqref{lmne}. 
	Moreover, set
	\begin{align}
		\kE_{\mathrm{dense}}:=\bigl\{ |\Lambda^{\e,\delta}_{M,L}(n)| \geq (1-\delta) |\Lambda(n\eb_1,M\sqrt{n})| \bigr\}.\label{eq: dense set event}
	\end{align}
	By Lemma \ref{lem:pec}, we can find $\delta_0, M_0$ such that for all $M\geq M_0$ and $\delta \leq \delta_0$, there exists $L_0=L_0(\mathfrak{g}_d\xi, \e,M,\delta)$ such that for all $L\geq L_0$ and $n\in \N$ large enough,
	\ben{ \label{upec}
		\pp (\kE_{\mathrm{dense}}^c) \leq \exp(- \mathfrak{g}_d \xi n).
	}
If $\kE_{\mathrm{dense}}$ occurs then  $\Lambda^{\e,\delta}_{M,L}(n) \neq \varnothing$, and thus  ${\rm T}(0,n\eb_1) \leq {\rm T}_1+(\mu+\e)n+{\rm T}_2$ by the triangle inequality. 
	Hence,  
	\begin{align}\label{eq:pf_4.11}
		\begin{split}
			& \pp({\rm T}(0,n\eb_1) \geq (\mu+\xi)n)\\
			&\leq \pp(\{ {\rm T}_1+(\mu+\e)n+{\rm T}_2\geq (\mu+\xi)n\} \cap\kE_{\mathrm{dense}})+\pp(\kE_{\mathrm{dense}}^c)\\
			&\leq \sum_{k \in \Iintv{2\e n,(\xi-2\e)n}} \pp(\{ {\rm T}_1 = k,\,{\rm T}_2\geq (\xi-\e)n-k \} \cap\kE_{\mathrm{dense}})\\
			&\quad +\pp(\{ {\rm T}_1 \geq (\xi-3\e)n \} \cap\kE_{\mathrm{dense}})+\pp(\{ {\rm T}_2 \geq (\xi-3\e)n \} \cap\kE_{\mathrm{dense}})+\pp(\kE_{\mathrm{dense}}^c).
		\end{split}
	\end{align}
	Let us estimate the first term of the right-hand side in \eqref{eq:pf_4.11}. Let $K$ be a large constant chosen later. 	
For all $n$ sufficiently large and $k>2\e n$,
	\begin{align}\label{pt12ke}
		\begin{split}
			&\pp( \{ {\rm T}_1= k,\,{\rm T}_2\geq (\xi-\e)n-k \} \cap \kE_{\mathrm{dense}})\\
			&\leq \pp( \{ |\mathrm{B}_{k-\e n}|<K,\,{\rm T}_2\geq (\xi-\e)n-k \} \cap \kE_{\mathrm{dense}})+\pp(|\mathrm{B}_{k-\e n}|\geq K,\,{\rm T}_1 \geq k)\\
			&= \pp(|\mathrm{B}_{k-\e n}|<K)\,\pp( \{ {\rm T}_2\geq (\xi-\e)n-k \} \cap \kE_{\mathrm{dense}})+\pp(|\mathrm{B}_{k-\e n}| \geq K,\,{\rm T}_1\geq k),
		\end{split}
	\end{align}
where, in the last step, we have used the fact that $\{ |\mathrm{B}_{k-\e n}|<K \}$ depends only on the frogs inside $\Lambda(K)$, and neither $\kE_{\mathrm{dense}}$ nor $\T_2$  uses any frogs in $\Lambda(K)$ when $n$ is large enough.
	Furthermore, Lemma~\ref{Step73d} implies that given $\e>0$, we can take $\delta>0$ and $M \in \N$ such that for all $s \geq \e n$,
	\begin{align} \label{txsgd}
		\sup_{A \in \kA^{\eb_1}_{M,\delta,n}} \pp \left(\min_{x \in A} t(x,n\eb_1) \geq s \right) \leq \exp(-(\mathfrak{g}_d-\varepsilon)s),
	\end{align}
	where
	\begin{align*}
		\kA^{\eb_1}_{M,\delta,n}:=\bigl\{ A \subset \Lambda(n\eb_1,M\sqrt{n}) \setminus\{n \eb_1\}: |A| \geq (1-\delta) |\Lambda(n\eb_1,M\sqrt{n})| \bigr\}.
	\end{align*}
Note also that $\kE_{\mathrm{dense}}$ occurs if and only if   $\Lambda^{\varepsilon,\delta}_{M,L}(n) \in \kA^{\eb_1}_{M,\delta,n}$. Moreover,  $\Lambda^{\varepsilon,\delta}_{M,L}(n)$ depends only on frogs outside $\Lambda(n \eb_1, M\sqrt{n})$. Therefore, for all $n$ large enough and  $s \geq \e n$, \begin{align*}
		\pp( \{ \T_2 \geq s \} \cap\kE_{\mathrm{dense}})
		&= \sum_{A \in \kA^{\eb_1}_{M,\delta,n}} \pp \Bigl( \min_{x \in A} t(x,n\eb_1) \geq s,\, \Lambda^{\varepsilon,\delta}_{M,L}(n)=A \Bigr) \\
		&= \sum_{A \in \kA^{\eb_1}_{M,\delta,n}} \pp \Bigl( \min_{x \in A} t(x,n\eb_1) \geq s \Bigr) \,\pp\bigl( \Lambda^{\varepsilon,\delta}_{M,L}(n)=A \bigr)\\
		& \leq \sup_{A \in \kA^{\eb_1}_{M,\delta,n}} \pp \Bigl( \min_{x \in A} t(x,n\eb_1) \geq s \bigr) \leq \exp(-(\mathfrak{g}_d-\varepsilon)s),
	\end{align*}
using \eqref{txsgd}. 
	In particular,  for all positive integers $k \leq (\xi-2\eps)n$,
	\begin{align}\label{T2E}
		\pp( \{ {\rm T}_2\geq (\xi-\e)n-k \} \cap\kE_{\mathrm{dense}}) \leq \exp(-(\mathfrak{g}_d-\e)((\xi-\e) n- k)).
	\end{align}
	On the other hand, by Lemma~\ref{kvth}, for all $n\in \N$ large enough and $k\geq 2 \e n$,
	\ben{ \label{bkkn}
		\pp(|\mathrm{B}_{k-\e n}|<K) \leq \exp(-\kappa (1-\e) (k-\e n)).
	}
	Therefore, if $n$ is large enough, then
		\begin{align}\label{pt12e}
		\begin{split}
			&\sum_{k \in \Iintv{2\e n, (\xi-2\e)n}} \pp(|\mathrm{B}_{k-\e n}|<K) \,\pp( \{ \T_2 \geq (\xi-\e)n-k \} \cap\kE_{\mathrm{dense}})\\
			&\leq \sum_{k \in \Iintv{2\e n, (\xi-2\e)n}}  \exp(-(\mathfrak{g}_d-\e)((\xi-\e) n- k)-\kappa (1-\e) (k-\e n)) \\
			&\leq  \exp(-\mathfrak{g}_d \xi n +\e (\xi+\kappa+\mathfrak{g}_d)n) \sum_{k \in \Iintv{2\e n, (\xi-2\e)n}}  \exp (k(\mathfrak{g}_d  +\kappa\e-\kappa ))\\
			&\leq \exp(-\mathfrak{g}_d \xi n +\e (\xi+\kappa+\mathfrak{g}_d)n),
\end{split}
\end{align}
using  $\mathfrak{g}_d < \kappa/2$, by Lemma \ref{kvth}.
	Next, given $\delta, M, L$ satisfying  \eqref{upec}, \eqref{txsgd} and $\delta \leq \e/4$, by Lemma \ref{lem:t1} we can find $K$, such that for all $k\geq 2 \e n$ and $n\in \N$ large enough, 
	\begin{align}\label{eq: T1 estimate}
			\pp \left( {\rm T}_1\geq k, \,  |\mathrm{B}_{k-\e n}|\geq K \right)			&\leq
			\pp \left(\exists \, A \in {\rm Conn}_{\geq  K}(0), \,\max_{\ell: 1 \leq |\ell| \leq L}  {\rm T}(A, \bL_{M,\ell}^\delta(n) ) \geq \e n  \right)\notag \\
			&\leq \exp(-\mathfrak{g}_d \xi n),
	\end{align}
where ${\rm Conn}_{\geq  K}(0)$ denotes the set of all connected subsets of $\Z^d$  containing $0$ whose cardinality is larger than $K$. 	Combining this with \eqref{pt12ke} and \eqref{pt12e} yields that
	\ben{ \label{spt1t2e}
		\sum_{k \in \Iintv{2\e n, (\xi-2\e)n}}\pp(\{{\rm T}_1\geq k,\,{\rm T}_2\geq (\xi-\e)n-k \} \cap \kE_{\mathrm{dense}} ) \leq \exp(-\mathfrak{g}_d \xi(1-C \e) n),
	}
	where $C=C(\xi, \kappa)$ is a positive constant. Moreover, similarly to  \eqref{T2E},
	\aln{ 
    	\pp(\{{\rm T}_2\geq (\xi-3\e)n\} \cap\kE_{\mathrm{dense}} ) \leq \exp(-\mathfrak{g}_d (\xi-3\e)(1- \e) n).
	}
Using similar arguments for \eqref{bkkn} and \eqref{eq: T1 estimate}, we have 
\ben{ \label{pt2e}
\pp(\{{\rm T}_1\geq (\xi-3\e)n\} \cap\kE_{\mathrm{dense}} ) \leq \exp(-\mathfrak{g}_d (\xi-C \e) n).
} 
	Now, by decreasing $\e$ if needed, we can deduce \eqref{uuld} from 
	\eqref{upec},  \eqref{eq:pf_4.11}, \eqref{spt1t2e}-- \eqref{pt2e}.
\end{proof}
\subsection{Breeding seeds: Proof of Lemma \ref{kvth}} \label{subsec:breed}
By the definition of $\kappa$ and the shift invariance, given $\e>0$ and $K \in \N$, there exists $m=m(\e, K)\in\N$ such that for all $k\geq m$ and  for all connected sets $A$  with $|A|\leq K$,
\begin{align}\label{kmea}
	\max_{x \in A}\pp(\kR^x_k \subset A ) \leq \exp(-(\kappa -\e)k/|A|).
\end{align}
Suppose that $|{\rm B}_{ n}|\leq K$.
Let $t_0:=0$ and $t_1:= \inf\{t\in [t_0, n]:|{\rm B}_t\setminus {\rm B}_{t_0}|\geq 1\}$.
Successively, we define $t_{j+1}:= \inf\{t\in [t_j, n]:|{\rm B}_t\setminus {\rm B}_{t_{j}}|\geq 1\}$ for $j \in \N$, and set $\ell:=\inf\{j\in \N:t_{j+1}=\infty\}$.
For any $k \in \N$, $\mathfrak{A}_{n,K}^k$ is defined as the set of all sequences $\mathbf{(s,A)}=(s_j,A_j)_{j=0}^k$ on $\N_0 \times \Z^d$ satisfying that
\begin{itemize}
	\item $0=s_0<s_1<\dots<s_{k-1} \leq s_k=n$,
	\item $\{0 \} =A_0 \subsetneq A_1 \subsetneq\dots\subsetneq A_{k-1} \subset A_k$ and $|A_k| \leq K$,
	\item for all $j \in \{1,\dots,k\}$, $A_j$ is connected and $\max_{x \in A_j} \min_{y \in A_{j-1}}|x-y|_1=1$.
\end{itemize}
Note that the set $\mathfrak{A}_{n,K}^k$ includes all the realizations of $t_0,\dots,t_k$ and ${\rm B}_{t_0},\dots, \mathrm{B}_{t_k}$ given that $\ell=k$ and $|\rmB_n| \leq K$.
By successively applying the Markov property, we have for $k\leq K$,
\begin{align*}
	&\pp(|{\rm B}_{ n}|\leq K,\,\ell=k)\\
	&\leq \sum_{\mathbf{(s,A)} \in \mathfrak{A}_{n,K}^k} \pp (\mathrm{B}_t=A_j \, \forall \, t\in[s_j,s_{j+1}), \, j\leq k-1) \\
	&\leq \sum_{\mathbf{(s,A)} \in \mathfrak{A}_{n,K}^k} \pp (\mathrm{B}_t=A_j \, \forall \, t\in[s_j,s_{j+1}), \, j\leq k-2) \left(\max_{x \in A_{k-1}}\pp\left( \kR^{x}_{s_{k}-s_{k-1}-1} \subset A_{k-1}\right)\right)^{|A_{k-1}|}\\
	&\leq \sum_{\mathbf{(s,A)} \in \mathfrak{A}_{n,K}^k} \, \,  \prod_{j=1}^k \left(\max_{x \in A_{j-1}}\pp\left( \kR^{x}_{s_{j}-s_{j-1}-1} \subset A_{j-1}\right)\right)^{|A_{j-1}|}.
\end{align*}
Fix $m \in \N$ such that \eqref{kmea} holds. For $\mathbf{(s,A)} \in \mathfrak{A}_{n,K}^k$, we set  $\Gamma\mathbf{(s,A)} :=\{j \in [k]:s_j -s_{j-1} -1 \geq m \}$.
Then, observe that for all $\mathbf{(s,A)}=(s_j,A_j)_{j=0}^k \in \mathfrak{A}_{n,K}^k$,
\begin{align*}
	\sum_{j \in \Gamma\mathbf{(s,A)}} ({s_j-s_{j-1} -1}) = n-k - \sum_{j \in [k] \setminus \Gamma\mathbf{(s,A)}} ({s_j-s_{j-1}-1}) \geq  n - mk-k.
\end{align*}
Hence, by \eqref{kmea}, one has for all $\mathbf{(s,A)}=(s_i,A_i)_{i=0}^k \in \mathfrak{A}_{n,K}^k$,
\begin{align*}
	&\prod_{j \in \Gamma\mathbf{(s,A)}} \left(\max_{x \in A_{j-1}}\pp\left( \kR^{x}_{s_{j}-s_{j-1}-1} \subset A_{j-1}\right)\right)^{|A_{j-1}|}\\
	&\leq  \exp \left( -(\kappa-\e) \sum_{j \in \Gamma\mathbf{(s,A)}}({s_j-s_{j-1}-1})  \right)
	\leq  \exp \left( -(\kappa-\e) ( n- km-k) \right).
\end{align*}
Therefore, since $|\mathfrak{A}_{n,K}^k| \leq 2^{(2K+1)^d K} n^{K}$, for all $n$ sufficiently large
\begin{align*}
	\pp(| {\rm B}_{ n}|\leq K) = \sum_{k=1}^K \pp(| {\rm B}_{ n}|\leq K, \ell=k) &\leq C n^{K} \exp \left( -(\kappa-\e) ( n- Km-K) \right) \\
	& \leq \exp \left( -(\kappa-2\e) n  \right)
\end{align*}
where $C=C(K)$ is a positive constant. Thus \eqref{eq:kappa} follows. 
Let us next prove that $\kappa >2 \mathfrak{g}_d$. 
Denote the local time of the simple random walk on set $A$ up to time $n$ by 
\begin{align*}
	l_n(A) = \sum_{i=0}^n \mathbf{1}(S_i^0\in A).
\end{align*}
Suppose that $\kR_n \subset A$.
Then, since $l_n(A)= n+1$, we have
\begin{align*}
	\sum_{\substack{x,y\in A\\x \not= y}}l_n(\{x,y\})
	= \sum_{\substack{x,y\in A\\x \not= y}}(l_n(\{x\})+l_n(\{y\}))
	= 2(|A|-1)l_n(A)
	\geq  2(|A|-1)\,n.
\end{align*}
Since $|\{ (x,y):x,y \in A,\,x \not= y \}|=|A|(|A|-1)$, this implies that $l_n(\{x,y\}) \geq 2n/|A|$ must hold for some $x,y \in A$ with $x \not=y$.
Therefore, by the union bound,
\begin{align}\label{loe1}
	\begin{split}
		\pp(\kR_n \subset A)
		&\leq \pp \left( \exists x,y \in A \text{ with } x \neq y \text{ such that } l_n(\{x,y\})  \geq 2n/|A| \right)\\
        & \leq \sum_{x \in A} \sum_{ y \in A \setminus \{x\}} \pp \left(  l_n(\{x,y\})  \geq 2n/|A| \right) \\
		& \leq \sum_{x \in A} \sum_{ z \in (A-x) \setminus \{0\}}\pp (l_{\infty}(\{0,z\}) \geq 2n/|A|),
	\end{split}
\end{align}
where $A-x=\{y-x: y \in A\}$.
To estimate the summand above, for any $x,y \in \Z^d$, let $\tau^y_x = \inf\{t\geq 0: S^y_t =x  \}$ and write $\tau_x:=\tau_x^0$ simply.
Moreover, define $G_d(z):=\E[l_{\infty}(\{z\})]$ for $z \in \Z^d$, which is the  Green function for the simple random walk. 
From \cite[Lemma~{2.1}]{Cetal}, there exists a constant $C>0$ such that for all $m\geq 1$ and $z \in \Z^d$,
\begin{align}\label{loela}
	\pp(l_{\infty}(\{0,z\}) \geq m) \leq C \left(\frac{\lambda -1}{\lambda}\right)^m,
\end{align}
where $\lambda:=G_d(0)+G_d(z)$ is the largest eigenvalue of the matrix
\begin{align*}
	\begin{pmatrix}
		G_d(0) & G_d(z)\\
		G_d(z) & G_d(0)
	\end{pmatrix}.
\end{align*}
Letting $\tau'_z = \inf \{t\geq 0: |S_t-z|_1=1  \}$, we use the strong Markov property to obtain for any $z \in \Z^d \setminus \{0\}$,
\begin{align*}
	\pp(\tau_z<\infty)
	= \E\left[ \mathbf{1}\left\{\tau^{S_{\tau'_z}}_z < \infty\right\} \mathbf{1}\{\tau'_z < \infty\} \right]
	&=\E\left[ \pp\left(\tau^{S_{\tau'_z}}_z < \infty \mid \tau'_z < \infty \right) \mathbf{1}\{\tau'_z < \infty\}\right]\\
	&= \pp(\tau_{\eb_1} < \infty) \pp(\tau'_z < \infty) \leq \pp(\tau_{\eb_1} < \infty).
\end{align*}
This, together with the fact that $G_d(z)=\pp(\tau_z<\infty) G_d(0)$, implies $G_d(z) \leq G_d(\eb_1)=G_d(0)-1$, and hence
\begin{align*}
	\lambda =  G_d(0)+G_d(z)  \leq 2 G_d(0)-1.
\end{align*}
Therefore, combining \eqref{loe1} and \eqref{loela}, we have for any finite connected set $A$,
\begin{align*}
	\pp(\kR_n \subset A) \leq C|A| ^2\exp\left(- \log \left(\frac{2G_d(0)-1}{2G_d(0)-2} \right) \times \frac{2n}{|A|} \right).
\end{align*}
Hence, from the definition of $\kappa$ in \eqref{eq:kappa}, we have that 
\begin{align*}
	\kappa \geq 2\log \left( \frac{2G_d(0)-1}{2G_d(0)-2} \right).
\end{align*}
To conclude that $\kappa >2 \mathfrak{g}_d$,  it remains to prove
\begin{align}\label{lokap}
	\log \left( \frac{2G_d(0)-1}{2G_d(0)-2} \right)>\mathfrak{g}_d,
\end{align}
where we recall that $\mathfrak{g}_d = \pp(S_i \neq 0 \, \forall \, i \geq 1 )$. For this, consider the function $f(t):=t \log (\tfrac{2t-1}{2t-2})-1$ for $t>1$. Using the inequality $\log (1+x) <x$,
\begin{align*}
	f'(t) = \log (1+\tfrac{1}{2t-2}) + \tfrac{2t}{2t-1} - \tfrac{2t}{2t-2} < \tfrac{1}{2t-2} + \tfrac{2t}{2t-1} - \tfrac{2t}{2t-2}=\tfrac{-1}{(2t-1)(2t-2)}<0.
\end{align*}
Hence, $f(t)$ is non-increasing in $t>1$.
From \cite[Lemma~1 and the argument below the lemma]{OS}, the return probability $\pi_d:=\pp(\exists i \geq 1: S_i=0)$ satisfies that $\pi_d \leq \pi_3<0.341$ for all $d\geq 3$.
This means that
\begin{align*}
	G_d(0) =1/(1-\pi_d)<1/(1-0.341)<1.52.
\end{align*}
Hence, since $f$ is non-increasing,
\begin{align*}
	G_d(0)\log\left( \frac{2G_d(0)-1}{2G_d(0)-2} \right)-1=f(G_d(0)) \geq f(1.52)>0.02>0,
\end{align*}
which implies \eqref{lokap} as $\mathfrak{g}_d=1/G_d(0)$.\qed
\subsection{Reaching the slices: Proof of Lemma~\ref{lem:t1}}\label{subsect:lem4.3} 
Throughout this subsection, we fix $s,\e>0$, $L,M \in \N$ and $\delta \in(0,\e/3]$. First, recall
\[ \bL_{M,\ell}^\delta(n)=[-\delta n, (1+\delta)n]\times \{ 2\ell {\lfloor M\sqrt{n}\rfloor} \}\times [-\delta n, \delta n]^{d-2} \cap \Z^d.  \]
Let $A$ be a connected set containing $0$.
Then, for any $R \geq 3$,
\begin{align}\label{taiii}
	\begin{split}
		&\pp\left( \max_{\ell: 1 \leq |\ell| \leq L}{\rm T}(A,\bL_{M,\ell}^\delta(n))\geq \e n \right)\\
		&\leq \pp\left( \max_{\ell: 1 \leq |\ell| \leq L}{\rm T}(A,\bL_{M,\ell}^\delta(n))\geq 3 \delta n; \, \mathrm{B}^A_{\delta n} \not \subset \Lambda(3R \sqrt{n})\right) + \pp(\mathrm{B}^A_{\delta n} \subset \Lambda(3R \sqrt{n})).
	\end{split}
\end{align}
By Proposition \ref{tln} (ii), there exists $K=K(s,\delta,R)\in \N$ such that for any connected set $A$ containing $0$ with $|A| \geq K$,
\begin{align}\label{eq:ball}
	\pp(\mathrm{B}^A_{\delta n} \subset \Lambda(3R \sqrt{n})) \leq \exp(-2sn).
\end{align}
A vertex $x$ is called \textbf{bad} if $\max_{\ell: 1 \leq |\ell| \leq L} {\rm T}_{\Lambda(x, \sqrt{n})} (x, \bL_{M,\ell}^\delta(n)) \geq 2\delta n$. We claim that 
\begin{align}\label{xibad}
	\begin{split}
		&\pp\Bigl( \max_{\ell: 1 \leq |\ell| \leq L}{\rm T}(A,\bL_{M,\ell}^\delta(n))\geq 3 \delta n; \, \mathrm{B}^A_{\delta n} \not \subset \Lambda(3R \sqrt{n}) \Bigr)\\
		& \leq \pp\left(
		\begin{minipage}{18em}
			$\exists \, (x_i)_{i\in[R]} \subset  \Lambda(3R\sqrt{n})$: $(x_i)_{i\in[R]}$ are bad;  $|x_i-x_j|_{\infty} \geq 3\sqrt{n}$, $1 \leq i \not= j \leq R$
		\end{minipage}
		\right)
	\end{split}
\end{align}
Indeed,  if 
$\mathrm{B}^A_{\delta n} \not\subset \Lambda(3R\sqrt{n})$,  there exist  vertices $x_1,\cdots,x_R \in \Lambda(3R \sqrt{n})\cap \mathrm{B}^A_{\delta n}$ such that $|x_i-x_j|_{\infty} \geq 3\sqrt{n}$ for all distinct $i,j \in [R]$ (this is possible since $\mathrm{B}^A_{\delta n}$ is connected). Assume  further that $\max_{1 \leq |\ell| \leq L} {\rm T} (A, \bL_{M,\ell}^\delta(n)) \geq 3\delta n$. Then,  for all $i \in [R]$,
\begin{align*}
	\max_{\ell: 1 \leq |\ell| \leq L}  {\rm T}_{\Lambda(x_i,\sqrt{n})}(x_i, \bL_{M,\ell}^\delta(n)) &\geq \max_{\ell: 1 \leq |\ell| \leq L}{\rm T}(x_i, \bL_{M,\ell}^\delta(n)) \\
	&\geq  \max_{\ell:1 \leq |\ell| \leq L} {\rm T}(A, \bL_{M,\ell}^\delta(n)) - {\rm T}(A,x_i) \geq  2\delta n,
\end{align*}
since   ${\rm T}(A,x_i)\leq \delta n$ (as $x_i\in \mathrm{B}^A_{\delta n}$). Hence,   $x_i$ is bad for all $i \in [R]$ and \eqref{xibad} holds.

 We next estimate the probability that a given vertex is bad. By Proposition~\ref{tln}~(i), there exists a positive constant $c=c(\delta)$ such that for all $x \in \Z^d$
 \begin{align*}
 	\pp(x \textrm{ is bad}) & \leq
 	\pp\left( x \textrm{ is bad}, \, {\rm T}_{\Lambda(x,\sqrt{n})}(x,y) \leq \delta n  \, \, \forall \, y \in \Lambda(x,\sqrt{n}) \right)  \notag \\
 	& \hspace{1 cm}+ \pp\bigl( \exists \, y \in \Lambda(x,\sqrt{n}): {\rm T}_{\Lambda(x,\sqrt{n})}(x,y)>\delta n \bigr) \notag \\
 	& \leq 	\pp\left( x \textrm{ is bad}, \, {\rm T}_{\Lambda(x,\sqrt{n})}(x,y) \leq \delta n  \, \, \forall \, y \in \Lambda(x,\sqrt{n}) \right)  + \exp(-2cn). 
 \end{align*}
  Suppose that  $x$ is bad and ${\rm T}_{\Lambda(x,\sqrt{n})}(x,y) \leq \delta n$ for all $y \in \Lambda(x,\sqrt{n})$. Then for all $ \ell \in \Z: 1 \leq |\ell| \leq L$
\be{
t(y,\bL_{M,\ell}^\delta(n))\geq {\rm T}_{\Lambda(x,\sqrt{n})}(x,\bL_{M,\ell}^\delta(n)) - {\rm T}_{\Lambda(x,\sqrt{n})}(x,y)  \geq \delta n.
}
Therefore, combining the last two display equations, we have 
\begin{align}
	\pp(x \textrm{ is bad}) & 
	\leq \pp \left( \max_{\ell: 1 \leq |\ell| \leq L}t(y,\bL_{M,\ell}^\delta(n))\geq  \delta n \,\, \forall \, y \in \Lambda(x,\sqrt{n}) \right) + \exp(-2cn). \label{pxbad}
\end{align}
Moreover, if $x \in \Lambda(3R\sqrt{n})$, then
\aln{
	&\pp\left( \max_{\ell: 1 \leq |\ell| \leq L}t(y,\bL_{M,\ell}^\delta(n)) \geq  \delta n \,\, \forall \, y \in \Lambda(x,\sqrt{n}) \right)\notag\\
	 & \leq \left(\max_{z \in \Lambda(4R\sqrt{n})} \pp \left( \max_{\ell: 1 \leq |\ell| \leq L}t(z,\bL_{M,\ell}^\delta(n))\geq  \delta n \right) \right)^{|\Lambda(\sqrt{n})|} \notag \\
	&= \left(1-\min_{z \in \Lambda(4R\sqrt{n})} \pp \left( \max_{\ell: 1 \leq |\ell| \leq L}t(z,\bL_{M,\ell}^\delta(n))\leq  \delta n \right) \right)^{|\Lambda(\sqrt{n})|}. \label{pyll}
}
We claim that for any $R,L,M$ and $\delta$ satisfying \eqref{upec}, there exists a positive constant $\theta=\theta(R,L,M,\delta)$ such that
\ben{ \label{pdp}
	\liminf_{n\rightarrow \infty}\min_{z \in \Lambda(4R\sqrt{n})} \pp \left( \max_{\ell: 1 \leq |\ell| \leq L}t(z,\bL_{M,\ell}^\delta(n))\leq  \delta n \right) \geq \theta.
}
Indeed, let $Q=2LM+4R$  and denote the upper and lower faces  along the second coordinate of the box $\Lambda(\lf Q\sqrt{n}\rf)$ by
\aln{\label{def: Fpm}
F_\pm = [-\lfloor Q\sqrt{n}\rfloor,\lf Q\sqrt{n}\rf] \times \{ \pm \lf Q\sqrt{n}\rf\} \times [-\lf Q\sqrt{n}\rf,\lf Q\sqrt{n}\rf]^{d-2}.
}
Observe that for any $z \in \Lambda(\lf Q\sqrt{n}\rf)$ if $\kR^z_{\delta n/2}$ touches both the faces $F_\pm$ then $\kR^z_{\delta n/2}$ intersects all the slices $\bL_{M,\ell}^\delta(n)$ with $|\ell| \leq L$. Hence, for all $n$ sufficiently large\footnote{For the general-direction analogue, see Lemma~\ref{lem:gen-slices}.}
\aln{\label{eq: hit L}
	&\min_{z \in \Lambda(4R\sqrt{n})} \pp \left( \max_{\ell: 1 \leq |\ell| \leq L}t(z,\bL_{M,\ell}^\delta(n))\leq  \delta n \right)\\
	& \geq \min_{z \in \Lambda(\lceil Q\sqrt{n}\rceil)} \pp \left( \kR^z_{\delta n/2} \cap F_ -\neq \varnothing, \kR^z_{\delta n/2} \cap F_+ \neq \varnothing \right) \geq \theta,\notag
}
    where $\theta=\theta(Q,\delta)$ is a  positive constant, by  Lemma  \ref{kr} (iii). Hence, \eqref{pdp} holds, and together with \eqref{pyll}, this implies
\be{
	\pp\left( \max_{\ell: 1 \leq |\ell| \leq L}t(y,\bL_{M,\ell}^\delta(n)) \geq  \delta n \,\, \forall \, y \in \Lambda(x,\sqrt{n}) \right) \leq \exp(-an^{d/2}),
}
where $a=a(R,L,M,\delta)$ is a positive constant. Combining the above inequality and  \eqref{pxbad}, we arrive at
\be{
	\pp(x \textrm{ is bad})  \leq \exp(-a n^{d/2})+ \exp(-2c n)\leq \exp(-cn).
}
In addition, if $|x-y|_\infty \geq 3\sqrt{n}$, then the events $\{x \textrm{ is bad} \}$ and $\{y \textrm{ is bad} \}$  are independent.
Thus, letting $R:=\max\{3,\lceil 3s/c \rceil \}$, we have
\begin{align*}
	\pp\left(
	\begin{minipage}{18em}
		$\exists \, (x_i)_{i\in[R]} \subset  \Lambda(3R\sqrt{n})$: $(x_i)_{i\in[R]}$ are bad;  $|x_i-x_j| \geq 3\sqrt{n}$, $1 \leq i \not= j \leq R$
	\end{minipage}
	\right)
	&\leq |\Lambda(3R\sqrt{n})|^R \exp(-cRn)\\
	& \leq \exp(-2sn).
\end{align*}
Combining this with \eqref{taiii}, \eqref{eq:ball} and \eqref{xibad}, we obtain for any connected set $A$ containing $0$ with $|A|\geq K$,
\begin{align} \label{mtaln}
	\pp\left( \max_{\ell: 1 \leq |\ell| \leq L}{\rm T}(A,\bL_{M,\ell}^\delta(n))\geq \e n \right) \leq 2 \exp(-2sn).
\end{align}
Finally, by using the above estimate and the union bound, it follows that for all $n$ sufficiently large,
\begin{align*}
	&\pp \left(\exists \, A \in {\rm Conn}_{\geq  K}(0), \, \max_{\ell:  1 \leq |\ell| \leq L}  {\rm T}(A, \bL_{M,\ell}^\delta(n)) \geq \e n \right) \\
	& \leq  \pp \left(\exists \, A \in {\rm Conn}_{K}(0), \, \max_{\ell:  1 \leq |\ell| \leq L}  {\rm T}(A, \bL_{M,\ell}^\delta(n)) \geq \e n \right)\\
	&\leq \sum_{ A \in {\rm Conn}_{K}(0)} \pp\left( \max_{\ell: 1 \leq |\ell| \leq L}{\rm T}(A,\bL_{M,\ell}^\delta(n))\geq \e n \right)\leq 2(2d)^{2K} \exp(-2sn)
	\leq \exp(-sn),
\end{align*}
here recall that ${\rm Conn}_{\geq  K}(0)$ (resp. ${\rm Conn}_{ K}(0)$) denotes the set of all connected sets containing $0$ whose cardinality is at least (resp. equal to)  $K$, and note that $|{\rm Conn}_{ K}(0)| \leq (2d)^{2K}$. \qed
\subsection{Invading the target box: Proof of Lemma~\ref{lem:pec}} \label{subsec:invad}
We first prepare  some lemmas.
\begin{lem} \label{lem:tbyl}
	Given $M,L,\e,\delta$, there exists $c=c(M,L,\e,\delta)>0$ such that for all sufficiently  large $n\in \N$,
	\begin{align*}
		\pp \bigl( \mathcal{V}_{M,L}^{\e,\delta}(n) \bigr) \geq 1-\exp(-cn^{d/2}),
	\end{align*}
	where $\mathcal{V}_{M,L}^{\e,\delta}(n)$ is the event that for all $\ell \in \Z$ with $1 \leq |\ell| \leq L$,
	\begin{align*}
		\Bigl| \Bigl\{ x \in \Lambda(n\eb_1,M\sqrt{n}) \setminus \{n\eb_1\}: \min_{ z \in \Lambda(y_{\ell},\sqrt{n})} \, t(z,x) \leq \tfrac{\e n}{2} \Bigr\} \Bigr| \geq (1-\delta) |\Lambda(n\eb_1,M\sqrt{n})|,
	\end{align*}
and for $\ell \in \Z$ \[ y_\ell:=n\mathbf{e}_1+2\ell \lf M\sqrt{n} \rf\mathbf{e}_2 \in \bL_{M,\ell}^\delta(n).
\]
\end{lem}
\begin{proof}
	Fix $\ell \in \Z$ with $1 \leq |\ell| \leq L$.
	For simplicity, we write $A=\Lambda(y_\ell,\sqrt{n})$ and $B=\Lambda(n\eb_1,M\sqrt{n})$.
	Due to Lemma~\ref{rnab}, there exists a constant $c>0$ such that for all $n$ sufficiently large,
	\begin{align*}
		&\pp \left(\Bigl| \{ x \in B: \min_{z \in A} \, t(z,x) \leq \tfrac{\e n}{2} \} \Bigr| \geq (1-\delta) | B|\right)\\
		&= \pp \left(| \kR_{\e n/2}^A \cap B| \geq \min\{ cn|A|,(1-\delta)|B| \} \right)
		\geq 1-\exp(-cn^{d/2}),
	\end{align*}
where we note that $|A|$ and $|B|$ have order $n^{d/2}$. 
	Since $\ell$ is arbitrary, we get the lemma by using the union bound.
\end{proof}
\begin{lem}\label{lem:tsl}
	Fix $s,\e>0$.
	There exist $M_0 \in \N$ and $\delta_0>0$ such that the following holds: for all $M\geq M_0$ and $\delta \in (0,\delta_0]$:
	there exists $L_0=L_0(s,\e,M,\delta)\in \N$ such that if $L\geq L_0$ and $n$ is large enough, then
	\begin{align*}
		\pp\bigl( \mathcal{W}_{M,L}^{\e,\delta}(n) \bigr) \geq 1- \exp(-sn),
	\end{align*}
	where 
	\begin{align*}
	\mathcal{W}_{M,L}^{\e,\delta}(n) = \Big \{ \min_{\ell: 1 \leq |\ell| \leq L}	\max_{y \in \bL_{M,\ell}^\delta(n)} \, \,\max_{z\in \Lambda(y_{\ell}, \sqrt{n})} \T_{\bS_{M,\ell}^\delta(n)} (y,z) \leq (\mu + \tfrac{\e}{2})n \Big\},
	\end{align*}
with  the slab $\bS_{M,\ell}^\delta(n)$ defined by 
	\begin{align}\label{bsll}
		[-2\delta n, (1+2\delta)n]\times [(2\ell-1) \lfloor M\sqrt{n} \rfloor +1, (2\ell+1) \lfloor M\sqrt{n} \rfloor -1 ]\times [-2\delta n, 2\delta n]^{d-2} \cap \Z^d,
	\end{align}
	and recall that  $y_\ell:=n\mathbf{e}_1+2\ell {\lfloor M\sqrt{n}\rfloor}\mathbf{e}_2 \in \bL_{M,\ell}^\delta(n)$.
\end{lem}
\vspace{1em}
Since the proof of Lemma~\ref{lem:tsl} is technical, we give it after the proof of Lemma~\ref{lem:pec}.
\begin{proof}[\underline{Proof of Lemma~\ref{lem:pec}}]
	Fix $s,\e>0$.
	Note that if $|\ell| \geq 1$, 
	\begin{align*}
		\bS_{M,\ell}^\delta(n) \subset \overline{\Lambda}_M(n)=\Lambda(M\sqrt{n})^c \cap \Lambda(n \eb_1, M\sqrt{n})^c,
	\end{align*}
    and thus for all $x, y, z 
    \in \Z^d$ and $|\ell| \geq 1$
    \ben{ \label{tlsxyz}
    {\rm T} _{\overline{\Lambda}_M(n)} (y,x)  \leq {\rm T} _{\bS_{M,\ell^*}^\delta(n)}(y,x)  \leq {\rm T} _{\bS_{M,\ell^*}^\delta(n)}(y,z)+t(z,x).
    }
 Suppose that  $ \mathcal{W}_{M,L}^{\e,\delta}(n)$ occurs. Then there exists  $\ell^* \in \Z$ with $1\leq |\ell^*| \leq L$ such that
\[
\max_{y \in \bL_{M,\ell^*}^\delta(n)}\max_{z \in \Lambda(y_{\ell^*},\sqrt{n})} {\rm T} _{\bS_{M,\ell^*}^\delta(n)}(y,z)\leq (\mu + \tfrac{\e}{2})n.
\]
Hence, using \eqref{tlsxyz} we have for all $x \in \Z^d$,
\al{
&\min_{1\leq |\ell| \leq L}\max_{y \in \bL_{M,\ell}^\delta(n)}{\rm T} _{\overline{\Lambda}_M(n)} (y,x)\\
&\leq \max_{y \in \bL_{M,\ell^*}^\delta(n)}\min_{z \in \Lambda(y_{\ell^*},\sqrt{n})} \Big[{\rm T} _{\bS_{M,\ell^*}^\delta(n)}(y,z)+t(z,x) \Big] \\
&\leq \max_{y \in \bL_{M,\ell^*}^\delta(n)}\max_{z \in \Lambda(y_{\ell^*},\sqrt{n})} {\rm T} _{\bS_{M,\ell^*}^\delta(n)}(y,z) + \min_{z \in \Lambda(y_{\ell^*},\sqrt{n})} t(z,x)\\
&\leq (\mu + \tfrac{\e}{2})n + \min_{z \in \Lambda(y_{\ell^*},\sqrt{n})} t(z,x).
}
Therefore, on $ \mathcal{W}_{M,L}^{\e,\delta}(n)$,
\al{
 \Lambda^{\e,\delta}_{M,L}(n)&:=\Bigl\{ x \in \Lambda(n\eb_1,M\sqrt{n}) \setminus \{n\eb_1\}: \min_{1\leq |\ell| \leq L}\max_{y \in \bL_{M,\ell}^\delta(n)}{\rm T} _{\overline{\Lambda}_M(n)} (y,x)\leq (\mu+\e)n \Bigr\} \\
 & \supseteq \Bigl\{ x \in \Lambda(n\eb_1,M\sqrt{n}) \setminus \{n\eb_1\}: \min_{ z \in \Lambda(y_{\ell^*},\sqrt{n})} \, t(z,x) \leq \tfrac{\e n}{2} \Bigr\}. 
}
Hence, by the definition of $\mathcal{V}_{M,L}^{\e,\delta}(n)$,
\be{
\{|\Lambda^{\e,\delta}_{M,L}(n)| \geq (1-\delta) |\Lambda(n\eb_1,M\sqrt{n})|\} \supseteq \mathcal{W}_{M,L}^{\e,\delta}(n) \cap  \mathcal{V}_{M,L}^{\e,\delta}(n).
}
Finally, the result follows from the above observation, and  
	 Lemmas~\ref{lem:tbyl} and \ref{lem:tsl}.
	 								\end{proof}
\vspace{1em}
We close this section with the proof of Lemma~\ref{lem:tsl}.
\begin{proof}[\underline{Proof of Lemma~\ref{lem:tsl}}]
	Fix $s,\e>0$ (and the dependence on $s,\e$ is omitted from the notation).
	Define for each $\ell \in \Iintv{-L,L}$,
	\begin{align*}
		&\kE_\ell^1(n):= \Big \{ \max_{y \in \bL_{M,\ell}^\delta(n)} \T_{\bS_{M,\ell}^\delta(n)} (y, y_\ell) \leq (\mu+\tfrac{\e}{4})n\Big \},\\
		&\kE_\ell^2(n):=\Big \{  \max_{z \in \Lambda(y_\ell, \sqrt{n})} \T_{\Lambda (y_{\ell},\lf  M\sqrt{n}/2\rf) }(y_\ell, z) \leq \tfrac{\e n}{4} \Big \}.
	\end{align*}
	Then, since $\Lambda(y_{\ell},\lf \tfrac{M}{2}\sqrt{n}\rf) \subset \bS_{M,\ell}^\delta(n)$,
	\begin{align*}
		\kE_\ell^1(n) \cap \kE_\ell^2(n) \subset \Big \{ \max_{y \in \bL_{M,\ell}^\delta(n)} \max_{z \in \Lambda(y_\ell, \sqrt{n})} \T_{\bS_{M,\ell}^\delta(n)} (y, z) \leq (\mu+\tfrac{\e}{2})n\Big \}.
	\end{align*}
	We notice that $(\kE_\ell^1(n) \cap \kE_\ell^2(n))_{|\ell| \leq L}$ are independent events, since $(\bS_{M,\ell}^\delta(n))_{1 \leq |\ell| \leq L}$ are disjoint sets.
	Moreover, these events have the same probability. Therefore,
	\begin{align*}
		\pp\bigl( \mathcal{W}_{M,L}^{\e,\delta}(n) \bigr)
		\geq \pp\biggl( \, \bigcup_{\ell: 1 \leq |\ell| \leq L} (\kE_\ell^1(n) \cap \kE_\ell^2(n)) \biggr)
		\geq 1- \pp\bigl( (\kE_0^1(n) \cap \kE_0^2(n))^c \bigr)^{2L}.
	\end{align*}
	Hence, to show Lemma~\ref{lem:tsl}, it suffices to prove that 
    there exists $c=c(\e,\delta,M)>0$ such that for $n\in \N$ large enough,
	\begin{align}\label{eell}
		\pp\bigl( \kE_0^1(n)^c)+\pp\bigl( \kE_0^2(n)^c \bigr) \leq \exp(-cn).
	\end{align}
	By Proposition~\ref{tln}~(i), there exists $c=c(M,\e)>0$ such that for all $n$ sufficiently large 
	\begin{align}\label{pe10c}
		\pp\bigl( \kE_0^2(n)^c \bigr)
		= \pp \left( \max_{z \in \Lambda(y_0, \sqrt{n})} \T_{\Lambda(y_{0},\lf  M\sqrt{n}/2\rf) }(y_0, z) > \tfrac{\e n}{4}  \right)
		\leq \exp(- cn).
	\end{align}
	Let us next estimate $\pp(\kE_0^1(n)^c)$.
	By the union bound,
	\begin{align}\label{pe10}
		\pp\bigl( \kE_0^1(n)^c \bigr)
		\leq \sum_{y \in \bL_{M,0}^\delta(n)}  \pp\left( \T_{\bS_{M,0}^\delta(n)} (y, y_0)>(\mu+\tfrac{\e}{4})n \right).
	\end{align}
	By Proposition~\ref{tuld}~(i),  for all sufficiently large $M$, there exists $c=c(M,\e)>0$ such that for all $k$ large enough,  $x \in \Lambda(k)$ and $k^2 \geq t \geq (1+\e^2) \mu(x)$,
	\ben{  \label{key3d1}
		\pp(\T_{\Lambda(Mk/2)} (0,x) \geq t) \leq \exp(-ct).
	}
	Given $y =(y^1, \ldots, y^d) \in \bL_{M,0}^\delta$, we define $y_*=y^1 \eb_1$. Then 
	\be{
		|y-y_*|_\infty \leq \delta n, \qquad y_* \in \Z \eb_1 \cap \bL_{M,0}^\delta(n).
	}
	There exists  a sequence of vertices $(u_i)_{i=1}^l \subset \bL_{M,0}^\delta(n)$ such that $u_1=y, u_l = y_*$ and $u_i \in \Lambda(u_{i-1}, \sqrt{n})$ for $2 \leq i \leq l$, and $\Lambda(u_{i}, M \sqrt{n}/2) \cap \Lambda(u_{j}, M\sqrt{n}/2) = \varnothing $ if $|i-j| \geq 4M$ and $l \leq 4d \delta\sqrt{n}$. We define $X_i:=\T_{\Lambda \left(u_{i},{M\sqrt{n}}/{2}\right)} (u_i,u_{i+1})$. Since $\Lambda(u,M\sqrt{n}/2) \subset \bS_{M,0}^\delta(n)$ for $u \in \bL_{M,0}^\delta$, using the triangle inequality
	\ben{ \label{tyysx}
		\T_{\bS_{M,0}^\delta(n)}(y,y_*)  \leq  \sum_{i=1}^{l-1} X_i.
	}
	By the properties of $(u_i)_{i=1}^l$ we have $X_i$ and $X_j$ are independent when $|i-j|\geq 4M$. Moreover, by \eqref{key3d1}, there exists a universal constant $C$ such that
	\be{
		\pp(X_i \geq t) \leq \exp(-t/C), \qquad \forall \, t \in [CM\sqrt{n}, n], \quad \forall \, i \in [l-1].
	}
	For each  $\e>0$,  taking  $\delta= \delta(\e)$ sufficiently small, we have $C  l M \sqrt{n} \leq \varepsilon^2 n/2$. Hence, using \eqref{tyysx} and applying Lemma \ref{lem:ldlc} (ii) to $(X_i)_{i=1}^{l-1}$, there exists  $c=c(\varepsilon,M)>0$ such that 
	\aln{ \label{tyys}
		\pp \left(\T_{\bS_{M,0}^\delta(n)}(y,y_*) \geq \varepsilon^2 n \right) &\leq \pp \left( \sum_{i=1}^{l-1} X_i \geq C  l M \sqrt{n}+ \varepsilon^2 n/2  \right)\notag\\
        &\leq \exp(-c  n).
	}
	We now consider the passage time from $y_*$ to $y_0$. Note that $y_0=n\eb_1 $ and $ y_* $ are in $\Z \eb_1 \cap \bL_{M,0}^\delta(n)$, so $|y_0-y_*| \leq (1+2 \delta)n$. Hence, there exists  $(v_i)_{i=1}^{l'} \subset \Z \eb_1 \cap \bL_{M,0}^\delta(n)$ such that $v_1=y_*, v_{l'}= y_0$ and $|v_i -v_{i-1}| =\lfloor \sqrt{n}\rfloor$ for $2 \leq i \leq l'$, and  $|v_{l'} -v_{l'-1}| \leq \lfloor \sqrt{n}\rfloor$  and  $l' \leq (1+ 3 \delta)\sqrt{n}$.  Using the same argument as above, we also have
	\ben{ \label{tsy0y}
		\T_{\bS_{M,0}^\delta(n)}(y_*,y_0)  \leq \sum_{i=1}^{l'-1} Y_i, \qquad Y_i:=\T_{\Lambda \left(v_{i},\lf M\sqrt{n}/2\rf \right)} (v_i,v_{i+1}),
	}
	and $Y_i$ and $Y_j$ are independent if $|i-j| \geq 4M$. By  \eqref{key3d1}, there exists a constant $c=c(M,\e) \in (0,1)$ such that for any $i\in [l'],$
	\be{
		\pp(Y_i \geq t) \leq \exp(-ct) \quad \forall \, t \in [(1+\varepsilon^2) \mu \sqrt{n},n].
	}
	If $\delta=\delta(\e)$ is sufficiently small  then $l' (1+\varepsilon^2) \mu \sqrt{n} \leq (1+2 \varepsilon^2)\mu n$. Thus using \eqref{tsy0y} and  Lemma \ref{lem:ldlc} (ii) yields 
    \al{
           		\pp \left(\T_{\bS_{M,0}^\delta(n)}(y_*,y_0) \geq (1+4 \e^2)\mu n \right)& \leq \pp \left(     \sum_{i=1}^{l'-1} Y_i \geq l' (1+\varepsilon^2) \mu \sqrt{n} + 2 \e^2\mu n \right) \\
        		&\leq \exp(-c\e^2 \mu  n/4).
                   	}
	Combining this estimate with \eqref{pe10} and \eqref{tyys}, and using the triangle inequality, we obtain
	\al{
		\pp\bigl( \kE_0^1(n)^c \bigr)
		&\leq \sum_{y \in \bL_{M,0}^\delta(n)}  \pp\left( \T_{\bS_{M,0}^\delta(n)} (y, y_0)>(\mu+\tfrac{\e}{4})n \right) \\
        &\leq n^d[ \exp(-c \e^2 n/16)+\exp(-c\mu \e^2 n/4)].
	}
	This together with \eqref{pe10c} implies the desired result \eqref{eell}.
\end{proof}
\subsection{General direction}\label{subsec: general direction upper d>2}
Let $v\in \R^d$ be an unit vector. We take the orthonormal basis $\mathbf{v} := (v_1,\ldots,v_d)$ such that $v_1=v$. To get the claim for the general direction $v$, we take $M $ and $L$ large integers, $\delta$ a small constant, and   
 replace the definition of slices 
$\bL_{M,\ell}^\delta(n)$  by
	\begin{align*}\label{slill2}
\bL_{M,\ell}^{\delta,{v}}(n):=
\Big\{ \Big \lfloor \sum_{i=1}^d x_i v_i\Big \rfloor\mid (x_i)_{i=1}^d\in [-\delta n, (1+\delta)n]\times \{ 2\ell {\lfloor M\sqrt{n}\rfloor} \}\times [-\delta n, \delta n]^{d-2} \Big \},
\end{align*}
and the slabs $\bS_{M,\ell}^{\delta}(n)$ by 
	\begin{align*}
&\bS_{M,\ell}^{\delta,{v}}(n):= \Big  \{	\Big \lfloor \sum_{i=1}^d x_i v_i \Big \rfloor \textrm{ such that}  \\
&  (x_i)_{i=1}^d\in 	[-2\delta n, (1+2\delta)n]\times [(2\ell-1) \lfloor M\sqrt{n} \rfloor +1 , (2\ell+1) \lfloor M\sqrt{n} \rfloor -1 ]\times [-2\delta n, 2\delta n]^{d-2} \Big \}.
	\end{align*}
In the proof outlined in Section~\ref{section: elements}, the same argument works except for Lemma~\ref{lem:t1}, in which we need to prove a version of \eqref{pdp} that is for all $R>0$
\ben{ \label{pdpv}
		\liminf_{n\rightarrow \infty}\min_{z \in \Lambda^{v}(4R\sqrt{n})} \pp \left( \max_{1 \leq |\ell| \leq L}t(z,\bL_{M,\ell}^{\delta,v}(n))\leq  \delta n \right) \geq \theta,
}
where $\theta =\theta(R,L,M,\delta)$ is a positive constant and for $t \in \R$
\[
\Lambda^v(t)= \Big\{ \Big\lfloor \sum_{i=1}^d x_i v_i\Big \rfloor \mid  (x_i)_{i=1}^d\in \Lambda(t) \Big\}. 
\]
As in the proof of \eqref{pdp}, we take $Q=2LM+4R$  and denote the upper and lower faces  along the second coordinate of the box $\Lambda^v( Q\sqrt{n})$ by
\al{
F_\pm^{v} &= \Big\{ \Big\lfloor \sum_{i=1}^d x_i v_i\Big \rfloor  \mid  (x_i)_{i=1}^d\in  [- Q\sqrt{n},  Q\sqrt{n}] \times \{ \pm  Q\sqrt{n}\} \times [- Q\sqrt{n}, Q\sqrt{n}]^{d-2}\Big\}.
}
The following lemma is the general-direction analogue of the observation used in \eqref{eq: hit L}.
\begin{lem}\label{lem:gen-slices}
For $|\ell| \leq L$, let
\[
\widetilde{\bL}_{M,\ell}^{\delta,v}(n):=
\Big\{ \sum_{i=1}^d x_i v_i \mid (x_i)_{i=1}^d\in [-\delta n, (1+\delta)n]\times \{ 2\ell \lf M\sqrt{n}\rf \}\times [-\delta n, \delta n]^{d-2} \Big\}\subset \R^d.
\]
Then, for all $n$ sufficiently large, if $z \in \Lambda^v(Q\sqrt{n})$ and $\kR^z_{\delta n/2}$ touches both the faces $F^v_\pm$, then $\kR^z_{\delta n/2}$ intersects all the slices $\bL_{M,\ell}^{\delta,v}(n)$ with $|\ell| \leq L$.
\end{lem}

\begin{proof}
Fix $z \in \Lambda^v(Q\sqrt{n})$ and assume that $\kR^z_{\delta n/2}$ touches both $F^v_-$ and $F^v_+$. Then there exists a nearest-neighbour path $\gamma$ contained in $\kR^z_{\delta n/2}$, joining some $x_- \in F^v_-$ to some $x_+ \in F^v_+$, and passing through $z$. In particular, every point of $\gamma$ is within $\ell^1$-distance at most $\delta n/2$ from $z$.

For $x \in \R^d$, write $x^{(i)}:=\langle x,v_i \rangle$. Since $z \in \Lambda^v(Q\sqrt{n})$, we can choose $(a_i)_{i=1}^d \in [-Q\sqrt{n},Q\sqrt{n}]^d$ so that
\[
z=\Big\lfloor \sum_{i=1}^d a_i v_i \Big\rfloor.
\]
Hence, for each $i \in [d]$,
\[
|z^{(i)}| \leq |a_i| + \Big| \Big\langle \sum_{j=1}^d a_j v_j-z,v_i \Big\rangle \Big| \leq Q\sqrt{n}+\sqrt{d}.
\]
Therefore, for all sufficiently large $n$, we have $|z^{(i)}| \leq \delta n/2$ for every $i \in [d]$. Since $|x-z|_1 \leq \delta n/2$ for every point $x$ of $\gamma$, it follows that
\[
|x^{(i)}| \leq |z^{(i)}| + |x-z|_1 \leq \delta n \qquad \forall\, i \in [d].
\]
In particular, every point $x$ of $\gamma$ satisfies
\[
x^{(1)} \in [-\delta n,(1+\delta)n], \qquad x^{(i)} \in [-\delta n,\delta n] \ \text{for } i=3,\dots,d.
\]

By the definition of $F^v_\pm$, we have
\[
x_-^{(2)} \leq -Q\sqrt{n}+\sqrt{d}, \qquad x_+^{(2)} \geq Q\sqrt{n}-\sqrt{d}.
\]
Since $Q=2LM+4R$, for all sufficiently large $n$ and every $|\ell| \leq L$,
\[
x_-^{(2)} < 2\ell \lf M\sqrt{n}\rf < x_+^{(2)}.
\]
Viewing $\gamma$ as a polygonal path in $\R^d$, the continuous function $x \mapsto x^{(2)}$ must therefore take the value $2\ell \lf M\sqrt{n}\rf$ somewhere on $\gamma$. Hence $\gamma$ intersects $\widetilde{\bL}_{M,\ell}^{\delta,v}(n)$.

If the intersection point is a lattice point, then it belongs to $\bL_{M,\ell}^{\delta,v}(n)$. Otherwise, it lies on an edge $[u,u\pm \eb_j]$ of $\gamma$, and its coordinatewise floor is either $u$ or $u\pm \eb_j$. Thus $\gamma$ visits a vertex of
\[
\{ \lfloor x \rfloor : x \in \widetilde{\bL}_{M,\ell}^{\delta,v}(n) \}=\bL_{M,\ell}^{\delta,v}(n).
\]
Therefore $\gamma$, and hence $\kR^z_{\delta n/2}$, intersects $\bL_{M,\ell}^{\delta,v}(n)$ for every $|\ell| \leq L$.
\end{proof}
By Lemma~\ref{lem:gen-slices},
\al{
	&\min_{z \in \Lambda^v(4R\sqrt{n})} \pp \left( \max_{1 \leq |\ell| \leq L}t(z,\bL_{M,\ell}^{\delta,v}(n))\leq  \delta n \right)\\
	& \geq \min_{z \in \Lambda^v( Q\sqrt{n})} \pp \left( \kR^z_{\delta n/2} \cap F_-^v \neq \varnothing, \kR^z_{\delta n/2} \cap F_+^{{v}} \neq \varnothing \right). 
}
Using the same argument as in Lemma \ref{kr} (iii), we can also prove that the last probability is larger than $\theta$, a positive constant depending only on $Q$ and $\delta$.  Hence, \eqref{pdpv} follows and this completes the proof is completed.  \qed

\section{Upper bound on the large deviation probability  in two dimensions}\label{sec:theomain2}
As in dimensions $d\geq 3$, we only consider the case $v=\eb_1$ since  the proof for general directions is similar. 
 The aim of this section is  to prove that for any $\xi>0$
\ben{ \label{ur2d}
	\pp({\rm T}(0,n\eb_1) \geq (\mu + \xi) n) \leq \exp(-(\pi \xi -o(1)) n/\log n),
}
where $\mu=\mu(\eb_1)$. 
\begin{lem} \label{Step6}
For any $s,\e>0$, there exist positive constants $\delta$ and $M$ such that for all  sufficiently large $n$,
\begin{align*}
    \pp\bigl( |\{x\in \Lambda(n\eb_1, M\sqrt{n}):{\rm T}(0,x)\leq (\mu+5\e)n\}|< (1-\delta) |\Lambda(n\eb_1, M\sqrt{n})| \bigr)
    \leq e^{-s n/\log{n}}.
\end{align*}
\end{lem}

\medskip
\noindent {\bf Proof of the upper bound \eqref{ur2d}}. 
Lemma~\ref{Step6} implies that for given $\xi,\e$, we can choose $M,\delta>0$ such that  for all $n\in \N$ large enough,
\begin{align*}
   \pp\bigl( \kE_M^{\e,\delta}(n)^c \bigr) \leq \exp(-2 \pi \xi n/ \log n),
\end{align*}
where we define the event $\kE_M^{\e,\delta}(n)$ by
\begin{align*}
    \bigl\{ |\{x\in \Lambda(n\eb_1, M\sqrt{n}) \setminus \{ n\eb_1 \}:{\rm T}(0,x)\leq (\mu+5\e)n\}| \geq  (1-\delta) |\Lambda(n\eb_1, M\sqrt{n})| \bigr\}.
\end{align*}
Recall the notation  
\begin{align}
	\kA_M^\delta(n):=\bigl\{ A\subset\Z^d \setminus \{0\}: |A\cap  \Lambda(M\sqrt{n})| \geq (1-\delta) |\Lambda(M\sqrt{n})| \bigr\}.
\end{align} 
Hence, due to Lemma  \ref{Step73d}  and the shift invariance, for any $\e,\delta$ sufficiently small and $M$ sufficiently large, if $n$ is large enough, then
\begin{align*}
    \pp( {\rm T}(0,n\eb_1) \geq (\mu+\xi) n )
    &\leq \pp\bigl( \kE_M^{\e,\delta}(n)^c \bigr)+\pp\bigl( \{ {\rm T}(0,n\eb_1) \geq (\mu+\xi) n \} \cap \kE_M^{\e,\delta}(n) \bigr)\\
    &\leq \exp(- 2\pi \xi n/ \log n)+\sup_{A \in \kA_M^\delta(n)} \pp\Bigl( \min_{x \in A} t(x,0) \geq (\xi -5 \e)n \Bigr)\\
    &\leq 2\exp(-(\pi-\e) (\xi-5 \e) n/\log n).
\end{align*}
Therefore, since $\e$ is arbitrary,    we obtain the upper bound \eqref{ur2d}.  \qed

\medskip
\noindent {\bf Proof of Lemma~\ref{Step6}}. 
We follow a similar scheme to that in the proofs of Lemmas \ref{lem:pec} and \ref{lem:t1}. 

Define for any $\delta>0$, $M,L \in\N$, $\ell \in [-L,L]$ and $n \in \N$,
\begin{align*}
    &\bS_{M,\ell}^\delta(n)=[-2\delta n, (1+2\delta)n]\times [(2\ell-1) \lfloor M\sqrt{n} \rfloor+1, (2\ell+1) \lfloor M\sqrt{n} \rfloor -1 ] \cap \Z^2,\\
    &\bL_{M,\ell}^\delta(n)=[-\delta n, (1+\delta)n]\times \{ 2\ell {\lfloor M\sqrt{n}\rfloor}\} \cap \Z^2.
\end{align*}
\begin{lem} \label{Step3}
For any $\e>0$, there exist $\delta\in(0,\e/4)$, $M_0 \in \N,c>0$ such that for any $M\geq M_0$,  $\ell \in \Z$ and $n\in \N$ large enough,
\begin{align*}
    \pp \left( \max_{x \in \bL_{M,\ell}^\delta(n)}{\rm T}_{\bS_{M,\ell}^\delta(n)}(x,y_\ell ) \leq (\mu+\e)n \right) \geq 1-\exp(-cn/\log{n}),
\end{align*}
where $y_\ell:=n\mathbf{e}_1+2\ell \lf M\sqrt{n}\rf\mathbf{e}_2 \in \bL_{M,\ell}^\delta(n)$.
\end{lem}
\begin{proof}
We can apply exactly the same argument as in the proof of Lemma~\ref{lem:tsl} by replacing Proposition~\ref{tuld}~(i) and Proposition~\ref{tln}~(i) with Proposition~\ref{tuld}~(ii) and Proposition~\ref{tln}~(iii), respectively (the factor $n/\log n$ comes from these changes).
\end{proof}
\begin{lem} \label{Step4}
For any $s,\e>0$, there exist $\delta>0$ and $M_0,L_0 \in \N$ such that for all $M\geq M_0$, $L\geq L_0$ and $n$ sufficiently large,
\begin{align*}
    \pp \left( \min_{\ell:  1 \leq |\ell| \leq L}  \,  \max_{x \in \Lambda(y_\ell, \sqrt{n})}{\rm T}(0,x) \leq  (\mu+4\e)n   \right)\geq 1- \exp(-sn/\log{n}).
\end{align*}
\end{lem}
\begin{proof}
Fix $s,\e>0$.
 Let $\delta \in (0,\e/4)$, $M_0$, $c$  as in Lemma \ref{Step3} and let $M \geq M_0$. 
By the triangle inequality, for all  $x \in \Lambda(y_\ell,\sqrt{n})$,
\begin{align*}
    {\rm T}(0,x) \leq {\rm T}(0,\bL_{M,\ell}^\delta(n))+ \max_{y \in \bL_{M,\ell}^\delta(n)}{\rm T}(y,x) \leq {\rm T}(0,\bL_{M,\ell}^\delta(n))+ \max_{y \in \bL_{M,\ell}^\delta(n)}\T_{\bS_{M,\ell}^\delta(n)}(y,x).
\end{align*}
Hence,
\al{
 \min_{\ell:  1 \leq |\ell| \leq L}  \,  \max_{x \in \Lambda(y_\ell, \sqrt{n})}{\rm T}(0,x) &\leq \max_{\ell:  1 \leq |\ell| \leq L} {\rm T}(0,\bL_{M,\ell}^\delta(n)) \\
 & \quad  + \min_{\ell:  1 \leq |\ell| \leq L} \max_{y \in \bL_{M,\ell}^\delta(n)} \max_{x \in \Lambda(y_\ell,\sqrt{n})} \T_{\bS_{M,\ell}^\delta(n)}(y,x).
}
Therefore,  it suffices to show that there exists $L_0 \in \N$ such that for all $L \geq L_0$ and $n$ sufficiently large,
\begin{align}
    \label{eq:TL1}
    &\pp\left(\max_{\ell: 1 \leq |\ell| \leq L} {\rm T}(0,\bL_{M,\ell}^\delta(n)) \leq 2\e n \right) \geq 1- \exp(-2sn/\log n),\\
    \label{eq:TL2}
    &\pp\biggl( \,   \bigcup_{\ell: 1 \leq |\ell| \leq L} \{ \bL_{M,\ell}^\delta(n) \textrm{ is good} \} \biggr) \geq 1- \exp(-2sn/\log n),
\end{align}
where
\be{
\{\bL_{M,\ell}^\delta(n) \textrm{ is good}\}= \Big \{\max_{y \in \bL_{M,\ell}^\delta(n)} \max_{x \in \Lambda(y_\ell,\sqrt{n})} \T_{\bS_{M,\ell}^\delta(n)}(y,x) \leq (\mu+2\e) n \Big \}.
}
Let us first treat \eqref{eq:TL2}. Observe that for all $\ell\in \Iintv{-L,L}\setminus\{0\}$,
\be{
\max_{y \in \bL_{M,\ell}^\delta(n)} \max_{x \in \Lambda(y_\ell,\sqrt{n})} \T_{\bS_{M,\ell}^\delta(n)}(y,x) \leq \max_{y \in \bL_{M,\ell}^\delta(n)} \T_{\bS_{M,\ell}^\delta(n)}(y,y_\ell) + \max_{x \in \Lambda(y_\ell,\sqrt{n})} \T_{\Lambda(y_\ell,\sqrt{n})}(y_\ell,x).
}
Therefore,
\bea{
&&\pp(\bL_{M,\ell}^\delta(n) \textrm{ is not good})\\
&\leq& \pp \left(\max_{y \in \bL_{M,\ell}^\delta(n)} \T_{\bS_{M,\ell}^\delta(n)}(y,y_\ell) \geq (\mu+\varepsilon)n \right) + \pp \left( \max_{x \in \Lambda(y_\ell,\sqrt{n})} \T_{\Lambda(y_\ell,\sqrt{n})}(y_\ell,x) \geq \varepsilon n\right)\\
&\leq&  \exp(-cn/\log n),
}
where $c=c(\varepsilon)$ is a positive constant, by using Proposition \ref{tuld} (ii) and Lemma \ref{Step3}. On the other hand, these events $(\{\bL_{M,\ell}^\delta(n) \textrm{ is good}\})_{\ell}$ are independent, since the slabs $(\bS_{M,\ell}^\delta(n))_{\ell}$ are disjoint. Hence, \eqref{eq:TL2} holds for all $L \geq s/c$.
The inequality \eqref{eq:TL1} follows from  the same argument as in the proof of Lemma~\ref{lem:t1}, in  particular,  the estimate \eqref{mtaln} (see Section~\ref{subsect:lem4.3} and replace Proposition~\ref{tln}~(i) and (ii) with Proposition~\ref{tln}~(iii) and (iv)).
\end{proof}
\begin{proof}[\underline{Proof of Lemma~\ref{Step6}}]
First, thanks to Lemma \ref{Step4},  we can assume that there exists $\ell \in \Iintv{-L,L}\setminus\{0\}$, such that ${\rm T}(0,x) \leq (\mu+4\e)n$ for all $x \in \Lambda(y_{\ell},\sqrt{n})$. Then the proof is analogous to that  of Lemma \ref{lem:tbyl}.
    \end{proof}

\section{Initial growth of infection: Proof of Proposition~\ref{tln}}\label{sec:proptln}

The aim of this section is to prove Proposition~\ref{tln}, which
roughly says that with high probability, the box $\Lambda(\sqrt{n})$ is fully occupied within $o(n)$ steps.  The key ingredient in the proof is Proposition~\ref{propk3d}, which describes the  growth of infection during the early stages. The main idea is to construct a coupling between the frog model and an exploration process with greater independence, facilitating analysis. The section is organized as follows. In Section~\ref{subsect:quadrup}, we introduce a tool to estimate the tail behavior of the passage time, which is essential for analyzing the coupled exploration process. At the beginning of Section~\ref{subsect:growth_begin}, we prove Part (i) of Proposition~\ref{tln}, assuming Proposition~\ref{propk3d}, which is then established in Sections~\ref{subsubsect:propk3db} and \ref{subsubsect:propk3da}. The remaining parts, (ii)–(iv), of Proposition~\ref{tln} are proved in Section~\ref{subsect:2-4}, using arguments similar to those in the proof of Part (i).

\subsection{Quadrupling time}\label{subsect:quadrup}
The aim of this subsection is to investigate the \textit{quadrupling time}, which is the duration required for an activated population to grow fourfold.
In particular, for any non-empty finite set $G\subset \Z^d$, we consider the following stopping times $\mathscr{T}_G$ and $\mathscr{T}_G^*$:
\begin{align*}
    &\mathscr{T}_G:=\inf \Big\{ t \geq 0: \Big| \bigcup_{x\in G} \kR^x_t \Big| \geq 4 |G| \Big\},\\
    &\mathscr{T}^*_G:=\inf \Big\{t \geq 0:\Big| \bigcup_{x\in G} \bigl( \kR^x_t \cap \Lambda(x,\sqrt{t}) \bigr) \Big| \geq 4 |G| \Big\}.
\end{align*}
Here, $\sct_G$ is the time it takes for the paths of active frogs starting from $G$ to cover a region of size larger than $4|G|$, while $\hsct_G$ is a constrained version of $\sct_G$ where the spread of active frogs is observed only within $\Lambda(G,\sqrt{t})$ over time $t$. These \textit{quadrupling times} help us describe the spread of active frogs in the early stage.

The following lemma provides precise  tail estimates for the quadrupling times.

\begin{lem} \label{lte}
Assume $d\geq 2$.
There exists a positive constant $c \in (0,1)$ such that for all non-empty finite sets $G \subset \Z^d$ and $t \geq 0$,
\begin{align*}
    &\pp( \hsct_G >t ) \leq
    \begin{cases}
        1 & \text{ if } t \leq c^{-1}|G|,\\
        \exp \left(-ct^{\tfrac{d}{2d-2}} |G|^{\tfrac{d-2}{d-1}} \right) & \text{ if } c^{-1}|G|<t \leq  c^{-1}|G|^2, \\
        \exp \left(-ct|G|^{(d-2)/d} \right) & \text{ if }t >  c^{-1}|G|^2,
    \end{cases}\\[0.5em]
    &\pp( \sct_G >t ) \leq
    \begin{cases}
        1 & \text{ if } t \leq c^{-1}|G|^{2/d},\\
        \exp \left(-ct|G|^{(d-2)/d} \right) &\text{ if }t >   c^{-1}|G|^{2/d}.
    \end{cases}
\end{align*}
\end{lem}

Before proving Lemma~\ref{lte}, we provide some estimates on the range of random walk in the following lemma. Its proof is postponed to the appendix since it is technical.

\begin{lem} \label{pma}
Assume $d\geq 2$. There exists a positive constant $c$ such that the following results hold for any non-empty finite set $A\subset \Z^d$:
\begin{itemize}
\item [(i)] For all $m\in \N$ with $m\geq c^{-1}|A|^{2/d}$,
\begin{align*}
    q(m,A):=\pp(|\kR_m \setminus A| < 8) \leq \exp(-cm|A|^{-2/d}).
\end{align*}

\item[(ii)] For all $m \in \N$,
\begin{align*}
    p(m,A)
    &:= \pp\bigl( |\kR_m \cap (\Lambda(\sqrt{m}) \setminus A)|<8 \bigr)\\
    &\leq
    \begin{cases}
        \exp\left(-c\frac{m}{|A|^{2/d}}\right) &\textrm{ if } |A| \leq \sqrt{m}/4,\\
        \exp\left(-c\left(\frac{m^{d/2}}{|A|}\right)^{\frac{1}{d-1}}\right) &\textrm{ if }  \sqrt{m}/4 < |A| \leq cm.
    \end{cases}
\end{align*}
\end{itemize}
\end{lem}

Each exponent in Lemma \ref{pma} is optimal up to constants depending only on $d$ (See Remark~\ref{remark: SRW range remark}).

\begin{proof}[Proof of Lemma \ref{lte}]
We only give the proof of Part~(i), since the same argument works for Part~(ii).
Enumerate $G=\{x_1,\ldots,x_k\}$ arbitrarily and define for any $t\in\N$,
\begin{align*}
    A_t^x:=\kR^x_t \cap \Lambda(x,\sqrt{t}),\qquad
    \tau_t:=\inf\Big\{ j\in \llbracket 1,k \rrbracket: \Big| \bigcup_{i=1}^j A^{x_i}_t \Big| \geq 4k \Big\}.
\end{align*}
Since $\{\hsct_G>t\} = \{\tau_t = \infty\}$, letting $X_j:= \mathbf{1}\{ |A^{x_j}_t \setminus \cup_{i=1}^{j-1} A^{x_i}_t|<8 \}$, one has
\begin{align}\label{tgj}
	\{\hsct_G >t\} \subset \biggl\{ \sum_{j=1}^k X_j \geq k/2,\,\tau_t=\infty \biggr\}.
\end{align}
 On the other hand, if $\tau_t >j-1$, then $|\cup_{i=1}^{j-1} A^{x_i}_t|< 4k$.
 Thus, by Lemma \ref{pma} (ii), there exists a constant $c_1>0$ such that if $t$ is sufficiently large and $k<c_1t/4$, then
\begin{equation} \label{ptk}
    \pp(X_j=1 \mid \tau_t >j-1)=\pp\left( |A^{x_j}_t \setminus \cup_{i=1}^{j-1} A^{x_i}_t|<8 \mid \tau_t >j-1 \right) \leq p_{t,k},
\end{equation}
where
\begin{equation*}
	p_{t,k}= \begin{cases}
	\exp(-c_1t(4k)^{-2/d}) \hspace{2cm} \textrm{if } k \leq c_1\sqrt{t}/4,\\
	\exp(-c_1(t^{d/2}/
	(4k))^{1/(d-1)}) \hspace{0.8cm} \textrm{if } c_1\sqrt{t}/4<k<c_1t/4.
	\end{cases}
\end{equation*}
Observe that when $t\geq 1/c_2$ with $c_2=c_2(c_1)$ a small constant,  $p_{t,k} \leq 2^{-4}$, and thus $2^k \leq p_{t,k}^{-k/4}$.
Therefore, \eqref{tgj}, \eqref{ptk} and Lemma~\ref{dobbi} yield that if $k<c_1t/4$ and $t\geq 1/c_2$, then
\begin{align*}
    \pp(\hsct_G>t)
    \leq \pp \Biggl( \sum_{j=1}^k X_j \geq k/2, \, \tau_t =\infty \Biggr)
    \leq \binom{k}{\lceil k/2\rceil}p_{t,k}^{\lceil k/2 \rceil}
    \leq 2^k p_{t,k}^{\lceil k/2 \rceil}  \leq  p_{t,k}^{k/4}.
\end{align*}
In summary, there exists a small constant $c>0$ depending on $c_1,c_2$ such that
	\begin{equation*}
	  \pp(\hsct_G>t) \leq
          \begin{cases}
          1 & \text{ if } t \leq k/c,\\
          \exp(-c k(t^{d/2}/
	k)^{1/(d-1)}) & \text{ if }  k/c< t \leq k^2/c , \\
		\exp(-ctk^{1-2/d}) & \text{ if } t >k^2/c,
        \end{cases}
        	\end{equation*}
which is the desired conclusion.
\end{proof}

\begin{lem}\label{lem6.4}
There exists a positive constant $c \in (0,1)$ such that if $n$ is sufficiently large, then  for any non-empty finite set $G \subset \Z^d$ and  $\alpha \in (1/\sqrt{n},c)$, 
\begin{align*}
    &\E \left[ \exp \left( \alpha \sqrt{n \hsct_G} \right) \right]\\
    &\leq  n^d\exp \left( \alpha \sqrt{n |G|/c}+{c^{-1} \alpha^2 n |G|^{-\frac{d-2}{d}}}+  \alpha \sqrt{n} \min \left\{ \tfrac{|G|^2}{c}, \left( \tfrac{2\alpha}{c} \right)^{2d-2} \tfrac{n^{d-1}}{|G|^{2d-4}} \right\}^{1/2} \right).
\end{align*}
\end{lem}
\begin{proof} Let $c \in (0,1)$ be the constant as in Lemma \ref{lte}. First, we have
\begin{align}\label{ba1}
   \E \left[ \exp \left( \alpha \sqrt{n \hsct_G} \right) \mathbf{1}(\hsct_G \leq |G|/c) \right] \leq  \exp \left(\alpha \sqrt{n |G|/c} \right).
\end{align}
By Lemma \ref{lte},
\ben{ \label{tgl}
\E \left[ \exp \left( \alpha \sqrt{n \hsct_G} \right) \mathbf{1}(\hsct_G \geq |G|^2/c) \right]
\leq \sum_{j \geq |G|^2/c}   \exp \left(\alpha \sqrt{n j} - cj |G|^{(d-2)/d} \right).
}
Since $ax-bx^2 \leq a^2/4b \leq a^2/b$ for all $a,x \in \R$ and $b>0$,
\be{
\alpha\sqrt{n j} - cj |G|^{(d-2)/d} \leq \frac{\alpha ^2 n}{c|G|^{(d-2)/d}}
    \quad \text{for all $j \geq 1$},
}
and since $\tfrac{1}{\sqrt{n}} \leq \alpha \leq c \leq 1$ and $|G| \geq 1$,
\begin{align*}
   \alpha\sqrt{n j} - cj |G|^{(d-2)/d}  \leq \alpha(\sqrt{nj}-j) \leq- \sqrt{j} \quad \text{for all $j \geq 4 n$}.
\end{align*}
These estimates yields
\begin{align*}
\begin{split}
    \sum_{j \geq |G|^2/c}   \exp \left(\alpha \sqrt{n j} - cj |G|^{(d-2)/d} \right)
    &\leq 4n \exp \left(\frac{\alpha ^2 n}{c |G|^{(d-2)/d}} \right) + \sum_{j \geq 4n}  \exp(-\sqrt{j}),
\end{split}
\end{align*}
which together with \eqref{tgl} implies that
\ben{ \label{ba2}
	\E \left[ \exp \left( \alpha \sqrt{n \hsct_G} \right) \mathbf{1}(\hsct_G \geq |G|^2/c) \right]
	\leq 8n \exp \left(\frac{\alpha ^2 n}{c |G|^{(d-2)/d}} \right).
}
By Lemma \ref{lte}, 
\al{
&\E \left[ \exp \left( \alpha \sqrt{n \hsct_G} \right) \mathbf{1}(|G|/c<\hsct_G<|G|^2/c) \right]\\
&\leq \sum_{|G|/c<j<|G|^2/c} \exp\left(\alpha \sqrt{n j} - cj^{\tfrac{d}{2d-2}} |G|^{\tfrac{d-2}{d-1}} \right).
}
It is clear that for all $j \geq (2\alpha/c)^{2d-2} n^{d-1}/|G|^{2d-4}$,
\begin{align*}
    \alpha \sqrt{n j} - cj^{\tfrac{d}{2d-2}} |G|^{\tfrac{d-2}{d-1}} \leq -\alpha \sqrt{nj}\leq -\sqrt{j}.
\end{align*}
Combining the last two estimates, we get
\begin{align}\label{eq: eq A3}
\begin{split}
    &\E \left[ \exp \left( \alpha \sqrt{n \hsct_G} \right) \mathbf{1}(|G|/c<\hsct_G<|G|^2/c) \right]\\
    &\leq \sum_{|G|/c<j\leq \min \left \{ \frac{|G|^2}{c}, \left( \frac{2\alpha}{c} \right)^{2d-2} \frac{n^{d-1}}{|G|^{2d-4}} \right\}}   \exp \left(\alpha \sqrt{n j} \right) + \sum_{j \geq 1} \exp(-\sqrt{j}) \\
    & \leq n^d \exp \left( \alpha \sqrt{n} \min \left\{ \tfrac{|G|^2}{c}, \left( \tfrac{2\alpha}{c} \right)^{2d-2} \tfrac{n^{d-1}}{|G|^{2d-4}} \right\}^{1/2} \right).
\end{split}
\end{align}
Combining \eqref{ba1}, \eqref{ba2} and \eqref{eq: eq A3}, we get the desired result by decreasing $c>0$ if needed.
\end{proof}

\subsection{Proof of Proposition~\ref{tln}~(i)}\label{subsect:growth_begin}
Let $d \geq 3$.
The following lemma tells us that Proposition~\ref{tln}~(i) holds for all $M,\e>0$ once we prove it with $M=\e=1$.
\begin{lem}\label{lem:Me1}
If Proposition \ref{tln} (i) holds for $M =\e =1$, then it holds for all $M,\e >0$.
\end{lem}
\begin{proof}  Assume that the statement of (i) holds for $M=\e =1$, that is, there exists $c>0$ such that for all $n$ sufficiently large,
\begin{align*}
    \pp \left( {\rm T}_{\Lambda(\sqrt{n})}(0,x) \leq n \quad \forall x \in \Lambda(\sqrt{n}) \right) \geq 1-\exp(-cn).
\end{align*}
Let $M,\e >0$ and $x \in \Lambda(\sqrt{n})$.
We define $\e'= \min \{M^2, \e/2\}$.
Applying the above assumption to $\e'n$, one has for all $n$ sufficiently large,
\begin{align}\label{me=1}
    \pp\left( {\rm T}_{\Lambda(\sqrt{\e'n})}(0,y) \leq  \e'n \quad \forall y \in \Lambda(\sqrt{\e'n}) \right) \geq 1- \exp(-c'n),
\end{align}
with $c'=c'(\e')>0$.
Observe that if ${\rm T}_{\Lambda(\sqrt{\e'n})}(0,y) \leq  \e' n$ and $t(y,x) \leq \e'n$ for some $y \in \Lambda(\sqrt{\e' n})$, then ${\rm T}_{\Lambda(M\sqrt{n})}(0,x) \leq {\rm T}_{\Lambda(\sqrt{\e' n})}(0,x) \leq 2 \e' n \leq \e n$.
Thus,  \eqref{me=1} gives
\begin{align}\label{tbmny}
    \pp \left( {\rm T}_{\Lambda(M\sqrt{n})}(0,x) \leq  \e n   \right) \geq 1- \exp(-c'n) - \pp \left( t(y,x) > \e ' n \quad \forall y \in \Lambda(\sqrt{\e' n}) \right).
\end{align}
 By Lemma~\ref{kr}~(i), there exists a positive constant $C=C(\e')>0$ such that if $x \in \Lambda(\sqrt{n})$ and $y \in \Lambda(\sqrt{\e' n})$ then
\be{
\pp \left( t(y,x) > \e ' n  \right) \leq 1 - Cn^{1-d/2}.
}
Hence, there exists a constant $c=c(C, \e')>0$ such that
\begin{align*}
    \pp \left( t(y,x) > \e 'n \quad \forall y \in \Lambda(\sqrt{\e'n}) \right) \leq \left(1 - Cn^{1-d/2}\right)^{(\e 'n)^{d/2}} \leq \exp(-cn).
\end{align*}
Combining this with \eqref{tbmny}, we obtain the desired result.
\end{proof}

Thanks to the above lemma, we may assume $M=\e=1$ in the proof of Proposition~\ref{tln}~(i).
 We first prepare some notation and proposition.
Let us define
\begin{align}\label{eq:Lambdai}
\theta:=1/(8d), \qquad \forall \, i\in \Iintv{0, \lceil d/2 \rceil}: \quad     \Lambda_i:=\Lambda\Bigl( \sqrt{n\left(\tfrac{1}{2}+ i \theta\right)} \Bigr).
    \end{align}
\begin{prop} \label{propk3d} Let $d\geq 3$. Then there exist  positive constants $(c_i)_{ i \in \Iintv{1, d/2} }$ such that the following results hold for all large $n$:
\begin{itemize}
\item [(a)]  $\pp(\kE_1) \geq 1-\exp(-c_1n)$;
\item[(b)] $\pp(\kE_i \mid \kE_{i-1}) \geq 1-\exp(-c_in)$  for all $2 \leq i \leq  \lceil d/2 \rceil$,
\end{itemize}
where for $i \in \Iintv{1, \lceil d/2 \rceil}$
\ben{ \label{dei}
\kE_i=\bigl\{ |\kC_i| \geq (c_in)^{i \wedge (d/2)} \bigr\}
}
with 
\ben{ \label{dci}
	\kC_i=\bigl\{ y \in \Lambda_i \setminus \Lambda_{i-1}: {\rm T}_{\Lambda_{i-1}}(0,y) \leq n\left(\tfrac{1}{2}+ i \theta\right) \bigr\}.
}
\end{prop}
We defer the proof to Sections~\ref{subsubsect:propk3db} and \ref{subsubsect:propk3da}.
Let us complete the proof of Proposition~\ref{tln}~(i) for now.
\begin{proof}[Proof of Proposition~\ref{tln}~(i)] By Proposition \ref{propk3d},  there exists a constant $c$ such that 
\begin{align*}
  \pp(\kE_{\lceil d/2\rceil })= \pp (|\kC_{\lceil d/2\rceil }| \geq (cn)^{d/2}) \geq  1-\exp(-cn).
\end{align*}
Assume that $\kE_{\lceil d/2\rceil }$ occurs. Then  $\kC_{\lceil d/2\rceil }$ is a subset of $\Lambda_{\lceil d/2 \rceil} \setminus \Lambda_{\lceil d/2 \rceil-1}$ with cardinality larger than $(cn)^{d/2}$, and for any $y \in \kC_{\lceil d/2\rceil }$, one has $
    {\rm T}_{\Lambda_{\lceil d/2\rceil -1}}(0,y) \leq 3n/4 $. Therefore,
    \be{
 \pp\bigl( \bigl\{ {\rm T}_{\Lambda(\sqrt{n})}(0,x) \geq n \bigr\} \cap \kE_{\lceil d/2\rceil } \bigr)
\leq \pp\bigl( \bigl\{ t(y,x) > n/4 \quad \forall \, y \in \kC_{\lceil d/2\rceil } \bigr\} \cap \kE_{\lceil d/2\rceil } \bigr).    
}
Note that the event $\kE_{\lceil d/2\rceil}$ depends only on the frogs inside $\Lambda_{\lceil d/2\rceil -1}$, so it is independent of the simple random walks starting from $\kC_{\lceil d/2\rceil }$ which is a subset of $\Lambda_{\lceil d/2 \rceil} \setminus \Lambda_{\lceil d/2 \rceil-1}$. 
Hence,  for all $x \in \Lambda(\sqrt{n})$
\begin{align*}
     &\pp\bigl( \bigl\{ t(y,x) > n/4 \quad \forall \, y \in \kC_{\lceil d/2\rceil } \bigr\} \cap \kE_{\lceil d/2\rceil } \bigr)\\
    &\leq \left(\max_{y \in \Lambda_{\lceil d/2 \rceil} \setminus \Lambda_{\lceil d/2 \rceil-1}} \max_{z \in \Lambda(\sqrt{n})} \pp(t(y,z) > n/4) \right)^{(cn)^{d/2}} \\  &\leq \left(\max_{y, z \in \Lambda(\sqrt{n})} \pp(t(y,z) > n/4) \right)^{(cn)^{d/2}}.
\end{align*}
By Lemma~\ref{kr}~(i),  $\max_{y, z \in \Lambda(\sqrt{n})} \pp(t(y,z) > n/4) \leq 1- Cn^{1-d/2}$ with $C$ a positive constant. Hence, 
\be{
\pp\bigl( \bigl\{ t(y,x) > n/4 \quad \forall \, y \in \kC_{\lceil d/2\rceil } \bigr\} \cap \kE_{\lceil d/2\rceil } \bigr) \leq (1- Cn^{1-d/2})^{(cn)^{d/2}} \leq \exp(-c'n),
}
with $c'=c'(C,c)$ a positive constant. Combining these estimates, we arrive at 
\be{
\pp\bigl(  {\rm T}_{\Lambda(\sqrt{n})}(0,x) \geq n   \bigr) \leq \exp(-cn) + \exp(-c'n).
}
Therefore, Proposition~\ref{tln}~(i) holds  with $M=\e=1$, and it also holds for all $M,\e>0$ by Lemma~\ref{lem:Me1}.
\end{proof}

\subsubsection{Blossoms of infection: Proof of Proposition \ref{propk3d} (b)}\label{subsubsect:propk3db}
We prove here the case $i=2$, and the proof for the general case is similar. We recall from  \eqref{eq:Lambdai}--\eqref{dci}  the definition of $(\Lambda_i)_{i\geq 1}$, $(\kE_i)_{i\geq 1}$ and $(\kC_i)_{i\geq 1}$.
In particular, if $\kE_1$ occurs, 
\begin{align}\label{eq:E1prop}
    |\kC_1| \geq c_1 n,\quad \kC_1 \subset \Lambda_1 \setminus \Lambda_0,\quad \max_{y \in \kC_1} {\rm T}_{\Lambda_0}(0,y) \leq n\Bigl( \frac{1}{2}+\theta \Bigr).
\end{align}
Hence, if $t(y,z) \leq \theta n$ for some $y \in \kC_1$, then ${\rm T}_{\Lambda_1}(0,z) \leq n(\tfrac{1}{2}+2\theta)$. Thus 
\begin{align*}
    |\kR^{\kC_1}_{\theta n} \cap (\Lambda_2 \setminus \Lambda_1)|
    &= \bigl| \{ z \in \Lambda_2 \setminus \Lambda_1:\exists y \in \kC_1 \text{ such that } t(y,z) \leq \theta n \} \bigr|\\
    &\leq \left| \left\{ z \in \Lambda_2 \setminus \Lambda_1:\T_{\Lambda_1}(0,z) \leq n\left( \tfrac{1}{2}+2\theta \right) \right\} \right|
    =|\kC_2|,
\end{align*}
where we recall  that $\kR^{\kC_1}_{\theta n}:=\bigcup_{y \in \kC_1} \kR^y_{\theta n}$. Let $\kE_2 = \{|\kC_2| \geq (c_2n)^{2 \wedge (d/2)} \}$ with $c_2=c_2(c_1)$ a positive constant chosen later. The above observation implies
\be{
\pp(\kE_2^c \mid \kE_1)=\pp\Bigl( |\kC_2|<(c_2n)^{2 \wedge (d/2)} \mid \kE_1 \Bigr)
\leq \pp\Bigl( |\kR^{\kC_1}_{\theta n} \cap (\Lambda_2 \setminus \Lambda_1)|<(c_2n)^{2 \wedge (d/2)} \mid \kE_1 \Bigr).
}
Notice that  $\kC_1$ is measurable with respect to the sigma-algebra generated by $(S^x_{\cdot})_{x\in\Lambda_0}$  and $\kC_1 \cap \Lambda_0 = \varnothing$. Hence, for any $\Gamma \subset \Lambda_1 \setminus \Lambda_0$ the event $\{\kC_1=\Gamma\}$ and the random set $\kR^{\Gamma}_{\theta n}$ are independent. Thus, 
\al{
& \pp\Bigl( |\kR^{\kC_1}_{\theta n} \cap (\Lambda_2 \setminus \Lambda_1)|<(c_2n)^{2 \wedge (d/2)} \mid \kE_1 \Bigr)\\
&\leq \sum_{\substack{\Gamma \subset \Lambda_1 \setminus \Lambda_0\\|\Gamma| \geq c_1n}} \pp \Bigl( |\kR^{\Gamma}_{\theta n}  \cap (\Lambda_2 \setminus \Lambda_1)| < (c_2n)^{2 \wedge (d/2)},\,\kC_1=\Gamma \mid \kE_1 \Bigr)\\
&\leq \max_{\substack{\Gamma\subset \Lambda_1 \setminus \Lambda_0\\|\Gamma| \geq c_1n}} \pp \left(|\kR^{\Gamma}_{\theta n}  \cap (\Lambda_2 \setminus \Lambda_1)| < (c_2n)^{2 \wedge (d/2)} \right).
}
By Lemma~\ref{rnab}, there exists a positive constant $c_2=c_2(c_1)$ such that
\begin{align*}
    \max_{\substack{\Gamma\subset \Lambda_1 \setminus \Lambda_0\\|\Gamma| \geq c_1n}} \pp \left(|\kR^{\Gamma}_{\theta n}  \cap (\Lambda_2 \setminus \Lambda_1)| < (c_2n)^{2 \wedge (d/2)} \right)
    \leq \exp(-c_2 n).
    \end{align*}
Combining the last three display equations, we have $\pp(\kE_2 \mid \kE_1) \geq 1-\exp(-c_2 n)$  and hence  Proposition \ref{propk3d} (b) follows.\qed

\subsubsection{The growth of infection around the starting point: Proof of Proposition \ref{propk3d} (a)}\label{subsubsect:propk3da}
For any non-empty finite set $G$ and $t\geq 0$, define
\begin{align*}
    A^G_t=\bigcup_{x\in G} \kR^x_t \cap \Lambda(x,\sqrt{t}).
\end{align*}
Then, we consider a modification of the frog model that can be viewed as an exploration process:
\begin{itemize}
\item [(0)] In generation $0$, let $G_0 =\{0\}$ be the position of the first active frog.
We then define
\begin{align*}
    \hsct_0:=\hsct_{G_0}=\inf\bigl\{ t \in \N:|A^{G_0}_t| \geq 4 \bigr\}.
\end{align*}
Assume that between time $0$ and time $\hsct_0-1$, only the frog in $G_0$ is active and the others are sleeping.
\item[(1)] Since $|A^{G_0}_{\hsct_0}| \geq 4$ and $|G_0|=1$,  we can choose a set $G_1 \subset A^{G_0}_{\hsct_0} \setminus G_0$ such that $|G_1|=2$.
We call $G_1$ the first generation.
Let
\begin{align*}
    \hsct_1:=\hsct_{G_1}=\inf\bigl\{ t \in \N: |A^{G_1}_t| \geq 4|G_1| \bigr\}.
\end{align*}
Assume that between time $\hsct_0$ and time $\hsct_0+\hsct_1-1$, the frogs in $G_1$ are active but the others are sleeping.
\item[(2)] Since $|A_{\hsct_1}^{G_1}|\geq 8$ and $|G_0|+|G_1|=3$, we can take a set $G_2 \subset A^{G_1}_{\hsct_1} \setminus(G_0 \cup G_1)$ such that $|G_2|=2^2$, and call $G_2$ the second generation.
Define also
\begin{align*}
    \hsct_2:=\mathscr{T}^*_{G_2}=\inf\bigl\{t \in \N: |A^{G_2}_t| \geq 4|G_2| \bigr\}.
\end{align*}
Similarly, between time $\hsct_0+\hsct_1$ and time $\hsct_0+\hsct_1+\hsct_2-1$, consider that the frogs in $G_2$ are active but the others are sleeping.
\item[(k)] Suppose that we have already constructed $(G_i,\hsct_i)_{i=1}^{k-1}$, satisfying $\hsct_i=\hsct_{G_{i}}$ and $G_i \subset A^{G_{i-1}}_{\hsct_{i-1}} \setminus (\bigcup_{j=1}^{i-1}G_j)$ and $|G_i|=2^i$ for all $i=1,\ldots,k-1$.
Since $|A^{G_{k-1}}_{\hsct_{k-1}}| \geq 4|G_{k-1}|=2^{k+1}$ and $|G_0|+ \ldots+|G_{k-1}|=2^{k}-1$, there exists a set $G_k \subset A^{G_{k-1}}_{\hsct_{k-1}}\setminus (\bigcup_{i=0}^{k-1}G_i)$ such that $|G_k|=2^k$ and call $G_k$ the $k$-th generation.
We define
\begin{align*}
    \hsct_k:=\mathscr{T}^*_{G_{k}}=\inf\bigl\{t \in \N: |A^{G_k}_t| \geq 4|G_k| \bigr\},
\end{align*}
and assume that between time $\hsct_0+\dots+\hsct_{k-1}$ and time $\hsct_0+\cdots+\hsct_k-1$, the frogs in $G_k$ are active but the others are sleeping.
\end{itemize}
By the construction of the modified frog model above, we make the following observations:
\begin{itemize}
    \item [(i)] At any time, an active frog of the modified frog model is also active in the original frog model.
    \item [(ii)] In the $k$-th generation, we use only the information on simple random walks starting from $G_{k}$.
\end{itemize}
 When the sequence $(G_i,\hsct_i)_{i=1}^k$ is constructed, we define
\begin{align*}
    D_k:=\sqrt{\hsct_0}+\dots+\sqrt{\hsct_{k-1}},\qquad H_k:=\hsct_0+\dots+\hsct_{k-1}.
\end{align*}
Since the $L^\infty$-distance between initial locations of frogs in  consecutive generations $G_i$ and $G_{i+1}$ is at most $\sqrt{\hsct_i}$, it is clear that
\begin{align*}
    G_k \subset \Lambda(D_k),\qquad
    \max_{x\in G_k} {\rm T}_{\Lambda(D_k) }(0,x) \leq H_k.
\end{align*}
Let $c_0$ be a small constant chosen later and define
\begin{align}\label{eq:kapppa}
    \kappa_n=\kappa_n(c_0):=\lceil \log_2 (c_0n) \rceil.
\end{align}
By the construction of $G_i$'s, $|G_{\kappa_n}| \geq c_0n$ holds.
Moreover, if $D_{\kappa_n} \leq \sqrt{n/2}$ then
\begin{align*}
    G_{\kappa_n} \subset \Lambda(D_{\kappa_n}) \subset \Lambda\bigl( \sqrt{n/2} \bigr)=\Lambda_0,\qquad
    \max_{x \in G_{\kappa_n}} {\rm T}_{\Lambda_0}(0,x) \leq H_{\kappa_n} \leq D_{\kappa_n}^2 \leq \frac{n}{2}.
\end{align*}
Assume in addition that $|\kR^{G_{\kappa_n}}_{\theta n} \cap (\Lambda_1 \setminus \Lambda_0)| \geq c_1 n$ with a positive constant $c_1$.
Then
\begin{align*}
    |\kC_1|=|\{y \in \Lambda_1 \setminus \Lambda_0: {\rm T}_{\Lambda_0}(0,y) \leq n(\theta + \tfrac{1}{2})\}| \geq c_1 n,
\end{align*}
and hence $\kE_1$ holds.
In summary,
\begin{align*}
    \pp(\kE_1) \geq \pp \left(D_{\kappa_n} \leq \sqrt{n/2}, \, \bigl| \kR^{G_{\kappa_n}}_{\theta n} \cap (\Lambda_1 \setminus \Lambda_0) \bigr| \geq c_1 n \right).
\end{align*}
Note that, by the construction of the modified frog model, $(\hsct_k,G_k)_{0 \leq k \leq \kappa_n}$ is independent of $(S^x_\cdot)_{x \in G_{\kappa_n}}$. 
Therefore,
\begin{align*}
    \pp \left(D_{\kappa_n} \leq \sqrt{n/2}, \, \bigl| \kR^{G_{\kappa_n}}_{\theta n} \cap (\Lambda_1 \setminus \Lambda_0) \bigr|<c_1 n \right) \leq  \underset{ |\Gamma| \geq c_0 n }{\max_{\Gamma\subset \Lambda_0}} \pp \Bigl( \bigl| \kR^{\Gamma}_{\theta n}  \cap (\Lambda_1 \setminus \Lambda_0) \bigr| < c_1n \Bigr).
\end{align*}
Lemma~\ref{rnab} implies that for some $c_1=c_1(c_0)$ suitably small, the last probability is bounded by $2^{-1}\exp(-c_1n)$, and hence
\begin{align}\label{eq:EEE1}
    \pp(\kE_1) \geq 1- \pp \left(D_{\kappa_n} >\sqrt{\tfrac{n}{2}}\right)-\frac{1}{2}\exp(-c_1n).
\end{align}
Therefore, once we can find a suitable constant $c_0$ such that
\begin{align}\label{dhn}
    \pp \left(D_{\kappa_n(c_0)}>\sqrt{\tfrac{n}{2}}\right) \leq \frac{1}{2}\exp(-c_1n)
\end{align}
(see \eqref{eq:kapppa} for the definition of $\kappa_n(c_0)$), Proposition~\ref{propk3d}~(a) immediately follows.
To prove \eqref{dhn}, we use Markov's inequality:
for any $\alpha>0$,
\begin{align}\label{dkn}
\begin{split}
    \pp \left(D_{\kappa_n} > \sqrt{\tfrac{n}{2}} \right)
    &\leq \exp \left(\tfrac{-\alpha}{\sqrt{2}}n\right) \E \left(\exp[\alpha \sqrt{n}D_{\kappa_n}]\right)\\
    &= \exp \left(\tfrac{-\alpha}{\sqrt{2}}n\right) \E \left(\exp \left[\alpha \sqrt{n} \big(\sqrt{\hsct_0}+\ldots+\sqrt{\hsct_{\kappa_n-1}} \big)\right]\right).
\end{split}
\end{align}
Applying Lemma~\ref{lem6.4} to $\hsct_k$ with $G=G_k$ (note that $|G_k|=2^k$), we have for any $\alpha \in (\tfrac{1}{\sqrt{n}},c)$ with the constant $c$ appearing in Lemma~\ref{lem6.4},
 \ben{ \label{hnk}
\E\Big(\exp \big(\alpha \sqrt{n\hsct_k} \big) \, \big | \, G_k\Big) \leq H(n,k),
}
where
\begin{align*}
    &H(n,k)= n^d\exp \left(\alpha \sqrt{n 2^k/c}+\alpha^2 n/(c2^{k(d-2)/d})+ \alpha \sqrt{n \gamma_{n,k}} \right),\\
    &\gamma_{n,k} = \min \left \{ \frac{2^{2k}}{c}, \left( \frac{2\alpha}{c} \right)^{2d-2} \frac{n^{d-1}}{2^{k(2d-4)}} \right\}.
\end{align*}
There exists a positive constant $\beta=\beta(c,d)$ such that
\be{
\sum_{k=1}^{\kappa_n(c_0)} \sqrt{n 2^k/c} \leq \beta \sqrt{c_0}n, \qquad  \sum_{k=1}^{\kappa_n(c_0)} n/(c 2^{k(d-2)/d}) \leq \beta n.
}
Let us define
\[
\ell_n= \lfloor \log_2(\nu \sqrt{n}) \rfloor, \quad \nu = 2 \alpha c^{-\tfrac{2d-3}{2d-2}}.
\]
Notice that if $k \leq \ell_n$ then $2^{2k}/c$ is the smaller term in the definition of $\gamma_{n,k}$ and is the larger term when $k \geq \ell_n +1$. Hence,
\be{
\forall \, k \leq \ell_n: \quad \gamma_{n,k} = \gamma_{n, \ell_n} 2^{2(k-\ell_n)}, \qquad  \forall k  \geq  \ell_n +1: \quad \gamma_{n,k} =\gamma_{n,\ell_n+1} 2^{(2d-4)(\ell_n+1-k)}. 
}
Note also that $\gamma_{n, \ell_n+1} \leq 4 \gamma_{n, \ell_n}$ and $\sqrt{\gamma_{n, \ell_n}} \leq \nu \sqrt{n/c}$. Thus
\be{
\sum_{k=1}^{\kappa_n(c_0)} \sqrt{\gamma_{n,k}}  \leq \sqrt{\gamma_{n, \ell_n}}   \sum_{j \geq 0}2^{-j}+ \sqrt{\gamma_{n, \ell_n+1}}  \sum_{j\geq 0}2^{-(d-2)j} \leq 6 \sqrt{\gamma_{n, \ell_n}}    \leq \alpha \lambda \sqrt{n},
}
for some constant $\lambda=\lambda(c,d)\in (0,\infty)$. Therefore, by choosing $\alpha=\alpha(\beta, \lambda) \in (0,c)$ and $c_0=c_0(\beta, \lambda) >0$ suitably small, if $n$ is large enough, then
\be{
\prod_{k=1}^{\kappa_n(c_0)} H(n,k) \leq n^{d \kappa_n(c_0)} \exp\left(\alpha \beta \sqrt{c_0} n + \alpha^2 \beta n + \alpha^2 \lambda n  \right) \leq \exp(\alpha n/4).
}
This estimate together with \eqref{dkn} and \eqref{hnk} implies
\be{
 \pp \left(D_{\kappa_n(c_0)} > \sqrt{\tfrac{n}{2}} \right)  \leq \exp(-\alpha n/4),
}
and \eqref{dhn} follows.\qed

\subsection{Proof of Proposition~\ref{tln} (ii)--(iv)}\label{subsect:2-4}
This subsection provides the proof of parts (ii)--(iv) in Proposition \ref{tln}.

\begin{proof}[Proof of Proposition~\ref{tln}~(ii)]
Assume $d \geq 3$.
Given $M, \delta, C$, our task is to find $K=2^k$ with $k\in \N$, such that for any connected set $A$ containing $0$ with $|A| \geq K$,
\ben{
	\pp(\mathrm{B}^A_{\delta n} \subset \Lambda(M\sqrt{n})) \leq \exp(-Cn).
}
For this purpose, first, in the generation $0$, let $G_0:=A$ and set
\[ \sct_0:=\sct_{G_0}= \inf \Big \{ t \in \N: \Big| \bigcup_{x \in G_0} \kR^x_t \Big| \geq 4 |G_0| \Big\}. \]
Then, by the same procedure as in the proof of Proposition~\ref{propk3d}~(a) with $A_t^G=\bigcup_{x \in G} \mathcal{R}_t^x$ (see Section~\ref{subsubsect:propk3da}), we can construct inductively the sequence $(\sct_i,G_i)_{i \geq 1}$, i.e., for each $i \in \N$, $|G_i| =K 2^i$ and
\[ \sct_i:=\sct_{G_i}= \inf \Big \{ t \in \N: \Big| \bigcup_{x \in G_i} \kR^x_t \Big| \geq 4 |G_i| \Big\}. \]
Observe that for any $\ell \geq 1$, one has ${\rm T}(A,x) \leq \sum_{i=0}^ \ell \sct_i$ for all $x \in G_\ell$. Therefore,  if $\mathrm{B}^A_{\delta n} \subset \Lambda(M\sqrt{n}) $ and $ \sum_{i=0}^\ell \sct_i \leq \delta n/2$, then
\be{
 G_\ell \subset \Lambda(M \sqrt{n}), \qquad t(x, \Lambda(M\sqrt{n})^c) \geq \delta n/2 \,\, \forall \, x \in G_\ell.
}
Therefore, for all $\ell \geq 1$
\begin{align*}
\begin{split}
    &\pp \left( \mathrm{B}^A_{\delta n} \subset \Lambda(M\sqrt{n})  \right)\\
    &\leq \pp \left( G_\ell \subset \Lambda(M \sqrt{n}), \, t(x, \Lambda(M\sqrt{n})^c) \geq \delta n/2 \,\, \forall \, x \in G_\ell  \right) + \pp \left(  \sum_{i=0}^\ell \sct_i \geq \delta n/2  \right)\\
    &\leq  \left( \max_{x \in \Lambda(M\sqrt{n})} \pp\left(t(x,\Lambda(M\sqrt{n})^c) \geq \delta n/2\right) \right)^{K 2^\ell} + \pp \left(  \sum_{i=0}^\ell \sct_i \geq \delta n/2  \right).
   \end{split}
\end{align*}
Here, in the second inequality, we have used $|G_\ell| = K 2^\ell$ and the fact that the random walks $(S^x_\cdot)_{x \in G_\ell}$ are independent of $G_\ell$ and of each other by the construction of the sequence $(\sct_i, G_i)_{i\geq 0}$.  There exists a positive constant $a=a(\delta,M)$ such that $  \pp\left(t(x,\Lambda(M\sqrt{n})^c) \geq \delta n/2\right) \leq \exp(-a)$ for all $x \in \Lambda(M\sqrt{n})$. Therefore, 
\ben{
\label{badn}
 \pp \left( \mathrm{B}^A_{\delta n} \subset \Lambda(M\sqrt{n})  \right) \leq  \exp(-a K 2^\ell) + \pp \left(  \sum_{i=0}^\ell \sct_i \geq \delta n/2  \right).
}
 On the other hand, for any $C>0$
\ben{ \label{pbhn}
\pp \left( \sum_{i=0}^{\ell} \sct_i \geq \delta n/2\right) \leq \exp(-2Cn) \E \left[ \exp \left( \sum_{i=0}^{\ell} 4C \sct_i/\delta \right) \right].
}
Let $c$ be the constant as in Lemma~\ref{lte}. Then using this lemma,  for any $i \geq 0$
\begin{align}\label{e2cti}
\begin{split}
    &\E \left[\exp \left( 4C \sct_i/\delta  \right) \mid G_i \right] \\
    & \leq \exp(4C|G_i|^{2/d}/(c\delta)) + \E \left[\exp \left( 4C \sct_i/\delta  \right) \mathbf{1} (\sct_i \geq |G_i|^{2/d}/c  ) \mid G_i \right] \\
    &\leq \exp(4C|G_i|^{2/d}/(c\delta)) + \sum_{j \geq c^{-1}|G_i|^{2/d}} \exp(4Cj/\delta-cj|G_i|^{\tfrac{d-2}{d}})\\
	&\leq \exp(4C|G_i|^{2/d}/(c\delta)) + 1 \leq 2 \exp(4C|G_i|^{2/d}/(c\delta)),
\end{split}
\end{align}
provided that
\be{
	8C/(c \delta) \leq  |G_i|^{\tfrac{d-2}{d}}.
}
Since $|G_i|=2^i |G_0|=K2^i$, we can take $K=2^k$ large enough such that the above condition holds for all $i\geq 0$.
Thus, due to \eqref{badn}, \eqref{pbhn} and \eqref{e2cti}, there exists a constant $C'=C'(C,\delta,c)>0$ such that
\begin{align*}
    &\pp \left( \mathrm{B}^A_{\delta n} \subset \Lambda(M\sqrt{n})  \right)\\
    &\leq \exp(-aK2^\ell)+\exp(-2Cn) 2^{\ell +1} \exp \left(\dfrac{4C}{c \delta} \sum_{i=0}^{\ell  }  (K 2^{i})^{2/d} \right) \\
	&\leq \exp(-aK2^\ell)+\exp(-2C n) 2^{\ell +1}\exp(C' 2^{2\ell/d}).
\end{align*}
Note that since $d\geq 3$ (and thus $2/d <1$), the second term in the right-hand  side above is bounded by $\exp(-3Cn/2)$ as long as $2^\ell$ is of order $n$. 

Consequently, taking $\ell=\lceil \log_2(2 C n/(aK)) \rceil$ proves that
\begin{align*}
    \pp \left( \mathrm{B}^A_{\delta n} \subset \Lambda(M\sqrt{n})  \right) \leq \exp(-Cn),
\end{align*}
and the proof is complete.
\end{proof}

\begin{proof}[Proof of Proposition~\ref{tln}~(iii)]
We assume $d=2$, and the proof is essentially the same as that of Proposition~\ref{tln} (i) (which treats the case where $d \geq 3$).
Using the arguments for the proof of Proposition~\ref{tln}~(i), the proof of (iii) is reduced to study the spread of active frogs at early stages, that is, to prove that there exists a positive constant $c_1$ such that
\begin{equation} \label{bn22}
\pp(\kE_1)=
\pp \bigl( \bigl| \{y \in A_1: {\rm T}_{\Lambda_0}(0,y) \leq n(\tfrac{1}{2}+\theta)\} \bigr| \geq c_1n \bigr)\geq 1-e^{-c_1n/\log n}.
\end{equation}
Furthermore, since Lemma~\ref{rnab} is used when we obtain \eqref{eq:EEE1}, in $d=2$, the following inequality holds as the counterpart of \eqref{eq:EEE1}:
let $c_0>0$ be a constant chosen later, and for some constant $c_1=c_1(c_0)$,
\ben{ \label{bn2kn}
\pp (\kE_1)\geq 1-\pp \left(D_{\kappa_n(c_0)} > \sqrt{\tfrac{n}{2}} \right)-\frac{1}{2}\exp(-c_1n/\log n)
}
(see \eqref{eq:kapppa} for the definition of $\kappa_n(c_0)$).
Moreover, by Markov's inequality  for any $\beta >0$,
\begin{align*}
    \pp \left(D_{\kappa_n(c_0)} > \sqrt{\tfrac{n}{2}} \right)
\leq \exp \left(\tfrac{-\beta n}{\sqrt{2} \log n}\right) \E \left(\exp \left[\beta \sqrt{n} \Big(\sqrt{\hsct_0}+\ldots+\sqrt{\hsct_{\kappa_n-1}} \Big)/\log n\right]\right).
\end{align*}
Applying Lemma~\ref{lem6.4} to $\hsct_k$ with $d=2$ and $|G_k|=2^k$, we have for any $\alpha\in (\tfrac{1}{\sqrt{n}},c)$ with the constant $c$ appearing in Lemma~\ref{lem6.4},
\begin{align*}
    \E\Big(\exp \big(\alpha \sqrt{n\hsct_k} \big) \, \big | \, G_k\Big) &\leq n^2 \exp \left(\alpha \sqrt{n |G_k|/c} +  \alpha^2 n/c + \alpha \sqrt{n} \min \big \{ \tfrac{|G_k|^2}{c}, \tfrac{(2\alpha)^2n}{c^2}  \big \}^{1/2}  \right)
    \\
  &\leq  n^2 \exp \left(\alpha \sqrt{n |G_k|/c} + 3 \alpha^2 n/c \right)=:     M(n,k).
\end{align*}
Letting $\alpha =\beta/\log n$ and recalling that $\kappa_n(c_0)=\log_2 \lfloor c_0 n \rfloor$ and $|G_k|=2^k$,  we have
\al{
\prod_{k=0}^{\kappa_n(c_0)} M(n,k) &\leq  n^{4 \kappa_n(c_0)}\exp \left( \frac{4\beta n\sqrt{c_0/c}}{\log n}  + \frac{ 6\beta^2 n \log_2(c_0)}{c\log n} \right)\\
& \leq \exp \left( \beta n/ (4 \log n)\right),
}
by taking  $\beta=\beta (c)$ and $c_0=c_0(c)$ suitably small. 
Combining the last three displays, one has
\be{
\pp \left(D_{\kappa_n} > \sqrt{\tfrac{n}{2}} \right) \leq \exp \left( -\beta n/ (4 \log n)\right),
}
which together with \eqref{bn2kn} implies the desired result \eqref{bn22}.
\end{proof}

\begin{proof}[Proof of Proposition~\ref{tln}~(iv)]
We now take $G_0:=\{0\}$ and repeat the same argument as in the proof of Proposition~\ref{tln}~(ii).
Since $|G_0|=1$, \eqref{badn} implies that for any $\ell \geq 1$, 
\al{
\pp \left( \mathrm{B}_{\delta n} \subset \Lambda(M\sqrt{n}) \right)  \leq  \exp(-a 2^\ell) + \pp \left(  \sum_{i=0}^\ell \sct_i \geq \delta n/2  \right),
}
where $a=a(\delta,M)$ is a positive constant.
Markov's inequality shows that for any $\alpha>0$,
\begin{align*}
\pp \left( \sum_{i=0}^{\ell} \sct_i \geq \delta n/2\right) \leq \exp(-\alpha n) \E \left[ \exp \left( \sum_{i=0}^{\ell} (2\alpha/\delta) \sct_i \right) \right].
\end{align*}
Moreover, we note that $d=2$ and use an argument similar to \eqref{e2cti}: if $(2\alpha/\delta) \leq c|G_0|/2 =c/2$ with $c$ a positive constant appearing in Lemma \ref{lte}, then for all $i\geq 0$,
\begin{align*}
    \E \left[\exp \left( {(2\alpha/\delta) \sct_i}  \right) \mid G_i \right] \leq 2 \exp(2\alpha|G_i|/( c \delta)) = 2 \exp(\alpha2^{i+1}/(c\delta)).
\end{align*}
Combining the last three estimates and taking $\ell = \lfloor \log_2(c\delta n/8) \rfloor$ and $\alpha=c\delta/4$, we arrive at
 \al{
\pp \left( \mathrm{B}_{\delta n} \subset \Lambda(M\sqrt{n}) \right)  \leq  \exp(-a 2^\ell) + 2\exp(-\alpha n/2) \leq \exp(-c'n),
}
where $c'=c'(\delta, c,M)$ is a positive constant.
\end{proof}

\section{Renormalization and covering process: Proof of Proposition~\ref{tuld} }\label{sect:LB_uld3}
We first recall the statement of  Proposition~\ref{tuld}: there exists $M \in \N$ such that for any $\e>0$, there exists $c=c(\e,M)>0$ such that for all $x \in \Z^d$, $t, n \in \N$ satisfying $|x|_1 \leq \sqrt{n}$ and $(1+\e)\mu(x) \leq t \leq n$, one has
\begin{align}\label{theo3.1}
	\pp(\T_{\Lambda(M\sqrt{n})}(0,x) \geq  t) \leq \exp(-c\mathfrak{s}_d(t)),
\end{align}
where 
\begin{align*}
\mathfrak{s}_d(t)= t \1{d\geq 3} + \tfrac{t}{\log t} \1{d=2}.
\end{align*}
The proof of \eqref{theo3.1} reduces to estimating the key event  $\{{\rm T} _{\Lambda(KN)}(0, N \eb_1) \geq t\}$ with $K$ a constant and $t \in [KN,N^2]$. Lemma~\ref{lem:seoset} plays a central role in this analysis. Roughly speaking, Lemma~\ref{lem:seoset} implies that if the above event occurs, then with overwhelming probability, we can identify a sequence of disjoint bad regions, say $(\Lambda(x_i, L_i))_{i=1}^\ell$, within which the transmission is extremely slow. More precisely, conditioned on an overwhelming event
\[
\{{\rm T} _{\Lambda(KN)}(0, N \eb_1) \geq t\}  \subset \bigcap_{i=1}^\ell \{ \exists \, y, z \in \Lambda(x_i, L_i):  {\rm T} _{\Lambda(x_i,L_i)}(y,z) \geq L_i^2 \},
\]      
with the additional constraint that
  $\sum_{i=1}^\ell L_i^2 \geq t/100$.  By Proposition~\ref{tln}, the probability of each event in the intersection is smaller than $\exp(-c \mathfrak{s}_d(L_i^2))$ with $c$ a positive constant. Then, using the independence of restricted passage times over disjoint regions, we deduce the desired estimate $\pp({\rm T} _{\Lambda(KN)}(0, N \eb_1) \geq t) \leq \exp(-c'\mathfrak{s}_d(t))$ with $c'$ a suitable constant.  

The bad regions $(\Lambda(x_i, L_i))_{i=1}^\ell$ are identified through a carefully tuned renormalization and covering process. First, by a renormalization procedure, we construct a sequence of consecutive boxes connecting $0$ and $N\eb_1$, which consists mostly of good boxes but includes some bad ones; see Lemma~\ref{lem:mdist} and the event $\kE_1$ in \eqref{poe1}.
 Then, the bad regions $(\Lambda(x_i, L_i))_{i=1}^\ell$
are revealed by applying a covering process to these bad boxes; further details can be found in Subsection~\ref{sect:LB_uld2}.

\subsection{Key estimates and proof of Proposition~\ref{tuld}}

First, we mention a generalized version of \cite[Theorem~{1.1}]{Kub19b}.

\begin{lem}\label{prop:kubota19b} 
Let $d\geq 2$. Assume that  there exists $C>0$ such that for all $x \in \Z^d$ and $t \geq C|x|_1$,
\begin{align} \label{assldp}
	\pp\left({\rm T} (0,x) \geq t \right) \leq C\exp(-g(t)/C),
\end{align}
where  $g$ is a function satisfying
\begin{align} \label{assog}
	\limsup_{t\to\infty} \frac{\log g(t)}{\log t} \in (0,1],
\end{align}
Then, for any $\varepsilon >0$, there exists $C'=C'(C,\varepsilon) >1$ such that for all $x \in \Z^d$,
\begin{align*}
	\pp\left({\rm T} (0,x) \geq (1+\varepsilon) \mu(x) \right) \leq C'\exp(-g(|x|_1)/C').
\end{align*}
\end{lem}

\begin{rem}
In the original version ~\cite[Theorem~{1.1}]{Kub19b}, the author considers the stretched large deviation corresponding to the case $g(t)=t^\alpha$ with $\alpha \in (0,1)$.  The proof for general functions is essentially the same.  For the convenience of the reader, let us briefly explain the proof (see \cite[Section~3]{Kub19b} for more details). 

Fix $\e>0$. Then, by using a percolation argument (in particular, an upper tail estimate for the chemical distance in a supercritical site percolation on $\Z^d$), there exists a constant $A>0$ such that for any $x \in \Z^d$,
\begin{align*}
	\mathbb{P}\biggl( \min_{\substack{y \in \Lambda(\e |x|_1/A)\\z \in \Lambda(x,\e |x|_1/A)}}{\rm T}(y,z) \geq (1+\e)\mu(x) \biggr) \leq A\exp(-|x|_1/A).
\end{align*}
This, combined with the triangle inequality and \eqref{assldp}, implies that  for all $x \in \Z^d$,
\begin{align*}
	&\mathbb{P}({\rm T}(0,x) \geq (1+\e)\mu(x))\\
	&\leq 2\mathbb{P}\biggl( \max_{y \in \Lambda(\e |x|_1/A)}{\rm T}(0,y) \geq \e|x|_1 \biggr)+\mathbb{P}\biggl( \min_{\substack{y \in \Lambda(\e |x|_1)\\z \in \Lambda(x,\e |x|_1)}}{\rm T}(y,z) \geq \Bigl( 1+\frac{\e}{3} \Bigr)\mu(x) \biggr)\\
	&\leq 2C\exp\bigl( -g(\e |x|_1)/C \bigr)+A\exp(-|x|_1/A) \leq C'\exp\bigl( -g(|x|_1)/C' \bigr),
\end{align*}
where $C' =C'(\varepsilon,C,A)$ is a positive constant.  Note that in the last inequality, we have used the assumption \eqref{assog} of the function $g$. For more details, we refer to \cite[Section~3]{Kub19b}.
\end{rem}

\begin{lem}\label{theo3.1 rough estimate}
 There exists a constant $C>0$, such that for all $z \in \Z^d$ with $|z|_1 \geq 2$ and for all $t\geq C|z|_1$ 
   	  \be{
\pp\left(\T_{\Lambda(C(\sqrt{t}+|z|_1))}(0,z) > t \right) \leq C \exp(-\mathfrak{s}_d(t)/C).
          }
\end{lem}

\noindent {\bf Proof of Proposition~\ref{tuld}  assuming Lemma \ref{theo3.1 rough estimate}}.
Let $\e>0$ and $x \in \Z^d$ with $2 \leq |x|_1 \leq \sqrt{n}$ and $(1+\varepsilon)\mu(x) \leq t\leq n$.
In addition, let $C$ be as in Lemma \ref{theo3.1 rough estimate} and take $M\geq 2C$.
Due to $M\sqrt{n}\geq C(\sqrt{t}+|x|_1)$ and Lemma~\ref{theo3.1 rough estimate}, we have for any $t\geq C|x|_1$,
 \ben{\label{eq: large t}
\pp\left(\T_{\Lambda(M\sqrt{n})}(0,x) > t \right) \leq \pp\left(\T_{\Lambda(C(\sqrt{t}+|x|_1))}(0,x) > t \right) \leq C \exp(-\mathfrak{s}_d(t)/C).
          }
Moreover, if $(1+\varepsilon) \mu(x) \leq t \leq C|x|_1$, then using  \eqref{eq: large t} and Lemma \ref{prop:kubota19b}
 \bea{
&&\pp(\T_{\Lambda(M\sqrt{n})}(0,x) \geq t ) \\
 & \leq& \pp(\T_{\Lambda(M\sqrt{n})}(0,x) \geq (1+\varepsilon)\mu(x) )\\ 
 &\leq & \pp( (1+\e) \mu (x) \leq  {\rm T}(0,x) \leq C|x|_1 ) + \pp(\T_{\Lambda(M\sqrt{n})}(0,x) > C|x|_1) \\
  &\leq&  C'\exp(-\mathfrak{s}_d(|x|_1)/C') + C\exp(-\mathfrak{s}_d(C |x|_1)/C)\\
  &\leq& C''\exp(-\mathfrak{s}_d(t)/C''),
}
where $C'=C'(C,\varepsilon)$ is a positive constant as in Lemma \ref{prop:kubota19b},  $C''=(C+C')C$, and  we used the fact that $\T_{\Lambda(M\sqrt{n})}(0,x)={\rm T}(0,x)$ if ${\rm T}(0,x) \leq C|x|_1$ and $C|x|_1 \leq M\sqrt{n}$ in the second inequality.
In conclusion, \eqref{theo3.1} holds. \qed

\begin{lem} \label{lem:seoset} There exist  positive constants $K$ and $c$  such that for $N^2 \geq t \geq KN$, 
	\ben{
		\pp(\kE) \geq 1- \exp(-c\mathfrak{s}_d(t)),
	}
	where $\kE$ is the event that we can find a sequence  $(x_i,L_i)_{i=1}^{\ell}\subset   \Z^d\times \N$ satisfying
	\begin{itemize}
		\item [(a)] $\ell \leq \sqrt{t} (\log t)^2$;
		\item[(b)]   $L_i\leq N$ for all $i=1, \ldots \ell$, and 
		$$	{\rm T}_{\Lambda(KN)}(0,N\eb_1) \leq \frac{t}{2} + 50 \sum^{\ell}_{i=1} L_i^2;$$
		\item[(c)] $ x_i\in \Lambda(2N)$ for all $i=1, \ldots, \ell$,  and   $ \Lambda(x_i,L_i)\cap \Lambda(  x_j,L_j)=\emptyset$ if $i\neq j$;
		\item [(d)] There exist $y_i,z_i\in \Lambda( x_i,L_i)$ such that ${\rm T}_{\Lambda( x_i,L_i)}(y_i,z_i)\geq L_i^2$ for all $i=1,\ldots, \ell$.
	\end{itemize}	
\end{lem}

\noindent \textbf{Proof of Lemma \ref{theo3.1 rough estimate} assuming Lemma \ref{lem:seoset}}. For simplicity, we assume that  $z=N \mathbf{e}_1$. If $t \geq N^2$ then Lemma \ref{theo3.1 rough estimate} directly follows from Proposition \ref{tln}. Hence, it suffices to prove that there exist  positive constants $K$ and $\kappa$, such that for all $N$ sufficiently large and $t \in [KN,N^2]$,
\ben{ \label{goal2d}
	\pp\left({\rm T} _{\Lambda(KN)}(0, N \mathbf{e}_1) \geq t \right) \leq  \exp\left(-\kappa \mathfrak{s}_d(t) \right).
}
 Let $K$, $c$ and $\kE$  be the constants and event  as in Lemma \ref{lem:seoset}.  
On the event $\kE \cap \{ {\rm T}_{\Lambda(KN)}(0,N\eb_1) \geq t\}$, we can find a sequence $(x_i,L_i)_{i=1}^{\ell}\subset \Z^d\times \N$ satisfying (a), (c), (d) and 
\begin{itemize}
	\item [(b')]  $L_i\leq N$ for all $i=1, \ldots \ell$, and 
	$$	\sum^{\ell}_{i=1} L_i^2\geq \frac{t}{100}. $$
\end{itemize}
Given $\ell\in\N$, let us denote 
\[\mathscr{L}_\ell :=\{ (x_i,L_i)_{i=1}^{\ell} \in \Z^d\times \N \textrm{ satisfying (b'), (c)}\}. \]
Then, with some $c'$ and $c''$ positive constants, we have
\bean{ \label{ldoe12}
	&&\pp\left( \kE \cap \{ {\rm T}_{\Lambda(KN)}(0,N\eb_1) \geq t\}\right) \notag \\
	& \leq &  \pp \Big(\exists \, \ell \leq \sqrt{t} (\log t)^2, \, \exists \, (x_i,L_i)_{i=1}^{\ell} \in \scl_\ell: \notag \\
	&&  \hspace{3 cm} \forall \, i\in  [\ell],\, \exists \, y_i,z_i\in  \Lambda(x_i,L_i),  \, {\rm T}_{ \Lambda(x_i,L_i)}(y_i,z_i)\geq L_i^2   \Big) \notag \\
	& \leq & \sum_{\ell=1}^{\lf\sqrt{t} (\log t)^2\rf} \, \sum_{(x_i,L_i)_{i=1}^{\ell} \in \scl_\ell} \, \sum_{y_i, z_i \in \Lambda( x_i,L_i) } \pp \left(\forall i\in [\ell],~{\rm T}_{ \Lambda(x_i,L_i)}(y_i,z_i)\geq L_i^2 \right) \notag\\
	&\leq&  \sum_{\ell=1}^{\lf\sqrt{t} (\log t)^2\rf} \,  \sum_{(x_i,L_i)_{i=1}^{\ell} \in \scl_\ell} \, \sum_{y_i, z_i \in \Lambda( x_i,L_i) } \exp{ \left(-c'\sum_{i=1}^{\ell}\mathfrak{s}_d(L_i^2) \right)} \notag \\
	&\leq& \sum_{\ell=1}^{\lf\sqrt{t} (\log t)^2\rf}\,  \sum_{(x_i,L_i)_{i=1}^{\ell} \in \scl_\ell} \, \sum_{y_i, z_i \in \Lambda( x_i,L_i) } \exp(-c''\mathfrak{s}_d(t)), \notag 
}
where, for the third inequality we have used Proposition \ref{tln}  with the help of the condition (c) and the  property that ${\rm T}_A$ and ${\rm T}_B$ are independent if $A \cap B =\emptyset$; and for the last inequality, we used the condition (b') and that $L_i \leq N \leq t$ for all $i=1, \ldots, \ell$.   Moreover,   since $L_i\leq N \leq t$ by (b') and $x_i \in \Lambda(2N)$,
\be{ 
	\sum_{\ell=1}^{\lf\sqrt{t} (\log t)^2\rf}\,  \sum_{(x_i,L_i)_{i=1}^{\ell} \in \scl_\ell} \, \sum_{y_i, z_i \in \Lambda( x_i,L_i) }   1  \leq \exp \left( \sqrt{t} (\log t)^4 \right).
	}
It follows from the last two displays that for $N$ sufficiently large 
\be{
	\pp\left(\kE \cap \{ {\rm T}_{\Lambda(KN)}(0,N\eb_1) \geq t\}\right) \leq \exp(-c''\mathfrak{s}_d(t)/2).
}
Combining this with Lemma~\ref{lem:seoset},  we get the desired result \eqref{goal2d}.  \hfill $\square$

\subsection{Renormalization and covering process: proof of Lemma \ref{lem:seoset} }  \label{sec:reno}
\subsubsection{Renormalization scheme}
Let $M_* ,C_* \geq 4$ be large integers chosen later.
We recall that $\Lambda(M_*  u,M_* ) : = M_*  u + [-M_*  ,M_*  ]^d$  for $u \in \Z^d $. 
We consider the renormalization based on the tessellation of $\Z^d$ into boxes of size $M_* $:
  \be{
  \Z^d = \bigcup_{ u\in \Z^d} \Lambda(M_*  u,M_*  ).
   }
An $M_*$-box $\Lambda(M_*  u,M_* )$ is called {\bf open} if and only if
\begin{align*}
   \tT(\Lambda(M_*  u,2M_* )) \leq  C_*   M_* , \quad \textrm{where } \tT(A) := \max_{u,v \in A} {\rm T}(u,v). 
\end{align*}
These open boxes induce the following site-percolation. A vertex $u \in \Z^d$ is called {\it open} if the $M_*$-box $\Lambda(M_*  u,M_* )$ is open. We denote the associated percolation by $\kP$. Note that here $\kP$ depends on $C_* $ and $M_* $, but for simplicity of notation, we omit the dependence of $\kP$ on parameters. 

By \cite[Proposition~{2.4}]{Kub19b}, there exist $A \geq 1$ and $\alpha \in (0,1]$ such that for all $x \in \Z^d$ and $t \geq A|x|_\infty$,
\begin{align*}\label{eq:upper_tail}
	\mathbb{P}({\rm T}(0,x) \geq t) \leq A\exp(-t^\alpha/A).
\end{align*}
Consequently, for $C_*$ sufficiently large,
  \begin{equation} \label{pminf}
   \lim_{M_*\rightarrow \infty}\pp(u \textrm{ is open}) = 1.
  \end{equation}
Furthermore, note that the event $\{u \textrm{ is open} \}$ depends only on SRWs $S^y_{\cdot}$, $y \in \Z^d$ with $|y-M_*  u|_1\leq C_*  M_* $.
In particular, 
  \ben{ \label{locd}
  \textrm{ if  $|u-v|_1\geq 4C_* $, then the two events $\{u \textrm{ is open}\}$ and $\{v\textrm{ is open}\}$ are independent. }
}

\begin{lem} \label{lem:reno} Let $d\geq 2$.
The following claims  hold:
	\begin{itemize}
		\item[(i)] For any $k \geq 2$ and  $u, v \in \Lambda(k) \subset \Z^d$,
		    \[ \T_{\Lambda(M_* (k+C_* ))}(M_*  u, M_*  v) \leq  C_*   M_*   d_{\kP_k} (u,v), \]
        where $d_{\kP_k}$ is the graph distance (chemical distance) in  $\kP$ restricted in $\Lambda(k)$.    
		\item [(ii)] For any $\e>0$, there exist positive constants $C_* ,M_* $ such that  the induced percolation $\kP$ stochastically dominates the independent Bernoulli site-percolation with parameter $1-\e$.
        	\end{itemize}
\end{lem}
\begin{proof} We start with (i). If $d_{\kP_k}(u,v) = \infty$ then the inequality is trivial. Suppose that $l:=d_{\kP_k}(u,v) <\infty$ and $(u_i)_{i=0}^l$ is an open path in $\Lambda(k)$ connecting $u$ and $v$, i.e., ${u_0=u}$, $u_l =v$ and $|u_i-u_{i-1}|_1=1$ for $i=1, \ldots, l$, and all $u_i$'s are open. In particular,   ${\rm T}(M_*  u_{i-1},M_*  u_{i}) \leq C_*   M_* $ for all $i=1, \ldots, l$ and thus $\T_{\Lambda(M_*  u_{i-1},C_*  M_* )}(M_*  u_{i-1},M_*  u_{i}) \leq C_*   M_* $.  Hence, by the triangle inequality and the fact that $\Lambda(M_*  u_{i-1},C_*  M_* ) \subset \Lambda(M_* (k+C_* ))$ for all $i=1, \ldots, l$, we have
$$\T_{\Lambda(M_* (k+C_* ))}(M_*  u,M_*  v) \leq \sum_{i=1}^{{l}} \T_{\Lambda(M_*  u_{i-1},C_*  M_* )}(M_*  u_{i-1},M_*  u_{i}) \leq  C_*  M_*  d_{\kP_k}(u,v).$$
Part (ii) follows from \eqref{pminf}, \eqref{locd} and the stochastic domination theorem for locally dependent percolation (see for instance \cite[Theorem~{7.65}]{Gri99_book} or \cite[Theorem~{B26}]{Lig99_book}).
\end{proof}

\begin{lem} \label{lem:mdist} Let $\pp_p$ denote the probability measure of the standard site percolation on $\Z^d$ with density $p$.
There exist positive constants $\e$ and $c$, such that for any $p>1-\e$ and $m$ sufficiently large and $m^2 \geq r \geq m$,
\begin{align*}
&	\pp_p \Bigl( \exists \text{a $\Z^d$-path $\gamma \subset \Lambda(2m)$ from $0$ to $m \eb_1$, } |{\rm Cl}(\gamma)| \leq \sqrt{r} (\log r)^2, \,\, |\gamma| \leq m/c \Bigr) \\
	& \geq  1-\exp(-cr),
\end{align*}
where
\begin{align*}
	{\rm Cl}(\gamma) :=\{u\in \gamma: u \textrm{ is closed} \}.
\end{align*}
\end{lem}
The proof of this lemma is presented in Appendix.

\subsubsection{Covering process and proof of Lemma \ref{lem:seoset}}\label{sect:LB_uld2}
Let $\varepsilon$ be the constant as in Lemma \ref{lem:mdist}. Then take $C_*$, $M_*$ the large constants satisfying Lemma \ref{lem:reno} with this $\varepsilon$. Recall the renormalization process in the previous subsection and the associated site percolation  $\kP$. We set
\[
w_N=\lfloor N/M_*  \rfloor \eb_1.
\]
Note that $N \eb_1 \in \Lambda(M_*  w_N,M_* )$ and by Lemma \ref{lem:reno},  $\kP$  stochastically dominates the independent site percolation with density $1-\e$.
Therefore,  Lemmas \ref{lem:reno} and  \ref{lem:mdist}  guarantee that there exists a positive constant $c$ such that
for $N/M_* \leq r \leq (N/M_*)^2$, 
\bean{ \label{poe1}
	\pp(\kE_1) \geq 1 - \exp(-cr)
} 
where
\begin{align*}
	\kE_1:=\left\{
	\begin{minipage}{28em}
		$\exists$ a $\Z^d$-path $\gamma =(v_i)_{i=1}^{l} \subset \Lambda(2N/M_* )$ such that $v_1=0$, $v_l=w_N$,\\
		$l \leq N/(cM_*)$ and $|\{i \in [l]:\Lambda(M_* v_i,M_* ) \textrm{ is closed} \}|\leq \sqrt{r}(\log r)^2$
	\end{minipage}
	\right\}.
\end{align*}
Below, we fix the constants $C_*,M_*>1$ and $c\in (0,1)$, and set 
\[
K= 24 + \lceil 2C_*M_*/c \rceil.
\]
Given $t$ as in Lemma \ref{lem:seoset}, set $r:=\lfloor t/M_*^2 \rfloor$. Then, for all $N$ sufficiently large, $N/M_* \leq r \leq (N/M_*)^2$.
We aim to estimate the passage time ${\rm T} (0,N\eb_1)$ using the passage times of boxes along the path $\gamma$, as defined by the event $\kE_1$. By construction, the passage time through open boxes is well controlled. The main challenge lies in understanding the passage time through closed boxes. To address this, we introduce a \textbf{covering process} that aggregates such boxes. We define the \textit{radius of bad neighbors} of an $M_*$-box $\Lambda(M_*u,M_* )$ as  
\ben{ \label{dlb}
	\kL_{M_*}(M_*u)=\inf\{\ell \in\N:~\forall k>\ell,~\tilde{\rm T}_{\Lambda(M_*u,3k)}(\Lambda(M_*u,k))<k^2\}. 
}	
Intuitively, this radius captures the extent of the region around a box where the infection spread is rarely slow, namely, where the maximal passage time grows quadratically with distance.
We then progressively merge the regions of closed boxes, ordered by decreasing values of the radius $\kL_{M_*}$. This process results in a disjoint cover of these bad regions. The detailed construction will be presented in Lemma~\ref{lem:seoset'}.

Note that the distribution of $\kL_{M_* }$ is independent of $u$. By Proposition \ref{tln} and the union bound,
\begin{align}\label{lbl}
\begin{split}
	\pp( \kL_{M_* }(0)>\ell)
	&\leq \sum_{k \geq \ell} \pp\bigl( \tilde{\rm T}_{\Lambda(2k)}(\Lambda(k)) \geq k^2 \bigr)\\
	&\leq \sum_{k\geq \ell}\sum_{x, y \in \Lambda(k)}  \pp\bigl( {\rm T} _{\Lambda(2k)}(x,y) \geq k^2 \bigr)
	\leq \exp(-\alpha \mathfrak{s}_d(\ell^2)),
\end{split}
\end{align}
for some $\alpha= \alpha(M_*,d)>0$.
In particular, $\kL_{M_* }(M_*  u) < \infty$ almost surely for all $u\in\Z^d$. We define 
\bea{
	\kE_2 = \big \{ \kL_{M_* }(M_*  u)\leq N \, \, \textrm{ for all } u\in  \Lambda(N^2) \big \}.
}
By \eqref{lbl} and the union bound, for $N/M_* \leq r \leq (N/M_*)^2$,
\bea{
	\pp(\kE_2^c) \leq \sum_{u\in \Lambda(N^2)} \pp (\kL_{M_* }(M_*  u)>N) \leq  \exp(- \alpha \mathfrak{s}_d(r)/2).
}
 Combining this with \eqref{poe1}, we get
\ben{ \label{e1a2c}
	\pp((\kE_1\cap\kE_2)^c)\leq \exp(-cr)+\exp(-\alpha \mathfrak{s}_d(r)/2). 
}

\begin{lem} \label{lem:seoset'}
On the event $\kE_1 \cap \kE_2$, we can find a sequence  $(u_i,L_i)_{i=1}^{\ell}\subset   \Z^d\times \N$ such that $(M_*u_i,L_i)_{i=1}^{\ell}$ satisfies conditions (a)--(d) of Lemma \ref{lem:seoset}. As a consequence, Lemma \ref{lem:seoset} holds true.
\end{lem}
\begin{proof}
Suppose that $\kE_1$ occurs, i.e., there exists a $\Z^d$ path $\gamma =(v_i)_{i=1}^{l} \subset \Lambda(2N/M_*)$ starting at $v_1=0$ and ending at $v_l=w_N$, containing at most $\sqrt{r}(\log r)^2$ closed sites with $l \leq   N/(cM_*)$.  We will find a subsequence $(u_i)_{i=1}^{\ell} \subset (v_i)_{i=1}^l$ such that  $(M_*u_i)_{i=1}^{\ell}$ satisfies (a)--(d) on  $\kE_2$. Let
\begin{align*}
	\kL_i=\kL_{M_* }(M_*v_i), \qquad  J=\{i\in\{1,\cdots,l\}:~\Lambda(M_*v_i,M_* )\text{ is closed}\},
\end{align*}
where  $\kL_{M_* }(M_*  v_i)$ is defined as in \eqref{dlb}.  By the definition of $\kE_1$,
\be{
|J| \leq \sqrt{r}(\log r)^2.
}
Let 
\begin{align*}
\zeta_1=\max \, {\rm argmax}\{\kL_i:~i\in J\}, \qquad J_1= \{i\in J:~M_* v_i\notin  \Lambda_{6\kL_{\zeta_{1}}}(M_* v_{\zeta_{1}}) \}.
\end{align*}
Inductively, we define
\begin{align*}
\zeta_{m+1}=\max \, {\rm argmax}\{\kL_i:~i\in J_m\}, \qquad J_{m+1}= \{i\in J_m:~M_*v_i\notin  \Lambda_{6\kL_{\zeta_{m+1}}}(M_*v_{\zeta_{m+1}}) \}.
\end{align*}
By the definition, if $J_{m}\neq \emptyset$, then $J_{m+1} \subset J_m$ and $|J_m \setminus J_{m+1}| \geq 1$, since  $\zeta_{m+1} \in J_m \setminus J_{m+1}$. Hence, there exists $s$  such that
\ben{ \label{voa}
s \leq |J| \leq \sqrt{r}(\log r)^2, \qquad 	J_{s-1}\neq \emptyset, \qquad J_{s}= \emptyset.
}
By the definition of $(\zeta_i)_{i=1}^{s}$, the sequence $(\kL_{\zeta_i})_{i=1}^{s}$ is non-increasing, so we can define
 $$\ell:=\max\{i\leq s:~\kL_{\zeta_i}\geq M_* \}.$$ 
Observe that  if $1\leq i < j\leq \ell$ then $|M_*v_{\zeta_i}-M_*v_{\zeta_j}|_\infty > 6\kL_{\zeta_i}$. Thus, since $\kL_{\zeta_i} \geq \kL_{\zeta_j}$,
\ben{ \label{voc}
	\Lambda(M_*v_{\zeta_i},3\kL_{\zeta_{i}}) \cap  \Lambda(M_*v_{\zeta_j},3\kL_{\zeta_{j}}) = \varnothing.
} 
By the definition  of  $\kL_{\zeta_i}$ (see \eqref{dlb}), we have for all $i=1,\ldots, \ell$,
\be{
{\tilde{{\rm T}}_{\Lambda(M_*v_{\zeta_i},3\kL_{\zeta_{i}})}(\Lambda(M_*v_{\zeta_i},\kL_{\zeta_{i}}))} \geq \kL_{\zeta_i}^2,
}
and thus there exist {$y_i, z_i$} in $\Lambda(M_*v_{\zeta_i},\kL_{\zeta_{i}})$ such that
\ben{ \label{vod}
	{\rm T}_{\Lambda(M_*v_{\zeta_i},\kL_{\zeta_{i}})} (y_i,z_i) \geq {\rm T}_{\Lambda(M_*v_{\zeta_i},3\kL_{\zeta_{i}})} (y_i,z_i) \geq \kL_{\zeta_i}^2.
}
We now define for $i=1, \ldots, \ell$,
\[u_i= v_{\zeta_i}, \quad L_i=\kL_{\zeta_i},\]
and show that $(M_*u_i,y_i,z_i,L_i)_{i=1}^{\ell}$ with $y_i, z_i$ as in \eqref{vod} satisfies the conditions in Lemma \ref{lem:seoset}. Observe that the condition (a) is verified by noting that  $\ell \leq s \leq \sqrt{r}(\log r)^2 \leq \sqrt{t}(\log t)^2$. Since $u_i \in \Lambda(N/M_* )$ for all $i=1,\ldots, \ell$, the condition~{(c)} follows from \eqref{voc}. Moreover,  (d) is checked by \eqref{vod}. Now it remains to verify condition (b). \\

Note that on the event $\kE_2$,  we have $L_i \leq N$ for all $i=1,\ldots, \ell$. Hence, we only need to check the condition on the sum  $\sum_{i=1}^\ell L_i^2$. 
\begin{lem}\label{lem: connect-time}
Let $k \in \N$ and let $x_1,\ldots,x_k\in \Z^d$ and $d_1, \ldots, d_k$ be non-negative integers.
Assume that $X=\cup_{i=1}^k \Lambda(x_i,d_i)$ is connected. Then for any set $\Lambda \subset \Z^d$  
	\ben{ \label{btkx}
		\tilde{\rm T}_{\Lambda} (X) \leq \sum_{i=1}^k \tilde{\rm T}_{\Lambda}(\Lambda(x_i,d_i+1)).
	}
\end{lem}
We postpone the proof of Lemma \ref{lem: connect-time} for a moment and complete  the verification of condition~{(b)}.
Let us define 
$$\kC=\bigcup_{\substack{1 \leq j \leq l \\ j \not \in J_* } } \Lambda(M_*v_j,M_* ) \cup \bigcup_{i=1}^{\ell}\Lambda(M_*  u_i,6L_i),$$
where
$$J_*=\{ \zeta_{1}, \ldots, \zeta_{\ell} \}.$$
To apply Lemma \ref{lem: connect-time}, we first show that $\kC$ is connected. Indeed, note that $L_i=\kL_{\zeta_i} \geq M_*$ for all $i=1, \ldots,\ell$, and thus 
\be{
\bigcup_{j  \in J_* }  \Lambda(M_*v_j,M_* ) = \bigcup_{i=1}^{\ell}  \Lambda( M_*  u_i,M_* ) \subset \bigcup_{i=1}^{\ell}  \Lambda( M_*  u_i,6L_i).
}
Therefore, $\kC$ can be rewritten as 
\be{
	\kC =  \kC_1 \cup \bigcup_{i=1}^{\ell}\Lambda(M_*  u_i,6L_i), \qquad \kC_1:=\bigcup_{j=1}^{l} \Lambda(M_*v_j,M_* ).
}
Since $\gamma$ is a nearest-neighbor path, the set $\kC_1$ is connected. Moreover, $\Lambda(M_*  u_i,6L_i)$ is connected and has non-empty intersection with $\kC_1$ for all $i=1,\ldots, \ell$. Hence, $\kC$ is connected. \\

Note that $0 \in \Lambda(M_*)= \Lambda(M_*  v_1,M_* ) \subset \kC$ (recall that $v_1=0$) and $ N\eb_1 \in \Lambda(M_* w_N,M_* )= \Lambda(M_*  v_l,M_* ) \subset  \kC$.  Hence,  using Lemma \ref{lem: connect-time},
\aln{ \label{ttlc}
	{\rm T}_{\Lambda(24N)}(0,N \eb_1) &\leq \tilde{\rm T}_{\Lambda(24N)} (\kC)\notag\\
    &\leq \sum_{\substack{1 \leq j \leq l \\ j \not \in J_*}} \tilde{\rm T}_{\Lambda(24N)}(\Lambda(M_*  v_j,M_*+1)) + \sum_{i=1}^{\ell} \tilde{\rm T}_{\Lambda(24N)}(\Lambda(M_*  u_i,6L_i+1)).
}
We remark that if $j\notin {J_*}$, then either $\Lambda(M_*  v_j,M_* )$ is an open box or $\kL_j<M_*$. In the first case,  $ \tilde{\rm T}(\Lambda(M_*  v_j,M_*+1)) \leq \tilde{\rm T}(\Lambda(M_*  v_j,2M_* )) \leq C_*M_*$. In the second case, by the definition of $\kL_j$, we have $\tilde{\rm T}(\Lambda(M_*  v_j,M_*+1)) < (M_*+1)^2$. Hence, in all cases, $\tilde{\rm T}(\Lambda(M_*  v_j,M_*+1)) \leq C_*M_*^2$. This, together with the fact that $M_*v_j \in \Lambda(2N)$ implies 
\ben{ \label{lmvj}
\tilde{\rm T}_{\Lambda(24N)}(\Lambda(M_*  v_j,M_*+1)) \leq C_*M_*^2.
}
Moreover, for $i=1, \ldots,\ell$, since $6L_i+1 > L_i=\kL_{M_*}( M_*  u_i)$ 
\be{
	\tilde{\rm T}_{\Lambda( M_*  u_i,18L_i+3)}(\Lambda( M_*  u_i,6L_i+1)) < (6L_i+1)^2 \leq 50 L_i^2.
}
Moreover,  for all $i=1, \ldots, \ell$, $M_*u_i \in \Lambda(2N)$ and   $L_i \leq N$  on $\kE_2$. Therefore, 
\ben{ \label{lmui}
\tilde{{\rm T} }_{\Lambda(24N)}(\Lambda( M_*  u_i,6L_i+1)) \leq    50 L_i^2.
}
Combining \eqref{ttlc}, \eqref{lmvj} and \eqref{lmui}, we have 
\ben{ \label{tlar}
	{\rm T}_{\Lambda(24N)}(0,N\eb_1)\leq C_*M_*^2l +  50 \sum_{i=1}^{\ell} L_i^2 \leq \frac{t}{2} + 50 \sum_{i=1}^{\ell} L_i^2,
}
since $l \leq N/(cM_*) \leq t/(2C_*M_*^2)$, as $t\geq KN \geq 2C_*M_*N/c$. Thus (b) is verified. 
\end{proof}

\begin{proof}[Proof of Lemma~\ref{lem: connect-time}]
	 We start by a simple observation that if $Y_i\subset\Z^d$ is connected and  $Y=\cup_{i=1}^k Y_i$ is connected as well, then we can find an index $1 \leq j \leq k$ such that $\cup_{i \neq j} Y_i$ is still connected.  To prove this, let $\pi(1) =1$. Since $Y=Y_{\pi(1)} \cup \bigcup_{i \neq \pi(1)} Y_i$ is connected, there exists $\pi(2) \neq \pi(1)$ such that  $Y_{\pi(1)} \cup Y_{\pi(2)}$ is connected. Continuing this procedure, we can find a permutation $\pi$ of $\{1, \ldots, k\}$ such that $\cup_{j\leq i} Y_{\pi(j)}$ is connected for all $i \leq k$. In particular, $\cup_{i \neq \pi(k)} Y_i$ is connected. \\
	
	Now we prove Lemma \ref{lem: connect-time} by induction in $k$. If $k=1$ the lemma is trivial. Suppose that we have already proved the result  for $k=m$; we need to  show it holds for $m+1$. 	Thanks to the above observation, there exists $j$ such that $\bar{X}:=\cup_{i \neq j} \Lambda(x_i,d_i)$ is connected. For any pair of vertices $x, y \in X$, there are two possibilities. First,  if both $x$ and $y$ are in $\bar{X}$ or both $x$ and $y$ are in $\Lambda(x_j,d_j)$, then the induction hypothesis  implies that
	\be{
		{\rm T}_{\Lambda}(x,y) \leq \max \{\tilde{\rm T}_{\Lambda}(\bar{X}), \tilde{\rm T}_{\Lambda}(\Lambda(x_j,d_j+1))\}\leq \sum_{i=1}^{m+1} \tilde{\rm T}_{\Lambda}(\Lambda(x_i,d_i+1)). 
	}
		Second, suppose that  one vertex, say $x$, is in $\bar{X}$ and the other, say $y$, is in $\Lambda(x_j,d_j)$. Since $X = \bar{X} \cup \Lambda(x_j,d_j)$ is connected, there exists $z \in \bar{X}$ such that $z$ is a nearest neighbor of $\Lambda(x_j,d_j)$ and thus $z \in \Lambda(x_j,d_j+1)$. Using the 
	induction hypothesis  for $\bar{X}$, which is the union of $m$ connected sets, we get 
	\be{
		\tilde{\rm T}_{\Lambda}(\bar{X}) \leq   \sum_{i \neq j } \tilde{\rm T}_{\Lambda}(\Lambda(x_i,d_i+1)).
	}
	Combining the above inequality, the triangle inequality and the fact that both $y$ and $z$ are in $\Lambda(x_j,d_j+1)$, we have 
	\be{
		{\rm T}_{\Lambda}(x,y) \leq {\rm T}_{\Lambda}(x,z) + {\rm T}_{\Lambda}(y,z) \leq \tilde{\rm T}_{\Lambda}(\bar{X}) +  \tilde{\rm T}_{\Lambda}(\Lambda(x_j,d_j+1)) \leq   \sum_{i =1 }^{m+1} \tilde{\rm T}_{\Lambda}(\Lambda(x_i,d_i+1)),
	}
which proves   the induction step and completes the proof of Lemma \ref{lem: connect-time}. 
\end{proof}

\appendix
\section{Proof of Lemma \ref{pma}}

We recall the claim of Lemma~\ref{pma}: 
Assume $d\geq 2$. There exists a positive constant $c$ such that the following results hold for any non-empty finite set $A\subset \Z^d$ with $0\in A$:
   \begin{itemize}
      \item [(i)] For all $m\in \N$ with $m\geq c^{-1}|A|^{2/d}$, 
             \be{
q(m,A):= \pp(|\kR_m \setminus A| < 8) \leq \exp(-cm|A|^{-2/d}). 
            }
         \item[(ii)] For all $m \in \N$ with $m\geq c^{-1} |A|$, 
           \begin{align*}
        p(m,A)  :=  & \pp(|\kR_m \cap (\Lambda(\sqrt{m}) \setminus A)|<8) \\
        \leq 
           &\begin{cases}
           \exp\left(-c\frac{m}{|A|^{2/d}}\right) &\textrm{ if } |A| \leq \sqrt{m}/4,\\
           \exp\left(-c\left(\frac{m^{d/2}}{|A|}\right)^{\frac{1}{d-1}}\right) &\textrm{ if }  \sqrt{m}/4 < |A|.
           \end{cases}
           \end{align*}
   \end{itemize}

\begin{proof}[\bf Proof of (i)]
Fix a non-empty finite set $A\subset \mathbb Z^d$ with $0\in A$ and recall that $\Lambda(A,8)$ denotes the set of all sites whose $\ell^\infty$-distance from $A$ is less than or equal to $8$ (see Subsection~\ref{sect:notation}). Observe that
\[
|\Lambda(A,8)|\le 20^d\,|A|, \qquad 
\{|\kR_m\setminus A|<8\}\subset \{\kR_m\subset \Lambda(A,8)\}.
\]
Hence, it suffices to show that there exists a constant \(c>0\) such that   for any connected set \(B\) and
$ 
m\ge c^{-1} |B|^{2/d}, 
$ 
\ben{ \label{rma8}
\mathbb{P}(\kR_m\subset B)\le \exp\Bigl(-c\, m\,|B|^{-2/d}\Bigr).
}
Since the proof of \eqref{rma8} in $d\geq 3$ is given in \cite[Lemma 1.2]{AC}, we only treat the case where $d=2$.
Our task is to prove that there exists a constant \(c>0\) such that for any connected finite set \(B\) containing \(0\),
\begin{equation} \label{rbsb}
\mathbb{P}(\kR_{|B|}\subset B)\le 1-c.
\end{equation}
Indeed, once \eqref{rbsb} is established, we can obtain \eqref{rma8} by dividing the time interval $\{0,1,\dots,m\}$ into \(\lfloor m/|B|\rfloor\) consecutive blocks of length \(|B|\) and using the Markov property at the endpoints of these blocks which gives the desired bound $(1-c)^{\lfloor m/|B|\rfloor}$ since $d=2$. 
We now prove \eqref{rbsb}. Let $k=|B|$. Applying Lemma~\ref{kr}~(ii) with $d=2$, $n=k$ and $a=1/2$, we obtain  there exists $c>0$ such that 
for every \(D\subset \Lambda(2\sqrt{k})\) with \(|D|\ge k\),
\[
\mathbb{P}\Bigl(|\kR_k\cap D|\ge \frac{ck}{\log k}\Bigr)\ge c.
\]
By choosing
$ 
D=\Lambda(2\sqrt{k})\cap B^c, 
$ (recall that $k=|B|$), 
we deduce that with probability at least \(c\) the walk visits some point in \(B^c\). Hence,
\[
\mathbb{P}(\kR_{|B|}\subset B)=1-\mathbb{P}\Bigl(|\kR_{|B|}\cap B^c|>0\Bigr)\le 1-c,
\]
which implies \eqref{rbsb}.
\end{proof}

\begin{proof}[\bf Proof of (ii)]
We consider two regimes: $|A| \leq \sqrt{m}/4$ and $\sqrt{m}/4 \leq |A| \leq cm$.

\noindent\textbf{Case 1: \(|A|\le \sqrt{m}/4\).}   Note that  if
$ 
\kR_m\cap \Lambda(\sqrt{m}/2)^c\neq\varnothing, 
$ 
then 
\[
|\kR_m\cap \Lambda(\sqrt{m})|\ge \frac{\sqrt{m}}{2}.
\]
Since \(|A|\le \sqrt{m}/4\), it follows that
\[
|\kR_m\cap (\Lambda(\sqrt{m})\setminus A)|\ge |\kR_m\cap (\Lambda(\sqrt{m})|-|A| \ge \frac{\sqrt{m}}{2}-\frac{\sqrt{m}}{4}=\frac{\sqrt{m}}{4}\ge 8,
\]
provided that \(m\) is sufficiently large. Consequently,
\[
\mathbb{P}\Bigl(|\kR_m\cap (\Lambda(\sqrt{m})\setminus A)|<8\Bigr)
=\mathbb{P}\Bigl(|\kR_m\cap (\Lambda(\sqrt{m})\setminus A)|<8,\;\kR_m\subset \Lambda(\sqrt{m}/2)\Bigr).
\]
Hence, noting that \(\kR_m\setminus A=\kR_m\cap (\Lambda(\sqrt{m})\setminus A)\) on the event \(\{\kR_m\subset \Lambda(\sqrt{m}/2)\}\), one can apply part~(i) to obtain
\begin{align}\label{rma2}
\begin{split}
	&\mathbb{P}\bigl(|\kR_m\cap (\Lambda(\sqrt{m})\setminus A)|<8\bigr)\\
	&=\mathbb{P}\bigl(|\kR_m\cap (\Lambda(\sqrt{m})\setminus A)|<8, \;\kR_m\subset \Lambda(\sqrt{m}/2)\bigr)\\
	&\leq \mathbb{P}\bigl(|\kR_m\setminus A|<8\bigr)\le \exp\bigl(-c\, m\,|A|^{-2/d}\bigr).
\end{split}
\end{align}
Therefore, part~(ii) holds in the case where $|A| \leq \sqrt{m}/4$.

\vspace{0.5em}

\noindent\textbf{Case 2: \(\sqrt{m}/4\le |A|\le c m\).}  
In this regime a multiscale (or chaining) argument is employed.
We first show that there exists a constant \(c_1\in (0,1)\) such that for any \(K\in \N\) and any set \(B \subset \Z^d\) satisfying
$|B|\le c_1\, K^d $,
one has
\begin{equation} \label{tauka}
\mathbb{P}\bigl(|\kR_{\sigma_{\Lambda(K)}}\cap (\Lambda(K)\setminus B)|<8\bigr)\le 1-c_1,
\end{equation}
where \(\sigma_{\Lambda(K)}\) denotes the exit time from the box \(\Lambda(K)\).
By \cite[Proposition~2.1.2~(b)]{LawLim10_book}, there exists a constant \(c_2\in (0,1)\) such that
\[
\mathbb{P}\bigl(\sigma_{\Lambda(K)}> c_2K^2\bigr)\ge \frac12.
\]
Thus,
\begin{align}\label{rtauk}
\begin{split}
	&\mathbb{P}\bigl(|\kR_{\sigma_{\Lambda(K)}}\cap (\Lambda(K)\setminus B)|\ge 8\bigr)\\
	&\geq \mathbb{P}\bigl(\sigma_{\Lambda(K)}>  c_2K^2\bigr)-  \mathbb{P}\bigl(\sigma_{\Lambda(K)}> c_2K^2,\;|\kR_{\sigma_{\Lambda(K)}}\cap (\Lambda(K)\setminus B)|< 8\bigr)\\
	&\geq \frac12 - \mathbb{P}\bigl(\sigma_{\Lambda(K)}> c_2K^2,\;|\kR_{\sigma_{\Lambda(K)}}\cap (\Lambda(K)\setminus B)|< 8\bigr).
\end{split}
\end{align}
Since $\kR_{c_2K^2}\cap (\Lambda(K)\setminus B)=\kR_{c_2K^2}\setminus B$ holds on the event \(\{\sigma_{\Lambda(K)}> c_2K^2\}\), an application of (i) (with \(m=c_2K^2\)) shows that if \(|B|\le (c_2c/2)^{d/2}K^d\), where $c$ is the constant from  part~(i), then
\begin{align*}
	&\mathbb{P}\Bigl(\sigma_{\Lambda(K)}\ge c_2K^2,\;|\kR_{\sigma_{\Lambda(K)}}\cap (\Lambda(K)\setminus B)|< 8\Bigr)\\
	&\leq \mathbb{P}\bigl(|\kR_{c_2K^2}\setminus B|<8\bigr)
	\le \exp\bigl(-c_2cK^2\,|B|^{-2/d}\bigr) \leq e^{-2}.
\end{align*}
Therefore, by \eqref{rtauk},  
\[
\mathbb{P}\bigl(|\kR_{\sigma_{\Lambda(K)}}\cap (\Lambda(K)\setminus B)|\ge 8\bigr)\ge  \frac12 -e^{-2} > \frac14,
\]
which implies \eqref{tauka} with $c_1= \min \{(cc_2/2)^{d/2}, 1/4\}$.

Given $A \subset \Z^d$ and $m \in \N$, define
\[
K:=\Bigl\lceil\Bigl(\frac{12|A|}{c_1\sqrt{m}}\Bigr)^{\frac{1}{d-1}}\Bigr\rceil,\qquad L:=\Bigl\lfloor\frac{\sqrt{m}+2K}{8K+4}\Bigr\rfloor,
\]
with $c_1$ as in \eqref{tauka}. For \(k\in\{0,1,\dots,2L-1\}\), set
\begin{align*}
	&\ell_k:=\inf\bigl\{t\in\N:\; \|S_t\|_\infty=k(2K+1)\bigr\},\\
	&\sigma_k:=\inf\bigl\{t\in\N:\;S_{\ell_k+t}\not\in \Lambda(S_{\ell_k},K)\bigr\}.
\end{align*}
By definition, the boxes $\Lambda(S_{\ell_k},K)$, $k\in\{0,1,\dots,2L-1\}$, are disjoint, and $\ell_k+ \sigma_k < \ell_{k+1}< \ell_{k+1} + \sigma_{k+1}$ holds for each $k\in\{0,1,\dots,2L-2\}$.
We also introduce the index set
\[
J=\bigl\{k \in \{0,1,\dots,2L-1\}:|\Lambda(S_{\ell_k},K)\cap A|\le c_1\, K^d\bigr\}.
\]
We will verify that
\begin{equation} \label{eltl}
\ell_{2L-1}+\sigma_{2L-1}\le \sigma_{\Lambda(\sqrt{m}/2)}\quad \text{and} \quad |J|\ge L.
\end{equation}
The first inequality follows from the choice of \(L\). Indeed,
\[
2L(2K+1)\le \frac{\sqrt{m}+2K}{2},
\]
which implies that the $\ell^\infty$-distance to the boundary of $\Lambda(S_{\ell_{2L-1}},K)$ is at most
\[
(2L-1)(2K+1)+K=2L(2K+1)-(K+1)<\frac{\sqrt{m}}{2}.
\]
Hence, the walk $S_\cdot$ must stay in $\Lambda(\sqrt{m}/2)$ until time $\ell_{2L-1}+\sigma_{2L-1}$, and the first inequality follows.
For the second inequality, note that the boxes $\Lambda(S_{\ell_k},K)$, $k \in \{0,1,\dots,2L-1\}$, are disjoint and $|\Lambda(S_{\ell_k},K)\cap A|> c_1\, K^d$ if $k \in J^c$.
Thus,
\[
|A| \geq \sum_{k \in J^c} |\Lambda(S_{\ell_k},K)\cap A| \geq c_1|J^c| K^d.
\]
This, combined with the choice of $K$ and $L$, yields that
\[
|J^c| \leq \frac{|A|}{c_1K^d} \leq \frac{\sqrt{m}}{12K}\le \frac{\sqrt{m}+2K}{8K+4}<L+1,
\]
which proves $|J^c|\le L$.
Therefore, $|J| = 2L-|J^c| \geq L$ holds, and we have \eqref{eltl}.

We now decompose
\begin{align*}
	\mathbb{P}\bigl(|\kR_m\cap (\Lambda(\sqrt{m})\setminus A)|<8\bigr)
	&=\mathbb{P}\Bigl(|\kR_m\cap (\Lambda(\sqrt{m})\setminus A)|<8,\;\sigma_{\Lambda(\sqrt{m}/2)}\le m\Bigr)\\
	&\qquad +\mathbb{P}\Bigl(|\kR_m\cap (\Lambda(\sqrt{m})\setminus A)|<8,\;\sigma_{\Lambda(\sqrt{m}/2)}> m\Bigr)\\
	&=:X_1+X_2.
\end{align*}
The second term \(X_2\) is controlled as in \eqref{rma2}:
\[
X_2\le \mathbb{P}\bigl(|\kR_m\setminus A|<8\bigr) \le \exp\bigl(-c_3\, m\,|A|^{-2/d}\bigr),
\]
with $c_3$ a positive constant.
Regarding the first term \(X_1\), let \(j_1<\cdots<j_L\) be the first \(L\) elements of \(J\), which exist because \(|J|\ge L\). By \eqref{eltl},  each time interval \([\ell_{j_r},\ell_{j_r}+\sigma_{j_r}]\), \(r\in [L]\) is contained in  \([0,\sigma_{\Lambda(\sqrt{m}/2)}]\). Therefore, 
\ben{ \label{x1sj8}
X_1 \leq \bigl(|\kR_{[\ell_{j_r},\ell_{j_r}+\sigma_{j_r}]}\cap (\Lambda(S_{\ell_{j_r}},K)\setminus A)|<8 \, \, \forall \, r \in [L]\bigr).
}
These intervals are disjoint since \(\ell_k+\sigma_k<\ell_{k+1}\) for all $k$. 
Moreover, \(j_r\in J\) implies \(|\Lambda(S_{\ell_{j_r}},K)\cap A|\le c_1K^d\). Hence,  \eqref{tauka} provides $$\mathbb{P}\bigl(|\kR_{[\ell_{j_r},\ell_{j_r}+\sigma_{j_r}]}\cap (\Lambda(S_{\ell_{j_r}},K)\setminus A)|<8|~\kF_{\ell_{j_r}}\bigr)\le 1-c_1,$$
where $\kR_I:=\{S_i\mid~i\in I\}$ and $\kF_\ell:=\sigma(S_i\mid~i\leq \ell)$. By \eqref{x1sj8}, repeated use of the strong Markov property at the times \(\ell_{j_1},\dots,\ell_{j_L}\)   yields
\[
X_1\le \mathbb{P}\bigl(|\kR_{[\ell_{j_r},\ell_{j_r}+\sigma_{j_r}]}\cap (\Lambda(S_{\ell_{j_r}},K)\setminus A)|<8 \, \, \forall \, r \in [L]\bigr) \le (1-c_1)^L.
\]
Since \(|A|\ge \sqrt{m}/4\), 
\[
K\le 2\Bigl(\frac{12|A|}{c_1\sqrt{m}}\Bigr)^{\frac{1}{d-1}}.
\]
By decreasing the constant \(c\) in the statement if necessary, we may assume that
\[
\frac{\sqrt{m}+2K}{8K+4}\ge 2.
\]
Hence,
\[
L\ge \frac{\sqrt{m}+2K}{16K+8}\ge \frac{\sqrt{m}}{24K}\ge c_4 m^{\frac{d}{2(d-1)}}\,|A|^{-1/(d-1)}
\]
with $c_4$ a positive constant. Therefore, the above inequality implies
\[
X_1\le \exp\Bigl(-c_5\,\Bigl(\frac{m^{d/2}}{|A|}\Bigr)^{1/(d-1)}\Bigr),
\]
with $c_5:=-c_4\log(1-c_1)$. 
Combining the estimates for \(X_1\) and \(X_2\) completes the proof of part~(ii).
\end{proof}

\begin{rem}[Sharpness of Lemma \ref{pma}]\label{remark: SRW range remark}
For each regime, one can choose \(A\subset\mathbb Z^d\) so that the corresponding upper bound for \(p(m,A)\) or \(q(m,A)\) is sharp up to the constant in the exponent.

 $\bullet$  Part (i) and Part (ii) with $|A|\le \sqrt{m}/4$. Let  $A:=\Lambda(L)$ with $L \leq m^{1/(2d)}$. Then, since $A \subset \Lambda(\sqrt{m})$,
\ben{ \label{rmas8}
\{ |\kR_m \cap (\Lambda(\sqrt{m}) \setminus A)|<8\}  \cap \{|\kR_m\setminus A|<8\}\supset\{\kR_m \subset A\}.
}
By Donsker's invariance principle (Skorokhod coupling), for any $x \in \Z^d$ there is a coupling of a simple random walk $S^x_n$ starting at $x$ and a Brownian motion $B$ starting at $0$ such that, uniformly,  $\sup_{t\in[0,1]}\big|(S^x_{\lfloor tL^2\rfloor}-x)/L-  B_t\big|\xrightarrow{\mathbb P}0$, and hence there exists $c_0>0$ such that
$$
\inf_{x\in\Lambda(L/2)}\mathbb{P}(\kR^x_{L^2} \subset \Lambda(L), \ S^x_{L^2} \in\Lambda(L/2))\ge c_0.
$$
Using this  inequality and the strong Markov property at times $L^2,2L^2,\dots$,  
$$
\pp(\kR_m \subset \Lambda(L))  \ge\ c_0^{\lceil m/L^2 \rceil}.
$$
Together with \eqref{rmas8}, this shows that the bounds of Part(i) and Part(ii) for $|A|\le \sqrt{m}/4$ are sharp. 

 $\bullet$ Part (ii) with $\sqrt{m}/4 < |A| \leq cm$. Let $L\in 2\N$ and consider  
\[
A=\Big\{x=(x_1,\ldots,x_d)\in\mathbb{Z}^d:\ |x_1|\le \sqrt{m},\ |x_i|\le L \, \, \forall  \, \, 2 \leq  i \leq d\Big\}.
\]
Note that  $L\asymp \Big(\frac{|A|}{\sqrt{m}}\Big)^{1/(d-1)} \leq (c\sqrt{m})^{1/(d-1)}$. Define 
 $$x_j=\frac{jL}{2} \eb_1; \quad j=0, \ldots,K, \qquad K= \lceil 6\sqrt{m}/L \rceil.$$
Then we have
\be{
\bigcup_{j=0}^{K-1} \Lambda(x_j, L) \cap (\Lambda(\sqrt{m}) \setminus A) = \varnothing; \quad \Lambda(x_K, \sqrt{m}) \cap \Lambda(\sqrt{m}) = \varnothing.
}
Therefore, 
\aln{ \label{rmej}
&\pp(|\kR_m \cap (\Lambda(\sqrt{m}) \setminus A)|<8) \notag   \geq \pp(\kR_m \cap   (\Lambda(\sqrt{m}) \setminus A) = \varnothing) \notag  \\
& \geq \pp \left( \kR_m \subset  \bigcup_{j=0}^{K-1} \Lambda(x_j, L) \cup \Lambda(x_K, \sqrt{m}) \right)   \geq \pp \left( \bigcap_{j=0}^{K-1} E_j \cap \kE \right),
}
where for $j=0, \ldots,K-1$
\[
E_j= \{\kR_{[jL^2,(j+1)L^2]} \subset \Lambda(x_j,L), \, S_{(j+1)L^2} \in \Lambda(x_j,L/2) \}
\]
and 
\[
\kE= \{\kR_{[KL^2,m]} \subset \Lambda(x_K,\sqrt{m}) \}.
\]
By Donsker's invariance principle again and the strong Markov property, there exists a universal constant $c_1\in(0,1)$ such that for $j=1, \ldots,K-1$
\al{
\min\{ \pp(E_0),\, \pp(E_j \mid \cap_{i=0}^{j-1} E_{i})\} \geq \inf_{x \in \Lambda(L/2)} \pp(\kR^x_{[0,L^2]} \subset \Lambda(L), S_{L^2} \in \Lambda(L/2)) \geq c_1,
}
and 
\al{
 \pp(\kE \mid \cap_{i=0}^{K-1} E_{i}) \geq \inf_{x \in \Lambda(x_K,\sqrt{m}/2)} \pp(\kR^x_{[0,m]} \subset \Lambda(x_K,\sqrt{m})) \geq c_1,
}
here we used that  on $E_{K-1}$, $S_{KL^2} \in \Lambda(x_{K-1},L/2) \subset \Lambda(x_K,\sqrt{m}/2)$.  Combining the last two display equations with \eqref{rmej} we obtain
\be{
\pp(|\kR_m \cap (\Lambda(\sqrt{m}) \setminus A)|<8) \geq c_1^{K+1}.
}
 Notice also that 
\[
K\asymp \sqrt{m}/L\asymp  \left(\frac{m^{d/2}}{|A|}\right)^{1/(d-1)} .
\]
The last two estimates show that the bound of Part (ii) is sharp. 
\end{rem}

\section{Proof of Lemma~\ref{lem:mdist}}
Since $\Z^2 \subset \Z^d$ for $d \geq 2$, it suffices to prove Lemma~\ref{lem:mdist} in the case where $d=2$.
Hence, we assume $d=2$ throughout this section.
Let us first prepare some notation and a lemma. 
For any $x \in \Z^2$ and $\ell \in \N$, consider the event
\begin{align*}
	S(x,\ell):=H(x,\ell) \cap V(x,\ell),
\end{align*}
where $c>0$ is a small constant to be chosen in Lemma~\ref{lem:C_new} and
\begin{align*}
	&H(x,\ell):=\Bigl\{ \text{$\exists$ an open left-right crossing path $\gamma$ of $x+[-\ell,\ell] \times [0,\ell]$ with $|\gamma| \leq \ell/c$} \Bigr\}\\
	&V(x,\ell):=\Bigl\{ \text{$\exists$ an open top-bottom crossing path $\gamma$ of $x+[0,\ell] \times [-\ell,\ell]$ with $|\gamma| \leq \ell/c$} \Bigr\}.
\end{align*}
Furthermore, recall that $\pp_p$ is the probability measure of the standard site percolation on $\Z^2$ with density $p$ and $p_c(2)$ is its critical percolation probability.

\begin{lem}\label{lem:C_new}
There exists $p' \in (p_c(2),1)$ such that if $p \in (p',1]$, then for some constant $c=c(p)>0$ and for all $\ell$ sufficiently large,
\begin{align*}
	\pp_p(S(0,\ell)^c) \leq \exp(-c\ell).
\end{align*}
\end{lem}
\begin{proof}
By translation invariance and symmetry, it suffices to prove the lemma with $S(0,\ell)$ replaced by the event
\begin{align*}
	H'(\ell):=\Bigl\{ \text{$\exists$ an open left-right crossing path $\gamma$ of $[0,2\ell] \times [-\ell/2,\ell/2]$ with $|\gamma| \leq \ell/c$} \Bigr\}.
\end{align*}
To this end, fix $\delta \in (0,1/48)$ and set $x_i:=3(i-1)\lceil \delta\ell \rceil \eb_1$ for $i \in \{ 0,1,\dots,L \}$, where $L=\lceil 1/\delta \rceil+1$.
Since $p>p_c(2)$, there exists almost surely a unique infinite open cluster, say $\mathcal{C}_\infty$ (see for instance \cite[Theorems~{1.10} and {8.1}]{Gri99_book}).
Moreover, for any $x,y \in \Z^2$, define the chemical distance $D(x,y)$ between $x$ and $y$ as the length of the shortest open path between these two sites. It is known from \cite[Theorem~{1.4} and inequality~(4)]{GarMar07} (see also \cite[Theorem~3]{DurSch88}) that there exists $p' \in (p_c(2),1)$ such that the following results hold whenever $p \in (p',1]$:
there exists a constant $c'=c'(p,\delta)>0$ such that
\begin{itemize}
\item if $\ell$ is large enough, then for all $|y|_\infty \geq \delta\ell$,
\begin{align*}
	\pp_p((1+\delta)|y|_\infty \leq D(0,y)<\infty) \leq \exp(-c'|y|_\infty),
\end{align*}

\item for all $r>0$ sufficiently large,
\begin{align*}
	\pp_p(\mathcal{C}_\infty \cap \Lambda(r)=\emptyset) \leq \exp(-c'r).
\end{align*}
\end{itemize}
These results tell us that there exists a constant $c''=c''(p,\delta)>0$ such that for all $\ell$ sufficiently large,
\begin{align}\label{eq:SS}
	\pp_p((H'_1(\ell) \cap H'_2(\ell))^c) \leq \exp(-c''\ell),
\end{align}
where
\begin{align*}
	&H'_1(\ell):=\left\{
	\begin{minipage}{27em}
		$D(x,y)<(1+\delta)|x-y|_\infty$ for all $x,y \in [-5\ell,5\ell] \times [-\ell/2,\ell/2]$\\
		with $D(x,y)<\infty$ and $|x-y|_\infty \geq \delta\ell$ 
	\end{minipage}
	\right\},\\
	&H'_2(\ell):=\Bigl\{ \mathcal{C}_\infty \cap \Lambda(x_i,\delta\ell) \not= \emptyset \text{ for all } i \in \{ 0,1,\dots,L \} \Bigr\}.
\end{align*}
Note that on the event $H'_1(\ell) \cap H'_2(\ell)$, we can find a site $z_i \in \mathcal{C}_\infty \cap \Lambda(x_i,\delta\ell)$ for every $i \in \{ 0,1,\dots, L \}$.
In addition, since $\delta\ell \leq |z_i-z_{i+1}|_\infty \leq 6\delta\ell$ holds for each $i \in \{ 0,1,\dots,L-1 \}$, we have for all $i \in \{ 0,1,\dots,L-1 \}$,
\begin{align*}
	D(z_i,z_{i+1})<(1+\delta)|z_i-z_{i+1}|_\infty \leq 6(1+\delta)\delta\ell<\frac{\ell}{4}.
\end{align*}
This implies that the shortest open path $\tilde{\gamma}$ connecting all $z_i$'s is included in the slab $\R \times [-\ell/2,\ell/2]$.
Note that all the boxes $\Lambda(x_i,\delta\ell)$ belong to the slab $\R \times [-\ell/2,\ell/2]$ but some of them are located outside of $[0,2\ell] \times [-\ell/2,\ell/2]$ (recall here that $x_0=-3\lceil \delta \ell \rceil \eb_1$ and $x_L=3\lceil 1/\delta \rceil \lceil \delta \ell \rceil \eb_1$). Hence, removing segments outside of $[0,2\ell] \times [-\ell/2,\ell/2]$ from $\tilde{\gamma}$, we can construct an open left-right crossing path $\gamma$ of $[0,2\ell] \times [-\ell/2,\ell/2]$.
In particular, the construction of $\gamma$ and $\delta \in (0,1/48)$ give
\begin{align*}
	|\gamma| \leq |\tilde{\gamma}| \leq \frac{\ell}{4} \times L \leq \frac{\ell}{\delta},
\end{align*}
which means that $H'_1(\ell) \cap H'_2(\ell)$ implies $H'(\ell)$. Therefore, by choosing $c=\min\{ \delta,c'' \}$, \eqref{eq:SS} implies that
\begin{align*}
	\pp_p(H'(\ell)^c) \leq \pp_p((H'_1(\ell) \cap H'_2(\ell))^c) \leq \exp(-c\ell),
\end{align*}
and this is our desired conclusion.
\end{proof}

We are now in a position to prove Lemma~\ref{lem:mdist}.

\begin{proof}[\underline{Proof of Lemma~\ref{lem:mdist}}]
Choose $\e:=1-p'$ with $p'$ as in Lemma \ref{lem:C_new}  and fix $p>1-\varepsilon$.
Moreover, for a given $m>0$ sufficiently large and  $m^2 \geq r \geq m$, we  set
\begin{align*}
	\ell:=\lfloor (\log m)^2 \rfloor,\qquad
	h:=\Bigl\lceil \frac{m}{\ell} \Bigr\rceil,\qquad
	k:= \Bigl\lfloor \frac{\sqrt{r}}{4\ell} \Bigr\rfloor.
\end{align*}
Then, for each $i=0,1,\dots,k$, we consider the event
\begin{align*}
	\kE_i:=\left\{
	\begin{minipage}{26em}
		$\exists$ a left-right crossing path $\eta$ of $[0, h\ell] \times [(3i-1)\ell,(3i+1)\ell]$\\
		such that $|\eta| \leq Cm$ and $|{\rm Cl}(\eta)| \leq \sqrt{r}\ell/2$
	\end{minipage}
	\right\},
\end{align*}
where $C \geq 3$ is a large constant to be chosen later.
It is clear that if the event $\bigcup_{i=1}^k \kE_i$ occurs, then there exists a $\Z^2$-path $\gamma \subset \Lambda(2m)$ from $0$ to $m \eb_1$ such that $|\gamma| \leq 2Cm$ and $|{\rm Cl}(\gamma)| \leq \sqrt{r} \ell$.
Indeed, on the event $\bigcup_{i=1}^k \kE_i$, there exists a left-right crossing path $\eta$ of $[0,h\ell] \times [0, \lfloor \sqrt{r} \rfloor]$ such that $|{\rm Cl}(\eta)| \leq \sqrt{r}\ell/2$ and $|\eta| \leq Cm$.
Letting $x$ and $y$ be the starting and ending points of $\eta$, we have $|x|_1 \leq \sqrt{r} \leq \min\{ \sqrt{r}\ell/4,m \}$ and $|y-m\eb_1|_1 \leq \sqrt{r} + \ell \leq \min\{ \sqrt{r}\ell/4,2m \}$, since $m$ is large enough.
Hence, the desired $\Z^2$-path $\gamma$ can be constructed by connecting $0$ (resp.~$y$) to $x$ (resp.~$m\eb_1$) with a shortest $\Z^2$-path.
It therefore remains to prove that there exists a positive constant $C$ such that for all $i=0,\ldots,k$,
\begin{align}\label{peir}
	\pp_p(\kE_i) \geq 1-\exp(-\ell \sqrt{r}/C).
\end{align}
Indeed,  since $\kE_1,\dots,\kE_k$ are independent,
\begin{align*}
	&\pp_p\Bigl( \text{$\exists$a $\Z^2$-path $\gamma \subset \Lambda(2m)$ from $0$ to $m \eb_1$ with $|{\rm Cl}(\gamma)| \leq \sqrt{r} \ell$ and  $|\gamma| \leq 2Cm$} \Bigr)\\
	&\geq \pp_p\left( \bigcup_{i=1}^k \kE_i \right) =  1- \prod_{i=1}^k\pp_p(\kE_i^c) \geq 1-\exp(-r/8C),
\end{align*}
and the lemma follows. Our task is now to prove \eqref{peir}.
By translation invariance, it suffices to check \eqref{peir} for $i=0$.
 Note that for any $x \in \Z^2$ and $j>k \geq 0$,
\begin{align}\label{sxil}
\begin{split}
	&\bigcap_{i=k}^j S(x+(i\ell,0),\ell) 
    \subset \left\{
	\begin{minipage}{23em}
		$\exists$ an open left-right crossing path $\tilde{\eta}$ of \\
		  $x+[(k-1)\ell,(j+1)\ell] \times [-\ell,\ell]$ s.t. $|\tilde{\eta}| \leq 2(j-k)\ell/c$
	\end{minipage}
	\right\},
\end{split}
\end{align}
where $c$ is the constant in the definition of $S(\cdot,\ell)$. For $i=0,1,\ldots,h-1$, let $x_i=(i\ell,0)$ and call
$\Lambda(x_i,\ell)$ \textbf{good} if $S(x_i,\ell)$ occurs, and
\textbf{bad} otherwise. We decompose the sequence of boxes into
maximal consecutive runs of good boxes, separated by bad boxes.
By \eqref{sxil}, each maximal run of consecutive good boxes can be crossed from left to right by an open path whose length is linear in the length of the run.
We then connect the crossing paths of successive good runs along the
boundaries of the intervening bad boxes. In this way, we obtain a
left--right crossing path
\(
\eta\subset[0,h\ell]\times[-\ell,\ell]
\)
such that
\begin{align}\label{gamhlc}
\begin{split}
	&|\eta| \leq \frac{2h \ell}{c} \leq \frac{4m}{c},\\
	&|{\rm Cl}(\eta)| \leq 8 \ell \times |\{i \in [0,h-1]:\text{$\Lambda(x_i,\ell)$ is bad}\}|.
\end{split}
\end{align}
It follows that
\begin{align} \label{ppe0}
	\pp_p(\kE^c_0)
	\leq \pp_p\left( \sum_{i=0}^{h-1} X_i \geq \frac{\sqrt{r}}{16} \right), \quad \textrm{where} \quad X_i = \textbf{1}\{ \text{$\Lambda(x_i,\ell)$ is bad}\}.
\end{align}
We remark that the random variables $X_i$ and $X_j$ are independent if $|i-j| \geq 3$.
By Lemma~\ref{lem:C_new}, if $m \in \N$ is large enough, then there exists a constant $c \in (0, 1)$ such that  for all $i \geq 0$,
\begin{align*}
	\E[e^{c\ell X_i/2}]  \leq 1-e^{-c\ell}+e^{-c\ell/2} \leq \exp(e^{-c \ell/2}) \leq \exp(m^{-1}).
\end{align*}
Here we used the fact that $(1+1/t)^t \leq e$ for all $t>0$ and $\ell= \lfloor (\log m)^2 \rfloor$. 
Therefore, Lemma~\ref{lem:ldlc}~(i) proves that for all $m$ sufficiently large,
\begin{align*}
	\pp_p\left( \sum_{i=0}^{h-1} X_i \geq \frac{\sqrt{r}}{16} \right)
	&\leq 3 \exp \left( -\left( \frac{c \ell}{2} \right) \frac{\sqrt{r}}{48}\right) \left( \max_{ 0\leq i \leq h-1} \E[e^{c\ell X_i/2}] \right)^{\lceil h/3 \rceil}\\
	&\leq 3 \exp \left( -\frac{c\ell\sqrt{r}}{96} \right) \times \exp\left( m^{-1} \left\lceil \frac{h}{3} \right\rceil \right)
	\leq  \exp \left( -\frac{c \ell \sqrt{r}}{100} \right),
\end{align*}
and \eqref{peir} immediately follows by using \eqref{ppe0}.
\end{proof}

\specialsection*{Declaration on the use of AI}
The results in this paper were already announced in 2023
(see \cite{CanKubNak26}). All mathematical arguments and
proofs were developed and verified by the authors.
AI tools were used only for checking the proofs.

\specialsection*{Acknowledgments}
The second author was supported by JSPS KAKENHI Grant Number JP25K07054.
The third author was supported by JSPS KAKENHI
Grant Numbers JP24K16937 and JP25K00911.

\bibliographystyle{abbrv}
\bibliography{main}
\end{document}